\documentclass[11pt,letterpaper]{amsart}
\usepackage[foot]{amsaddr}

\usepackage{mathtools}
\usepackage[T1]{fontenc}
\usepackage[latin9]{inputenc}
\usepackage{lmodern} 
\usepackage{geometry}
\usepackage{multicol}
\usepackage{graphicx}
\usepackage{soul}
\usepackage{xcolor}
\usepackage{amssymb}
\usepackage{placeins}
\usepackage{cite}
\usepackage{caption}
\usepackage{enumerate}
\usepackage{afterpage}
\usepackage{enumitem}
\usepackage{bmpsize}
\usepackage{hyperref}
\usepackage{tabu}
\usepackage{enumitem}
\numberwithin{equation}{section}
\usepackage{stmaryrd}
\usepackage{tikz}
\usetikzlibrary{matrix,graphs,arrows,positioning,calc,decorations.markings,shapes.symbols}

\def\Fe{ \mathsf{F}}

\def\Re{ \mathsf{Re}}
\def\Im{ \mathsf{Im}}
\def\i{ \mathsf{i}}
\newcommand{\eps}{\varepsilon}

\newcommand{\darkbluevertex}{\raisebox{-1pt}{\begin{tikzpicture}[scale=1.2]
		\filldraw[fill=black, draw=black] (0,0) circle(0.1);
		\end{tikzpicture}}}

\newcommand{\lightbluevertex}{\raisebox{-1pt}{\begin{tikzpicture}[scale=1.2]
		\filldraw[fill=white, draw=black, thick] (0,0) circle(0.1);
		\end{tikzpicture}}}

\newcommand{\I}{\mathbf{i}}

\makeatletter
\renewcommand{\email}[2][]{%
  \ifx\emails\@empty\relax\else{\g@addto@macro\emails{,\space}}\fi%
  \@ifnotempty{#1}{\g@addto@macro\emails{\textrm{(#1)}\space}}%
  \g@addto@macro\emails{#2}%
}
\makeatother

\newtheorem{theorem}{Theorem}[section]
\newtheorem{lemma}[theorem]{Lemma}
\newtheorem{proposition}[theorem]{Proposition}

\newtheorem{corollary}[theorem]{Corollary}
{ \theoremstyle{definition}
\newtheorem{definition}[theorem]{Definition}}
{ \theoremstyle{remark}
\newtheorem{remark}[theorem]{Remark}}

\newcommand{\N}{\mathbb{N}}
\newcommand{\Z}{\mathbb{Z}}
\newcommand{\R}{\mathbb{R}}

\renewcommand{\P}{\mathbb{P}}

\newcommand{\MY}{\mathsf{Y}}

\newcommand{\PM}{\mathbb{P}^{w}_N}
\newcommand{\ZM}{Z^{w}_N}

\newcommand{\XS}{X^s}
\newcommand{\YS}{Y^s}
\newcommand{\XE}{X^e}
\newcommand{\YE}{Y^e}
\newcommand{\LP}{L}
\newcommand{\SP}{S}

\title{Maximal free energy of the log-gamma polymer}

\author[G. Barraquand]{Guillaume Barraquand}
\address{G. Barraquand,
Laboratoire de physique de l'{\'e}cole normale sup\'erieure, ENS, Universit{\'e} PSL, CNRS, Sorbonne Universit{\'e}, Universit{\'e} de Paris, Paris, France}
\email{guillaume.barraquand@ens.fr}
\author[I. Corwin]{Ivan Corwin}
\address{I. Corwin, Department of Mathematics, Columbia University, New York, NY 10027, USA} \email{ivan.corwin@gmail.com}
\author[E. Dimitrov]{Evgeni Dimitrov}
\address{E. Dimitrov, Department of Mathematics, University of Southern California, Los Angeles, CA 90089, USA} \email{edimitro52gmail.edu}

\begin{document}

\begin{abstract}
We prove a phase transition for the law of large numbers and fluctuations of $\mathsf F_N$, the maximum  of the free energy of the  log-gamma directed polymer with parameter $\theta$, maximized over all possible starting and ending points in an $N\times N$ square. In particular, we find an explicit critical value $\theta_c=2\Psi^{-1}(0)>0$ ($\Psi$ is the digamma function)
such that:
\begin{itemize}
\item For $\theta<\theta_c$, $\mathsf F_N+2\Psi(\theta/2)N$ has order $N^{1/3}$ GUE Tracy-Widom fluctuations;
\item For $\theta=\theta_c$, $\mathsf F_N= \Theta(N^{1/3}(\log N)^{2/3})$;
\item For $\theta>\theta_c$, $\mathsf F_N=\Theta(\log N)$.
\end{itemize}
Using the same techniques for analyzing $\mathsf F_N$, we also show that an analogous phase transition occurs in a certain  free start/end-point polymer measure with inverse gamma weights. By exploiting a connection between the log-gamma polymer and a certain random operator on the honeycomb lattice, recently found by Kotowski and Vir\'ag (Comm. Math. Phys. 370, 2019), we also deduce a similar phase transition for the asymptotic behavior of the smallest positive eigenvalue of the  aforementioned random operator.
\end{abstract}

\maketitle

\tableofcontents

%

\section{Introduction and main results} \label{Section1}

%
\subsection{Maximal free energy of the log-gamma polymer}\label{Section1.1}

The log-gamma polymer is an exactly solvable model for a directed polymer on the $\Z^2$ lattice introduced in \cite{Sep12}. A variety of  probabilistic tools and exact formulas have been developed to study this model, allowing one to obtain very precise information about the asymptotic behavior of its free energy and the geometry of paths under the polymer measure. In this paper we consider $\mathsf{F}_N$,  the maximum of the point-to-point free energies of the log-gamma polymer as the starting and ending point vary inside an $N\times N$ square in the $\Z^2$ lattice. The log-gamma polymer depends on $\theta>0$,  the shape parameter of the inverse-gamma  distributed weights used to define the model. We prove that the asymptotic behavior of the maximal free energy obeys a sharp transition depending on the value of $\theta$.

In the remainder of this section we define the polymer model and state our main results about $\mathsf{F}_N$. In Section \ref{SectionS} below we provide two applications of our work, which we believe to be of separate interest. In Section \ref{SectionS.1} we define a certain free start/end-point polymer measure in the $N \times N$ square in the $\Z^2$ lattice with inverse gamma weights. We call this the {\em log-gamma polymer measure} and show that it mirrors the phase transition of $\mathsf{F}_N$; for example, we prove a sharp transition in the polymer length as one varies $\theta$. Our second application comes from a connection established in \cite{kotowski2019tracy} between $\mathsf{F}_N$ and the smallest positive eigenvalue of a certain random operator defined on a finite honeycomb lattice. Based on this connection, one can directly apply our results about $\mathsf{F}_N$ to obtain an analogous description of this smallest positive eigenvalue. We will discuss this further in Section \ref{sec:randomoperator}.\\

\begin{definition}\label{Def1} A continuous random variable $X$ is said to have the inverse-gamma distribution with parameter $\theta > 0 $ if its density is given by
\begin{equation}\label{S1invGammaDens}
f_\theta(x) = \frac{{\bf 1} \{ x > 0 \} }{\Gamma(\theta)} \cdot x^{-\theta - 1} \cdot \exp( - x^{-1}),
\end{equation}
where $\Gamma$ is the Euler gamma function.
We let $w = \left( w_{i,j} :  i, j \in \mathbb{Z} \right)$ denote the random matrix such that $w_{i,j}$ are i.i.d. with density $f_{\theta}$ as in (\ref{S1invGammaDens}). We denote by $(\Omega, \mathcal{F}, \mathbb{P})$ the probability space on which $w$ is defined. A {\em directed lattice path} is a sequence of vertices $(x_1, y_1), \dots, (x_k, y_k) \in \mathbb{Z}^2$ such that $x_1 \leq x_2 \leq \cdots \leq x_k$, $y_1 \leq y_2 \leq \cdots \leq y_k$ and $(x_i - x_{i-1}) + (y_i - y_{i-1}) = 1$ for $i = 2, \dots, k$. In words, a directed lattice path is an up-right path on $\mathbb{Z}^2$, which makes unit steps in the coordinate directions (see Figure \ref{fig:loggammapolymer}).

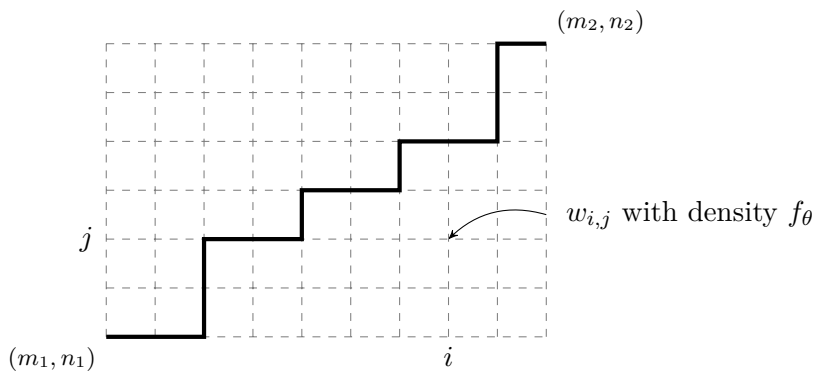
\begin{figure}[h]
 \captionsetup{width=\linewidth}
\begin{center}
\begin{tikzpicture}[scale=0.65]
\begin{scope}
\draw (0,0) node[anchor = north east]{{\footnotesize $(m_1,n_1)$}};
\draw (9,6) node[anchor=south west]{{\footnotesize $(m_2,n_2)$}};
\draw[->, >=stealth'] (9,2.5) node[anchor=west]{{ $w_{i,j}$ with density $f_{\theta}$}} to[bend right] (7,2);
\draw[dashed, gray] (0,0) grid (9,6);
\draw[ultra thick] (0,0) -- (1,0) -- (2,0) -- (2,1) -- (2,2) -- (3,2) -- (4,2) -- (4,3) -- (5,3) -- (6,3) -- (6,4) -- (7,4) -- (8,4) -- (8,6) -- (9,6);
\draw (7, -0.4) node{$i$};
\draw (-0.4,2) node{$j$};
\end{scope}
\end{tikzpicture}
\end{center}
\caption{A directed lattice path $\pi\in \Pi(m_1,n_1;m_2,n_2)$ for the log-gamma polymer. }
\label{fig:loggammapolymer}
\end{figure}

For any $m_1, n_1, m_2, n_2 \in \mathbb{Z}$ we let $\Pi(m_1,n_1;m_2,n_2)$ denote the set of directed paths $\pi$ in $\mathbb{Z}^2$ from $(m_1,n_1)$ to $(m_2,n_2)$. Observe that $\Pi(m_1,n_1;m_2,n_2)$ is non-empty if and only if $m_1 \leq m_2$ and $n_1 \leq n_2$ (which we write as $(m_1, n_1) \leq (m_2, n_2)$). Given a directed path $\pi$ we define its {\em weight}
\begin{equation}\label{PathWeight}
w(\pi) = \prod\nolimits_{(i,j) \in \pi} w_{i,j}.
\end{equation}
For $m_1, n_1, m_2, n_2 \in \mathbb{Z}$ with $ (m_1, n_1) \leq (m_2, n_2) $ we define the point-to-point {\em partition function} as
\begin{equation}\label{PartitionFunct}
Z(m_1,n_1;m_2,n_2) = \sum\nolimits_{ \pi \in\Pi(m_1,n_1;m_2,n_2)} w(\pi).
\end{equation}
For each $p, q \in \mathbb{Z}$ with $p\leq q$ we write $\llbracket p, q \rrbracket = \{p, p+1, \dots, q\}$. For each $N \in \mathbb{N}$ we define
\begin{equation}\label{Fmax}
\Fe_N = \max_{\substack{(m_1, n_1) \leq (m_2,n_2)\\  m_1, n_1, m_2, n_2 \in \llbracket 1, N \rrbracket}} \log Z(m_1,n_1;m_2,n_2),
\end{equation}
the maximal point-to-point free energy of the log-gamma polymer in a square of side length $N$.
\label{def:loggammapolymer}
\end{definition}

In this paper we are interested in understanding the asymptotic behavior of $\Fe_N$ as $N$ tends to infinity. Before we state our main result we introduce a bit of notation.

Let $\Psi(x)$ denote the digamma function, i.e.
\begin{equation}\label{digammaS1}
\Psi(z) = \frac{\Gamma'(z)}{\Gamma(z)} = - \gamma_{E} + \sum_{n = 0}^\infty \left[\frac{1}{n + 1} - \frac{1}{n+z} \right],
\end{equation}
where  $\gamma_{E}$ is the Euler constant.
From the formula
$$\Psi'(z) = \sum_{n = 0}^\infty \frac{1}{(n+z)^2},$$
we know that $\Psi$ is strictly increasing on $(0, \infty)$. Moreover, $\lim_{z \rightarrow 0} \Psi(z) = - \infty$ and $\lim_{z \rightarrow \infty} \Psi(z) = \infty$ (see \cite[(8), pp. 33]{Luke}). This means that there is a unique point $\gamma_* \in (0, \infty)$ such that $\Psi(\gamma_*) = 0$, and we let $\theta_c := 2 \gamma_*$ so that
\begin{equation}\label{DefThetaC}
\Psi(\theta/2)  < 0 \mbox{ for $\theta \in (0,\theta_c)$},\quad \Psi(\theta/2)  = 0 \mbox{ for $\theta = \theta_c$,}\quad \mbox{and }\Psi(\theta/2)  >  0\mbox{ for $\theta > \theta_c$}.
\end{equation}

Finally, for any $\theta > 0$ we define
\begin{equation}\label{DefSigma}
\sigma_{\theta} := [- \Psi''(\theta/2)]^{1/3}  = \left[2 \sum_{n = 0}^\infty \frac{1}{(n+ \theta/2)^3}  \right]^{1/3}.
\end{equation}

With the above notation we are ready to state our main results (see also Figure \ref{fig:PhaseDiagram}). It should be noted that the limit theorem in the subcritical case (Theorem \ref{PM1}) is the same one that is satisfied when $\Fe_N$ is replaced by  $\log Z(1,1;N,N)$ (i.e., the maximal separation of endpoints). That result has been previously shown in the works \cite{BCR, KQ, BCDA} and is recalled and used here as Proposition \ref{LGPCT}. We also mention that while the results below are formulated for the square domain $N \times N$, one could adapt the techniques of this paper to more general rectangular domaints $N \times M$ and obtain analogous results for those. We will not pursue this generalization in the present paper and work exclusively with square domains.

\begin{theorem}\label{PM1}{ \bf [Subcritical case]} Suppose that $\theta \in (0, \theta_c)$, where $\theta_c$ is as in (\ref{DefThetaC}). For all $y\in \R$
\begin{equation}\label{HLConv}
\lim_{N \rightarrow \infty} \mathbb{P} \left( \frac{\Fe_N + 2 N \Psi(\theta/2)}{\sigma_{\theta} N^{1/3}} \leq y\right) = F_{\rm GUE}(y),
\end{equation}
where $F_{\rm GUE}$ is the GUE Tracy-Widom distribution \cite{TWPaper}.
\end{theorem}

\begin{theorem}\label{PM2}{ \bf [Critical case]} Suppose that $\theta  = \theta_c$, where $\theta_c$ is as in (\ref{DefThetaC}). There are positive constants $c_2>  c_1 > 0$ such that
	\begin{equation}\label{PM2E1}
	\lim_{N \rightarrow \infty} \mathbb{P} \left(c_1 N^{1/3} (\log N)^{2/3} \leq \Fe_N \leq c_2 N^{1/3} (\log N)^{2/3} \right) = 1.
	\end{equation}
\end{theorem}

\begin{theorem}\label{PM3}{ \bf [Supercritical case]} Suppose that $\theta  > \theta_c$, where $\theta_c$ is as in (\ref{DefThetaC}). There exist constants $c_2 > c_1 > 0$ (depending on $\theta$) such that
	\begin{equation}\label{PM3E1}
	\lim_{N \rightarrow \infty} \mathbb{P} \left(c_1 \log N \leq \Fe_N \leq c_2 \log N \right) = 1 .
	\end{equation}
\end{theorem}

\begin{figure}[h]
\scalebox{0.45}{\includegraphics{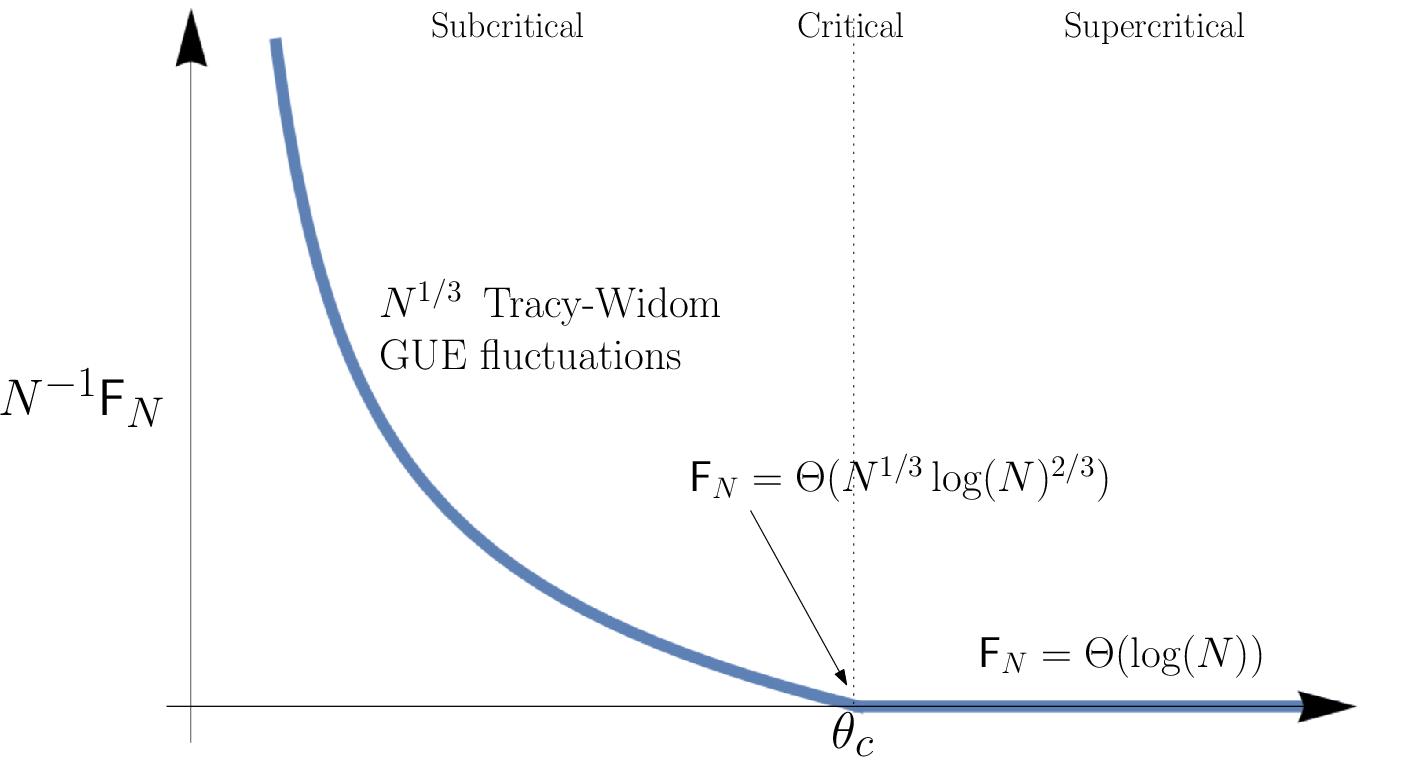}}
\captionsetup{width=\linewidth}
 \caption{The phase diagram proved in Theorems \ref{PM1}, \ref{PM2} and \ref{PM3}.}
\label{fig:PhaseDiagram}
\end{figure}


\cite[Theorem 1.1]{kotowski2019tracy} proved \eqref{HLConv} in the case that $\theta \in (0,\theta_c/2)$. The limitation of their result to half of the subcritical range can be understood in terms of the energy / entropy balance. We call the energy of a single path the logarithm of its weight, or equivalently the sum of $\log w_{i,j}$ along the path. The free energy $\log Z(m_1,n_1;m_2,n_2)$ will exceed the single path energy from $(m_1,n_1)$ to $(m_2,n_2)$. This is due to entropy -- the fact that there are exponentially many different paths with the same starting and ending points. The balance of energy and entropy is quite non-trivial and accounts for the complicated nature of the free energy. The range $\theta\in (0,\theta_c/2)$ is exactly when $\mathbb E[\log w_{i,j}]>0$, and hence the single path energy is strictly positive. To maximize the single path energy in that case, one maximizes the length of the path. Entropy only further encourages longer path lengths and hence, for $\theta\in (0,\theta_c/2)$, both energy and entropy push towards having longer paths. The specifics of the proof of \cite[Theorem 1.1]{kotowski2019tracy} rely on the strict positivity $\mathbb E[\log w_{i,j}]>0$ and this logic. Hence the argument does not extend to $\theta\geq \theta_c/2$.

For $\theta>\theta_c/2$, the single path energy becomes negative. However, up to $\theta_c$, the free energy remains positive since the number of paths overwhelms the average negativity of each one. Beyond $\theta_c$, the balance shifts and the entropy is not strong enough of a force to keep the free energy positive. This represents the transition to the supercritical phase. In order to address the full range of $\theta$ (in particular, $\theta\geq \theta_c/2$) we develop a  different approach than used in \cite{kotowski2019tracy}, one which relies on the free energy directly and not the single path energy.


The following heuristic shows how to use the free energy to argue that the critical value should be $\theta_c$ (as we show) and not $\theta_c/2$.
For points $(s_1,s_2)\leq (t_1,t_2)$ in the unit square $[0,1]^2$, the free energy $N^{-1} \log Z(\lfloor s_1N \rfloor ,\lfloor s_2N \rfloor;\lfloor t_1N \rfloor ,\lfloor t_2 N \rfloor)$ converges to an explicit function $g_{\theta}(s_1,s_2;t_1,t_2)$ (we do not define it here), which is determined by the law of large numbers proved in \cite[Theorem 2.4]{Sep12}. One expects that $N^{-1}\mathsf{F}_N$ behaves like the maximum of $g_{\theta}(s_1,s_2;t_1,t_2)$ over endpoints varying in the unit square. Due to convexity of $g_{\theta}$ (a general property of such limit shape functions), the maximizing endpoints will lie on a diagonal. For $(s_1,s_2)=(a,a)$ and $(t_1,t_2)=(b,b)$, $g_{\theta}(a,a;b,b) = -2\Psi(\theta/2)(b-a)$. If $-2\Psi(\theta/2)>0$, this will be maximized for $a=0$ and $b=1$ and take value $-2\Psi(\theta/2)$; otherwise it will be maximized for any choice of $a=b$ and take value $0$. One easily checks that $\theta_c$ is the value for $\theta$ which divides between these two cases. For $\theta\in (0,\theta_c)$, the maximizer will be $a=0$ and $b=1$ whereas for $\theta>\theta_c$, it will be $a=b$. This explains the limit for $N^{-1}\mathsf{F}_N$. We mention here that near the end of \cite[Section 1]{kotowski2019tracy} there is a related (abstract) discussion for the existence of a critical value and a phase transition. This phase transition is not formulated in terms of the polymer model, but through a certain related random operator (we discuss this relation in detail in Section \ref{sec:randomoperator} below). Our results above can be viewed as a confirmation of the existence of such a phase transition with an explicit formula for the critical value where it occurs.

Let us turn now to the higher resolution information contained in our results, beyond the limit of $N^{-1}\mathsf{F}_N$. In the subcritical regime, $\theta\in (0,\theta_c)$, we show that $\mathsf{F}_N$ and $\log Z(1,1;N,N)$ behave the same down to order $N^{1/3}$, and thus conclude our theorem by known asymptotic results for the later random variable \cite{BCR, KQ, BCDA}. In the supercritical regime, $\theta>\theta_c$ one might hope that the maximizing points will be order one apart, and hence the $\log(N)$ behavior for $\mathsf{F}_N$ makes sense from standard extreme value theory. In the critical case, $\theta=\theta_c$, the $N^{1/3}\log(N)^{2/3}$ behavior is a bit more subtle and comes from a competition between many roughly independent possible maximizers on different scales. The $\log(N)^{2/3}$ factor is related to the upper tail $3/2$ exponent for the point-to-point free energy.

\subsection{Ideas in the proof}
In this section we explain what are the main challenges and ideas needed to overcome them in the proofs of our theorems. We will focus on the subcritical case since its proof is the most involved and makes use of all of the tools employed in the other proofs.
We will also explain the key roles of our recent works \cite{BCDA} and \cite{BCDB} in the proof. These works are both based upon special integrable structure (see, e.g. \cite{COSZ,BorCor}) enjoyed by the log-gamma polymer -- the first relies on a Fredholm determinant formula for the Laplace transform of the polymer partition function while the second relies on the Gibbsian line ensemble (or equivalently, the Whittaker process) into which the polymer free energy embeds. Let us assume below that $\theta\in (0,\theta_c)$, which is the subcritical case.

The heuristic law of large numbers argument presented above suggests that the maximizing points defining $\mathsf{F}_N$ should lie within a $o(N)$ window of $(1,1)$ and $(N,N)$. By taking into account the KPZ scaling which describes the fluctuations around the law of large numbers, we can see that the maximizing points should really lie in windows of side length $k=N^{1/3+\epsilon}$ (for any positive $\epsilon$) touching $(1,1)$ and $(N,N)$. This can be seen from the strict convexity of the law of large numbers which suggests that even on a more local scale, there is a linear drift which pushes the maximizer towards the extreme points $(1,1)$ and $(N,N)$. This drift is balanced once the fluctuations of the free energy are commensurable with it, which happens in a scale of order $N^{1/3}$ from the extreme points. \cite[Theorem 1.2]{BCDA} (quoted here as Proposition \ref{LGPCT}) implies the $N^{1/3}$ fluctuations (see also earlier work of \cite{BCR, KQ}). In order to discount all other starting and ending points besides those in the $(k\times k)$-windows around the extreme points, we use \cite[Theorem 1.7]{BCDA} (quoted here as Proposition \ref{S1LDE1}) which provides a moderate deviation upper bound on the upper-tail of the free energy fluctuations. This upper bound enables us to take a union bound over all starting and ending points outside the windows.

Under suitable centering and KPZ scaling, the free energy profile for varying starting and ending points is believed to converge to the directed landscape (constructed in \cite{DOV18}). If we had such a result at our disposal, then estimates on the local modulus of continuity for the directed landscape \cite[Proposition 1.6]{DOV18} would imply that the maximum over the $(k\times k)$-windows is, in the $N^{1/3}$ scale, the same as $\log Z(1,1;N,N)$. Such convergence to the directed landscape has, however, only been established for zero-temperature models \cite{DOV18,dauvergne2021scaling}.

For the log-gamma polymer we can control a much weaker type of local modulus of continuity. Namely, if we fix the starting point and then vary the ending point along a single down-right path, then we can show that the free energy profile (properly centered and scaled) looks similar (up to a Radon-Nikodym derivative) to a centered random walk with log-gamma distributed jump distribution. This result is shown in \cite{BCDB} as a consequence of the fact that this free energy profile can be embedded as the lowest labeled curve of a Gibbsian line ensemble (see also \cite{COSZ,JonstonOConnell,Wu19}). Specifically, \cite[Theorem 1.10]{BCDB}, quoted here as Proposition \ref{ThmTight}, shows tightness of the free energy profile.

The problem for us is that both the starting and ending points vary in their respective $(k\times k)$-windows. With a bit of work, we can restrict to only consider varying the starting and ending points along the one-dimensional down-right path on the boundary of the $(k\times k)$-windows (we call these frames). But still, even if we know that for each given starting point, the free energy profile is tight as the endpoint varies, this does not imply that the tightness holds simultaneously as both vary. This is because there are order $k\gg 1$ different starting points. Without more quantitative control on the modulus of continuity, such as moment or tail bounds, this approach is not feasible. The proof of \cite[Theorem 1.3]{CGH19} provides a route (in the case of the KPZ equation) for using the Brownian Gibbs property to establish such moment / tail bounds on the modulus of continuity. Almost all of the pieces used there are present for the log-gamma polymer, except for a lower tail bound on the free energy fluctuations. For the KPZ equation, that piece is supplied by \cite{10.1215/00127094-2019-0079}. Unfortunately, such a lower tail bound for the log-gamma polymer is currently unavailable (and the approach of \cite{10.1215/00127094-2019-0079} does not readily apply). As such, we do not currently know how to establish modulus of continuity bounds when we simultaneously vary the starting and ending points. (The key identity \cite[Proposition 4.1]{DOV18} in the  construction of the directed landscape in \cite{DOV18} is generalized  in \cite{corwin2020invariance} to the log-gamma polymer -- a glimmer of hope that the methods of \cite{DOV18} may eventually be lifted to polymers.)

So, the remaining challenge in our work is how to only use the one-free endpoint tightness result of Proposition \ref{ThmTight} to deduce that varying both starting and ending points along the frames of the $(k\times k)$-windows cannot result in $\mathsf{F}_N$ exceeding $\log Z(1,1;N,N)$ on the $N^{1/3}$ scale. This is achieved by expanding to a larger window and considering $\log Z(-K,-K;N+K,N+K)$ for $K$ such that $K\ll N$ and $K^{2/3} \gg k$. We can lower bound $Z(-K,-K;N+K,N+K)$ by a summation indexed by points on the two frames, with summands given by the product of three partition functions -- one from $(-K,-K)$ to the point on the frame near $(1,1)$, one from that point to the point on the frame near $(N,N)$, and one from that frame point to $(N+K,N+K)$. Since $(-K,-K)$ and $(N+K,N+K)$ are fixed,  the two outer partition functions in this summand involve a fixed starting point and varying ending point. Thus, we can use Proposition \ref{ThmTight} to control how much these can vary as the points on the frame vary. Here we use that $k\ll K^{2/3}$ which means that on the frame we are zooming into the local modulus of continuity is as discussed above. Using this, we can reverse the inequality to upper bound the maximal frame-to-frame partition function by $Z(-K,-K;N+K,N+K)$ minus well-controlled point-to-frame partition functions. Since $K\ll N$, the contribution from the point-to-frame partition functions ends up being negligible (the free energy fluctuations are order $K^{1/3}\ll N^{1/3}$). Similarly, since $K\ll N$, the free energy fluctuations of  $\log Z(-K,-K;N+K,N+K)$ on the scale $N^{1/3}$ have the same law as those of $Z(1,1;N,N)$. This completes the sketch of the argument used to prove Theorem \ref{PM1}.

\subsection{Further questions}
\label{sec:ideasloggamma}

We have not attempted to optimize the values of the constants $c_1,c_2$ in Theorems \ref{PM2} and \ref{PM3} or determine the scale and distribution of fluctuations of $\Fe_N$. For instance, our results do not rule out the possibility that $\Fe_N$ actually fluctuates on a smaller scale than the window specified by $c_1$ and $c_2$. We can easily deduce some numerical bounds on $c_1$ and $c_2$ from our proofs. In the critical case, we show (see Section \ref{Section4.1}) that one may chose $c_1= 20^{-1} \sigma_{\theta_c}$, where $\sigma_{\theta}$ is as in \eqref{DefSigma}. The value of $c_2$ that we may choose depends on the value of the constant $D_2$ in Lemma \ref{S2S2L2}, for which  we do not have a numerical estimate. In the supercritical case, the proof of Lemma \ref{S5LB} shows that one may take any $c_1< 2/\theta$. In the case $\theta>3$ (note that $\theta_c<3$), we may choose any constant $c_2>2$ (see  Remark \ref{rem:betterc2} for more details). When $\theta_c < \theta\leq 3$ the constant $c_2$ depends again on the value of the constant $D_2$ in Lemma \ref{S2S2L2}, as in the critical case.

Another natural problem concerns the location of the maximizing starting and ending points. While the proof of Theorem \ref{PM1} tells us that the maximum free energy is attained for very long polymers with starting and ending points located near the corners of the $N\times N$ square (see for instance Lemma \ref{LCorners} below), the location of maximizing starting and ending points in the critical and supercritical is less clear. In the supercritical case $\theta>\theta_c$ we expect that the maximizing starting and ending points are close to each other. Note that, even for a large (but fixed) $\theta$, the probability that the starting point and ending point of the maximal partition function are not the same does not go to zero as $N$ goes to infinity.  Indeed, let us call $v \in \llbracket 1,N\rrbracket^2$ the vertex where the  maximal weight is attained. Consider the event that on one of the vertices adjacent to $v$, one of the weights is larger than $1$, and call $w$ this weight. The probability of this event does not vanish as $N$ goes to infinity. Moreover, on this event, we have $\Fe_N\geq \log w^{\max}_N +\log w$ and the maximal path hence cannot be reduced to a single point.

\subsubsection*{Outline of the paper} In Section \ref{SectionS} we introduce the log-gamma polymer measure and the random operator from \cite{kotowski2019tracy}, and state our results about these models. In Section \ref{Section2} we recall some preliminary results from our recent papers \cite{BCDA} and \cite{BCDB}, and establish some estimates on the distribution of $\log Z (1,1; M,N)$, in particular the important tail estimate from Lemma \ref{S2S2L2},  that we will use in the following sections. Section \ref{Section3} is devoted to the proof of Theorem \ref{PM1} (subcritical case). Section \ref{Section4} is devoted to the proof of Theorem \ref{PM2} (the critical case) and Section \ref{Section5} is devoted to the proof of Theorem \ref{PM3} (supercritical case). Section \ref{Section7} contains the proofs of our results about the log-gamma polymer measure from Section \ref{SectionS}. In Appendix \ref{Section6} we prove Proposition \ref{S2LBProp}, an estimate for the speed of convergence of $\log Z(1,1;M,N)$ to the Tracy-Widom distribution that we use in Section \ref{Section4}. 

\subsubsection*{Acknowledgments}

We thank Marcin Kotowski and B\'alint Vir\'ag for useful conversations about their work \cite{kotowski2019tracy}. G.B. also thanks Herbert Spohn for illuminating explanations about the relations between random operators and directed polymers. I.C. is partially supported by the NSF grants DMS:1811143 and DMS:1664650 as well as a Packard Foundation Fellowship for Science and Engineering. G.B. was partially supported by NSF grant DMS:1664650 as well. E.D. is partially supported by the Minerva Foundation Fellowship and NSF Grant DMS:2054703.

%
\section{Applications}\label{SectionS}  In Section \ref{SectionS.1} we consider a polymer measure that naturally arises in the log-gamma polymer model and we investigate the asymptotic behavior of paths sampled according to this measure. Specifically, we show that as one varies the $\theta$ parameter of the log-gamma polymer the polymer measure behaves quite differently depending on whether $\theta \in (0, \theta_c)$, $\theta = \theta_c$ or $\theta \in (\theta_c, \infty)$, where $\theta_c$ is as in (\ref{DefThetaC}). In Section \ref{sec:randomoperator} we recall the random operator from \cite{kotowski2019tracy} and show how our results from Section \ref{Section1.1} can be directly applied to study the asymptotic behavior of the smallest positive eigenvalue of this operator. Throughout this section we use freely notation from Section \ref{Section1.1}.

%
\subsection{The log-gamma polymer measure}\label{SectionS.1} We proceed to define our main object of study in this section.
\begin{definition}\label{Def2} Fix $\theta > 0$, $N \in \N$ and continue with the same notation as in Definition \ref{Def1}. Denote $\Pi_N = \cup_{m_1, n_1, m_2, n_2  \in \llbracket 1 ,N \rrbracket} \Pi(m_1, n_1; m_2, n_2)$, which is the collection of all directed lattice paths, which start and end in the square $\llbracket 1, N \rrbracket \times \llbracket 1, N \rrbracket$. We define the {\em log-gamma polymer measure} to be the (depending on $w$ and hence random) probability measure on $\Pi_N$ defined by
\begin{equation}\label{PolMeasEqn}
\PM (\pi) = \left(\ZM \right)^{-1} \cdot w(\pi), \mbox{ where } \ZM = \sum_{\pi \in \Pi_N} w(\pi) = \sum_{\substack{(m_1, n_1) \leq (m_2,n_2)\\  m_1, n_1, m_2, n_2 \in \llbracket 1, N \rrbracket}} Z(m_1,n_1;m_2,n_2).
\end{equation}
By possibly enlarging the probability space $(\Omega, \mathcal{F}, \mathbb{P})$, we may assume that in addition to $w$ it supports, for each $N \in \mathbb{N}$, a random $\pi_N$ in $\Pi_N$, whose conditional distribution given $w$ is $\PM(\cdot)$. We continue to call this new space $(\Omega, \mathcal{F}, \mathbb{P})$.

If $\pi$ is given by the sequence $(x_1, y_1), \dots, (x_k, y_k)$ as in Definition \ref{Def1}, we denote its {\em starting point} by $(\XS(\pi), \YS(\pi)) = (x_1, y_1)$ and its {\em ending point} by $(\XE(\pi), \YE(\pi)) = (x_k, y_k)$. We also denote the {\em length} of the path by $\LP(\pi) = k-1 = (x_k - x_1) + (y_k - y_1)$. Finally, we define the {\em slope} of the path by $\SP(\pi) = \frac{y_k - y_1 + 1}{x_k - x_1 + 1} \in (0, \infty)$. Using the random up-right path $\pi_N$, we can now define the following random variables on $(\Omega, \mathcal{F}, \mathbb{P})$
\begin{equation}\label{MainRV1}
\XS_N = \XS(\pi_N), \hspace{1mm} \YS_N = \YS(\pi_N),  \hspace{1mm} \XE_N = \XE(\pi_N), \hspace{1mm} \YE_N = \YE(\pi_N), \hspace{1mm} \LP_N = \LP(\pi_N) \mbox{, } \SP_N = \SP(\pi_N).
\end{equation}
In view of Definition \ref{Def1}, we have that $\XS_N,\YS_N  , \XE_N  ,\YE_N$ are random variables in $\llbracket 1, N \rrbracket$, $\LP_N $ is a random variable in $\llbracket 0 , 2N - 2\rrbracket$ and $\SP_N$ is a random variable in $[1/N, N]$. 
\end{definition}

We are interested in describing the asymptotic behavior of the variables in (\ref{MainRV1}) as $N \rightarrow \infty$. We summarize our results in the next sequence of propositions, whose proofs are given in Section \ref{Section7}.

\begin{proposition}\label{PMP1}{ \bf [Subcritical case]}  Suppose that $\theta \in (0, \theta_c)$, where $\theta_c$ is as in (\ref{DefThetaC}). For any $\epsilon > 0$ we have that the following sequences 
\begin{equation}\label{SubCritAll}
\begin{split}
\frac{\XS_N }{N^{1/3 + \epsilon}}, \hspace{1mm} \frac{\YS_N }{N^{1/3 + \epsilon}}, \frac{\XE_N - N }{N^{1/3 + \epsilon}}, \hspace{1mm} \frac{\YE_N - N }{N^{1/3 + \epsilon}}, \hspace{1mm} \frac{\LP_N - 2N }{N^{1/3 + \epsilon}}, \hspace{1mm} \frac{N (\SP_N - 1)}{ N^{1/3 + \epsilon}}
\end{split}
\end{equation}
all converge to $0$ in probability as $N \rightarrow \infty$.
\end{proposition}

\begin{proposition}\label{PMP2}{ \bf [Critical case]}  Suppose that $\theta  = \theta_c$, where $\theta_c$ is as in (\ref{DefThetaC}). There exists a constant $c_1 \in (0,1)$ such that 
 \begin{equation}\label{CritLength}
\begin{split}
\lim_{N \rightarrow \infty} \mathbb{P}(  c_1 N \leq \LP_N \leq 2N ) = 1.
\end{split}
\end{equation}
In addition, for each $\epsilon > 0$ we have that $\frac{N (\SP_N - 1)}{ N^{2/3 + \epsilon}}$ converges to $0$ in probability as $N \rightarrow \infty$.
\end{proposition}

\begin{proposition}\label{PMP3}{ \bf [Supercritical case]}  Suppose that $\theta \in (\theta_c, \infty)$, where $\theta_c$ is as in (\ref{DefThetaC}). There exists a constant $c_2 > 0$ (depending on $\theta$) such that 
 \begin{equation}\label{SuperCritLength}
\begin{split}
\lim_{N \rightarrow \infty} \mathbb{P}(  \LP_N \leq c_2 \log N ) = 1.
\end{split}
\end{equation}
If $\theta \in (3, \infty)$, then the sequence $\LP_N$ is tight. In addition, for any $\theta \in (\theta_c, \infty)$ the sequence of random vectors $(N^{-1} \XS_N, N^{-1} \YS_N)$ converges weakly to the uniform measure on $(0,1)^2$.
\end{proposition}
\begin{remark} From (\ref{CritLength}) and (\ref{SuperCritLength}) we see that the behavior of $L_N$ in the critical case $\theta = \theta_c$ is clearly different from the supercritical case $\theta > \theta_c$. However, our methods do not provide enough information in the critical case to distinguish it from the subcritical case $\theta \in (0, \theta_c)$. We also mention that while $S_N \rightarrow 1$ both in the critical and subcritical case, the speed at which this convergence (provably) occurs in Proposition \ref{PMP1} is {\em faster} than in Proposition \ref{PMP2}.
\end{remark}
\begin{remark} From \cite[page 259]{AS65} we have that $\theta_c \approx 2.92326$. Consequently, Proposition \ref{PMP3} states that for $\theta \in (\theta_c, 3]$ the path length $\LP_N$ is with high probability $O(\log N)$, while for $\theta > 3$ it is with high probability $O(1)$. It would be very interesting to see if there is a range of parameters in $(\theta_c, 3]$ where the length $\LP_N$ is indeed of logarithmic order, or if for any $\theta > \theta_c$ the length is of unit order. At this point our methods appear to be insufficient to address this question as they can only provide an {\em upper bound} for the range $\theta \in (\theta_c, 3]$.
\end{remark}
\begin{remark} 
A great variety of polymer models, directed or not, have been considered in polymer physics \cite{de1979scaling}. Directed polymer models with variable length may arise for instance in the context of polymer melts \cite{le1991statistical}. The log-gamma polymer measure that we consider is hardly a model of polymer melt, since we are considering a single polymer chain. However, the weights in our model can be thought of as the exponential of some monomer-monomer bond energy, and we prove the existence of a sharp transition in polymer length, as the temperature varies, in presence of quenched disorder.  
\end{remark}

%
\subsection{A random operator on the honeycomb lattice}
\label{sec:randomoperator}
Part of our motivation for studying the asymptotic behavior of $\Fe_N$ comes from a random operator, introduced by Kotowski and Vir\'ag  in \cite{kotowski2019tracy}. The authors showed that the smallest positive eigenvalue of this operator is related to the random variable $\Fe_N$ for the log-gamma polymer model, and this allowed them to deduce Tracy-Widom fluctuations for the logarithm of the smallest positive eigenvalue in a certain range of parameters. Before explaining the exact relation between the random operator and the log-gamma polymer (see Section \ref{sec:relation}), let us first define the model and state our results.

\begin{definition} \label{def:randomoperator}
For all $N\geq 1$, we define a planar graph  $G_N$, depicted in Figure \ref{fig:lattice}, as follows. Consider a tiling of the Euclidean  plane by regular hexagons of sidelength $1$. To identify this tiling among its translations and rotations, assume that three hexagons meet at the origin, and their edges form angles $0, 2\pi/3, 4\pi/3$. A graph $G=(V,E)$ is associated  to this tiling, where the vertex set $V$ consists of all points in the plane where three hexagons meet, and the edge set $E$ consists of pairs of vertices connected by the common edge of two adjacent hexagons. We define $G_N=(V_N, E_N)$ as the subgraph of $G$ consisting of vertices (we identify the plane with  $\mathbb{C}$)
$$\darkbluevertex_{x,y}:= x(1+e^{-\I\pi/3})+y(1+e^{\I\pi/3})\quad \text{ and }\quad \lightbluevertex_{x,y}:=1+\darkbluevertex_{x,y},$$
for all $x,y\in \llbracket 1,N\rrbracket$. As can be seen in Figure \ref{fig:lattice}, $G_N$ is a portion of a honeycomb lattice. The edge set $E_N$ consists of edges in $E$ connecting two vertices in $V_N$.
\begin{figure}
	 \captionsetup{width=\linewidth}
	\begin{center}
		\begin{tikzpicture}[scale=0.6]
		\foreach \a  in {0, ..., 3}
		\foreach \b in {0, ..., 3}
		{\draw[line width=2pt, blue!70!white] ({3/2*(\a+\b)},{sqrt(3)/2*(\b-\a)}) -- ({3/2*(\a+\b)+1},{sqrt(3)/2*(\b-\a)});}
		\foreach \a  in {0, ..., 3}
		\foreach \b in {0, ..., 2}
		\draw[thick, red] ({3/2*(\a+\b)+1},{sqrt(3)/2*(\b-\a)})-- ({3/2*(\a+\b+1)},{sqrt(3)/2*(1+\b-\a)});
		\foreach \a  in {0, ..., 2}
		\foreach \b in {0, ..., 3}
		\draw[thick, red] ({3/2*(\a+\b)+1},{sqrt(3)/2*(\b-\a)})-- ({3/2*(\a+\b+1)},{sqrt(3)/2*(-1+\b-\a)});
		\foreach \a  in {0, ..., 3}
		\foreach \b in {0, ..., 3}
		{	\filldraw[fill=black, draw=black] ({3/2*(\a+\b)},{sqrt(3)/2*(\b-\a)}) circle(0.12);
			\filldraw[fill=white, draw=black, thick] ({3/2*(\a+\b)+1},{sqrt(3)/2*(\b-\a)}) circle(0.12);}
		\draw[thick, gray, ->] (-3/2,0) -- (4.5,{-4*sqrt(3)/2}) node[anchor=west] {$x$};
		\draw[thick, gray, ->] (-3/2,0) -- (4.5,{4*sqrt(3)/2}) node[anchor=west] {$y$};
		
		
		\end{tikzpicture}
	\end{center}
	\caption{The planar graph $G_N$ for $N=4$.}
	\label{fig:lattice}
\end{figure}

The graph $G_N$ is bipartite, as depicted in Figure \ref{fig:lattice} where black vertices (of the form $\darkbluevertex_{x,y}$) are connected to white vertices (of the form $\lightbluevertex_{x,y}$). We associate weights to the edges of $G_N$, in such a way that for each (blue) horizontal edge $e=e_{x,y}$ in $E_N$ connecting $\darkbluevertex_{x,y}$ to $\lightbluevertex_{x,y}$, its weight $w(e)=1/w_{x,y}$ is a gamma random variable with shape parameter $\theta$, that is with density
$ \frac{1}{\Gamma(\theta)}x^{\theta-1}e^{-x} \mathbf{1}_{x>0},$
 and all gamma  random variables are independent. All other edges (the red ones) have weight $1$ .

Because $G_N$ is bipartite, its weights adjacency matrix $A_N$ can be written in the form
\begin{equation}
A_N = \begin{pmatrix}
0 & \widetilde{A}  \\
\widetilde{A}^t & 0
\end{pmatrix}
\label{eq:defAN}
\end{equation}
where $\widetilde A$ is an $N^2$ by $N^2$ matrix whose rows are indexed by black vertices and columns are indexed by white vertices. We may consider that rows and columns are both indexed by points $(x,y) \in \llbracket 1,N\rrbracket^2$. The matrix $A_N$ may then be thought of as an operator acting on functions $\llbracket 1,N\rrbracket^2 \to \mathbb C^2$. The matrix $\widetilde A$ acts on functions  $f:  \llbracket 1,N\rrbracket^2 \to \mathbb C$ as
\begin{equation}\widetilde A = \nabla_x^{-} +\nabla_y^{-} + w\label{eq:Aoperator}
\end{equation}
where $w$ acts by pointwise multiplication $wf(x,y) = \frac{1}{w_{x,y}}f(x,y)$, and the $\nabla^-$ are shift operators
\begin{equation}
\nabla_x^{-} f (x,y) = \begin{cases} f(x-1,y) \text{ for }x>0, \\ 0 \text{ for }x=0, \end{cases} \text{ and }\;\;\nabla_y^{-} f (x,y) = \begin{cases}f(x,y-1)  \text{ for }y>0, \\ 0 \text{ for }y=0. \end{cases}\label{eq:defnabla}
\end{equation}
\end{definition}

$A_N$ is almost surely invertible \cite[Lemma 2.1]{kotowski2019tracy} with spectrum consisting of pairs of eigenvalues  $\pm \lambda_1, \dots, \pm \lambda_{N^2}$, $0< \lambda_1 < \dots <\lambda_{N^2}$. It was shown in \cite{kotowski2019tracy} that when $\theta \in (0, \theta_c/2)$,  the smallest positive eigenvalue $\lambda_1$ is exponentially small, with GUE Tracy-Widom distributed fluctuations.   Our results above about the asymptotics of $\Fe_N$ allow us to extend this result and prove the following immediate corollaries (of our results from Section \ref{Section1.1}) about the asymptotic behavior of  $\lambda_{1}$,  the smallest positive eigenvalue of $A_N$. For $\theta \in (0, \theta_c/2)$, the result of Corollary \ref{cor:subcritical} below is already present in \cite[Theorem 1.1]{kotowski2019tracy}. Our results provide a complete phase diagram and show that $\theta_c$ is the critical point dividing exponential decay from polynomial decay of $\lambda_1$ with $N$.
\begin{corollary}{\bf [Subcritical case]} \label{cor:subcritical}
Suppose that $\theta \in (0, \theta_c)$, where $\theta_c$ is as in (\ref{DefThetaC}). For all $r\in \R$,
\begin{equation*}
\lim_{N\to\infty} \mathbb P\left( \frac{-\log \lambda_{1} + 2 N \Psi(\theta/2)}{\sigma_{\theta} N^{1/3}} \leq r \right) = F_{\rm GUE}(r).
\end{equation*}
\label{cor:PM1}
\end{corollary}
\begin{corollary}{\bf[Critical case]}\label{cor:critical}
Suppose that $\theta  = \theta_c$,  where $\theta_c$ is as in (\ref{DefThetaC}). With the same constants $c_2 > c_1 > 0$ as in Theorem \ref{PM2}, we have
\begin{equation*}
\lim_{N \rightarrow \infty} \mathbb{P} \left(c_1 N^{1/3}(\log N)^{2/3} \leq -\log \lambda_1 \leq c_2 N^{1/3}(\log N)^{2/3} +4\log N   \right) = 1.
\end{equation*}
\label{cor:PM2}
\end{corollary}
\begin{corollary}{\bf[Supercritical case]}\label{cor:supercritical}
Suppose that $\theta >  \theta_c$,  where $\theta_c$ is as in (\ref{DefThetaC}). With the same constants $c_2 > c_1 > 0$ as in Theorem \ref{PM3}, we have
\begin{equation*}
\lim_{N \rightarrow \infty} \mathbb{P} \left(c_1\log N  \leq  -\log \lambda_1 \leq (4+c_2) \log N\right) = 1.
\end{equation*}
\label{cor:PM3}
\end{corollary}

The $4\log(N)$ which arises in the critical and supercritical case follows from
Proposition \ref{propktspecial}, which itself comes from the $k=1$ case of (\ref{eq:ksingularvalues}). For the subcritical case, this factor of $4\log(N)$ is still present but can be absorbed into the $N^{1/3}$ fluctuation scale. In Section \ref{sec:relation} we explain how these corollaries follow from Theorems \ref{PM1}, \ref{PM2} and \ref{PM3}.

%
\subsubsection{Relation between the log-gamma polymer and the random operator}
\label{sec:relation}

Corollaries \ref{cor:subcritical}, \ref{cor:critical} and \ref{cor:supercritical} are deduced respectively from Theorems \ref{PM1}, \ref{PM2} and \ref{PM3}, using an exact algebraic relation between the log-gamma polymer from Definition \ref{def:loggammapolymer} and the random operator $A_N$ from Definition \ref{def:randomoperator}. This relation was discovered in \cite{kotowski2019tracy} and is stated in the following proposition.
\begin{proposition}[{Equivalent to \cite[Proposition 2.3]{kotowski2019tracy}}]
The inverse of the matrix  $\widetilde A$, introduced in \eqref{eq:defAN}, satisfies for all $S,T\in \llbracket 1, N\rrbracket^2$,
\begin{equation}
\label{eq:inverseelements}
\big(\widetilde A^{-1}\big)_{T,S} = Z(S; T)(-1)^{\Vert T-S\Vert_1},
\end{equation}
where for $S=(s_1,s_2)$, $T=(t_1,t_2)$ we define $\Vert T-S\Vert_1=\vert t_1-s_1\vert +\vert t_2-s_2\vert $ and  we recall from \eqref{PartitionFunct} that $Z(S;T)=Z(s_1,s_2;t_1,t_2)$ is the partition function of the log-gamma polymer with weights $w_{i,j}= 1/w(e_{i,j})$ as defined in Definition \ref{def:loggammapolymer}.
\label{prop:inverseelements}
\end{proposition}
The statement of Proposition \ref{prop:inverseelements} differs slightly from \cite[Proposition 2.3]{kotowski2019tracy}. In particular, \cite{kotowski2019tracy} considers a more general directed polymer model (see Section \ref{sec:othereigenvalues} below), which in the case of present interest would be a variant of the log-gamma polymer model with edges weighted by $-1$. This is why we included the factor $(-1)^{\Vert T-S\Vert_1}$ in \eqref{eq:inverseelements}. Although the proof of Proposition  \ref{prop:inverseelements} can be found in \cite{kotowski2019tracy}, let us explain where it comes from. Using \eqref{eq:Aoperator}, one can see that \eqref{eq:inverseelements} is equivalent to the recurrence relation of the partition function, which reads
$$ \frac{1}{w_{x,y}} Z(S;T) = \nabla^-_x Z(S;T)+  \nabla^-_y Z(S;T), \quad \textrm{for }  S\leq T \textrm{ with }S\neq T$$
where $\nabla^-_x, \nabla^-_y$ act on the $T$ variable and have been defined in \eqref{eq:defnabla}.

Proposition \ref{prop:inverseelements} allows us to estimate the smallest positive eigenvalue of the adjacency matrix $A_N$ in terms of the partition function of the log-gamma polymer.
\begin{proposition}[{Special case of \cite[Theorem 2.5]{kotowski2019tracy}}]\label{propktspecial}
	Fix $N\in \mathbb{N}$ and recall that $\lambda_{1}$ is the smallest positive eigenvalue of $A_N$. We have
	$ \Fe_N   \leq -\log \lambda_1 \leq \Fe_N + 4\log N.$
	\label{theo:KotowskiVirag}
\end{proposition}
Note that Proposition \ref{theo:KotowskiVirag}  may appear different from the statement of \cite[Theorem 2.5]{kotowski2019tracy}. First,  \cite[Theorem 2.5]{kotowski2019tracy} is more general than Proposition \ref{theo:KotowskiVirag} (see more details in Section \ref{sec:othereigenvalues} below), and second, it provides an estimate for the largest singular value $\sigma_1$ of $\widetilde A^{-1}$, but this singular value is exactly equal to $1/\lambda_1$ hence the statement of Proposition \ref{theo:KotowskiVirag}. \\

In the remainder of this section we use  Proposition \ref{theo:KotowskiVirag} to deduce Corollaries \ref{cor:subcritical}, \ref{cor:critical} and \ref{cor:supercritical} from Theorems \ref{PM1}, \ref{PM2} and \ref{PM3}, respectively.

\begin{proof}[Proof of Corollary \ref{cor:subcritical}]
From Proposition \ref{theo:KotowskiVirag} we have for each $r \in \mathbb{R}$
$$ \mathbb P\left( \frac{\mathsf{F}_N + 2 N \Psi(\theta/2) + 4 \log N }{\sigma_{\theta} N^{1/3}} \leq r \right) \leq \mathbb P\left( \frac{-\log \lambda_{1} + 2 N \Psi(\theta/2)}{\sigma_{\theta} N^{1/3}} \leq r \right)  \leq \mathbb P\left( \frac{\mathsf{F}_N + 2 N \Psi(\theta/2) }{\sigma_{\theta} N^{1/3}} \leq r \right).$$
From Theorem \ref{PM1} and the continuity of $F_{\rm GUE}(x)$ we see that the first and third term above both converge to $F_{\rm GUE}(r)$ as $N \rightarrow \infty$. By the squeeze theorem the same is true for the middle term.
\end{proof}

\begin{proof}[Proof of Corollary \ref{cor:critical}] From Proposition \ref{theo:KotowskiVirag} we have the inclusion of events
\begin{equation*}
\begin{split}
&\left \{ c_1 N^{1/3}(\log N)^{2/3} \leq \mathsf{F}_N \leq c_2 N^{1/3}(\log N)^{2/3}  \right\} \subseteq \\ 
&\left \{ c_1 N^{1/3}(\log N)^{2/3} \leq -\log \lambda_1 \leq c_2 N^{1/3}(\log N)^{2/3} +4\log N \right\}.
\end{split}
\end{equation*}
From Theorem \ref{PM2} the probabilities of the events on the first line converge to $1$ as $N \rightarrow \infty$, and so the same is true for the probabilities of the events of the second line.
\end{proof}

\begin{proof}[Proof of Corollary \ref{cor:supercritical}]  From Proposition \ref{theo:KotowskiVirag} we have the inclusion of events
\begin{equation*}
\begin{split}
&\left \{ c_1\log N  \leq  \mathsf{F}_N \leq c_2 \log N  \right\} \subseteq \left \{ c_1\log N  \leq  -\log \lambda_1 \leq (4+c_2) \log N \right\}.
\end{split}
\end{equation*}
From Theorem \ref{PM3} the probabilities of the left events converge to $1$ as $N \rightarrow \infty$, and so the same is true for the probabilities of the events on the right.
\end{proof}

%

%

%

\subsubsection{Possible extensions}
\label{sec:moregeneralweights}

	Instead of assuming that the weights of the log-gamma  polymer (resp. of the operator $A_N$) are distributed as inverse-gamma random variables (resp. gamma random variables) of shape parameter $\theta$, we may assume that they are distributed as inverse-gamma  random variables (resp. gamma random variables) with shape parameter $\theta$ and scale parameter $\beta$. This extension was considered in \cite{kotowski2019tracy}. The hypothesis used in \cite[Theorem 1.1]{kotowski2019tracy} remains that $ \mathbb E \left[ \log w_{i,j} \right] >0,$ so that the condition on the parameter $\theta$ becomes $\psi(\theta) -\log \beta <0$.
	In presence of such a scale parameter $\beta$, we expect that the optimal condition for the statements  of Theorem \ref{PM1} and Corollary \ref{cor:PM1} to be valid  is $\psi(\theta/2) -\log \beta <0,$
that is when $\lim_{N\to\infty} \frac{\log Z(1,1;N,N)}{N} >0$.

\bigskip

	\begin{figure}[h]
	\captionsetup{width=\linewidth}
	\begin{center}
		\begin{tikzpicture}[scale=0.5]
		\foreach \a  in {0, ..., 3}
		\foreach \b in {0, ..., 3}
		{\draw[line width=2.5pt, blue!70!white] ({3/2*(\a+\b)},{sqrt(3)/2*(\b-\a)}) -- ({3/2*(\a+\b)+1},{sqrt(3)/2*(\b-\a)});}
		\foreach \a  in {0, ..., 3}
		\foreach \b in {0, ..., 2}
		\draw[thick, red] ({3/2*(\a+\b)+1},{sqrt(3)/2*(\b-\a)})-- ({3/2*(\a+\b+1)},{sqrt(3)/2*(1+\b-\a)});
		\foreach \a  in {0, ..., 2}
		\foreach \b in {0, ..., 3}
		\draw[thick, red] ({3/2*(\a+\b)+1},{sqrt(3)/2*(\b-\a)})-- ({3/2*(\a+\b+1)},{sqrt(3)/2*(-1+\b-\a)});
		\foreach \a  in {0, ..., 3}
		\foreach \b in {0, ..., 3}
		{	\filldraw[fill=black, draw=black] ({3/2*(\a+\b)},{sqrt(3)/2*(\b-\a)}) circle(0.12);
			\filldraw[fill=white, draw=black] ({3/2*(\a+\b)+1},{sqrt(3)/2*(\b-\a)}) circle(0.12);}
		\draw[thick, gray, ->] (-3/2,0) -- (4.5,{-4*sqrt(3)/2}) node[anchor=west] {$x$};
		\draw[thick, gray, ->] (-3/2,0) -- (4.5,{4*sqrt(3)/2}) node[anchor=west] {$y$};
		
		\begin{scope}[xshift=16cm, yshift=-3.5cm, scale=1.6]
		\draw[thick, red] (1,1) grid (4,4);
		\foreach \a in {1, ...,4}
		\foreach \b in {1, ..., 4}
		\fill[blue!70!white] (\a,\b) circle(0.08);
		\draw[thick, gray, ->] (0,0) -- (4,0) node[anchor=south] {$x$};
		\draw[thick, gray, ->] (0,0) -- (0,4) node[anchor=west] {$y$};
		\end{scope}
		
		\end{tikzpicture}
	\end{center}
	\caption{We associate to the planar graph $G_N$ (left) a subset of the $\mathbb Z^2$ lattice (right) as follows. The horizontal blue edges of $G_N$ are mapped to vertices in $\Z^2$, while the red oblique edges of $G_N$, which always join two horizontal blue edges, are mapped to the edges of $\Z^2$. For arbitrary weights, the adjacency matrix on $G_N$ is related to a polymer model on $\Z^2$, see Section \ref{sec:moregeneralweights} and \cite{kotowski2019tracy}.}
	\label{fig:latticecorrespondance}
\end{figure}
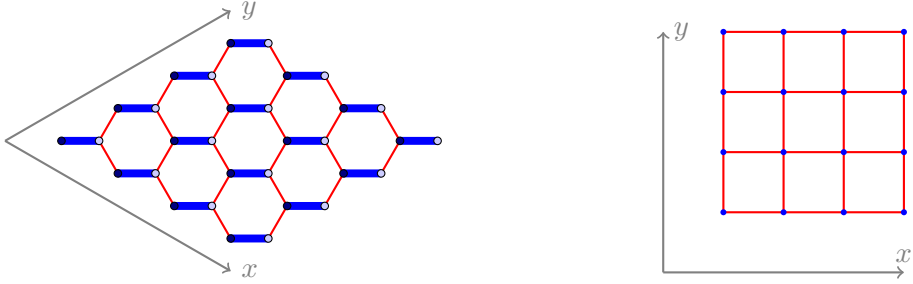

	The statement given in \cite[Proposition 2.3]{kotowski2019tracy} is more general than Proposition \ref{prop:inverseelements}. For any choice of weights $w(e) $ on the edges of $G_N$, \cite[Proposition 2.3]{kotowski2019tracy} shows that
	$ \widetilde A^{-1}_{T,S} = Z(S;T)(-1)^{\Vert T-S\Vert_1},$
	where the partition function $Z(S;T)$ corresponds to a directed polymer model on the $\Z^2$ lattice with arbitrary weights on both vertices and edges. The correspondence between the edges of $G_N$ and the vertices and edges of the $\Z^2$ lattice (on which the polymer model is defined) is illustrated in Figure \ref{fig:latticecorrespondance}. Horizontal edges of $G_N$ are mapped to vertices in $\mathbb Z^2$ while oblique edges in $G_N$ (the red ones) are mapped to edges of $\mathbb Z^2$.  The weights on vertices of $\Z^2$ are the inverses of weights on horizontal edges in $G_N$, as in the statement of Proposition \ref{prop:inverseelements}. However, the polymer weights on edges of $\Z^2$ are the opposite of weights on the red edges of $G_N$.  The result can be further generalized to more general bipartite graphs $G_N$ endowed with a perfect matching, when the matrix $\widetilde A$ defines a directed graph which must contain no cycles except for loops at every vertex (see \cite{kotowski2019tracy} for more details).
	
	This correspondence between directed polymer models and eigenvalues of random operators on the honeycomb lattice could also be applicable  to other integrable polymer models such as the strict-weak polymer \cite{o2014tracy, corwin2014strict} and the beta polymer \cite{barraquand2017random}, as this was already mentioned in \cite{kotowski2019tracy}. In the case of the strict weak polymer, only half of the oblique vertices in $G_N$ would be weighted by independent gamma random variables -- say up-right red edges linking $\lightbluevertex_{x,y}$ to  $\darkbluevertex_{x,y+1}$ -- while all other edges would have weight one. In the case of the beta polymer, horizontal edges would have weight $1$ while the oblique edges would carry  beta  distributed weights: up-right red edges linking $\lightbluevertex_{x,y}$ to  $\darkbluevertex_{x,y+1}$ would carry a beta distributed weight $B_{x,y}$ and down-right red edges linking $\lightbluevertex_{x,y}$ to  $\darkbluevertex_{x+1,y}$ would carry the weight  $1-B_{x,y}$, the collection of $B_{x,y}$ being independent.

\subsubsection{Other eigenvalues}
\label{sec:othereigenvalues}

The statement given  in \cite[Theorem 2.5]{kotowski2019tracy} is more general than Proposition \ref{theo:KotowskiVirag}.  It estimates  the product of the first $k$ largest singular values of $\widetilde A^{-1}$ (i.e. the first $k$ smallest positive eigenvalues of $A_N$) using partition functions for $k$ non-intersecting paths in the log-gamma  polymer maximized over $k$-tuples of starting and ending points. More precisely we have  from \cite[Theorem 2.5]{kotowski2019tracy} that for all $1\leq k\leq N^2$,
\begin{equation} \max_{\mathbf S,\mathbf T} \vert Z(\mathbf S;\mathbf T) \vert  \leq \prod_{i=1}^k \sigma_i\left( \widetilde A^{-1} \right)  = \prod_{i=1}^k \frac{1}{\lambda_i}\leq \binom{N^2}{k}^2 \max_{\mathbf S,\mathbf T} \vert Z(\mathbf S;\mathbf T) \vert,
\label{eq:ksingularvalues}
\end{equation}
where now the maximum runs over pairs of $k$-tuples of distinct points $\mathbf S=(S_1, \dots, S_k), \mathbf T=(T_1, \dots, T_k)$ where   $S_i, T_j\in \llbracket 1,  N\rrbracket^2$, $Z(\mathbf S;\mathbf T)$ is the sum of weights of $k$-tuples of non-intersecting paths with starting and ending points $S_i,T_i$, and $\sigma_i(\cdot)$ denote the singular values of a matrix in decreasing order.
Note that we consider absolute values of partition functions $\vert Z(\mathbf S; \mathbf T)\vert$ in \eqref{eq:ksingularvalues} because the result hold for arbitrary -- not necessarily positive -- weights,  as in Section \ref{sec:moregeneralweights}. We refer to \cite{kotowski2019tracy} for details.
The case $k=1$ of \eqref{eq:ksingularvalues} corresponds to Proposition \ref{theo:KotowskiVirag}.
	
\medskip

When $\theta=0$, that is, if we replace all gamma distributed weights by $0$, then $0$ is an eigenvalue of the matrix $A_N$ with multiplicity $2N$. Indeed, $A_N$ can be decomposed into $2N$ blocks corresponding to connected components in the graph $G_N$ and each of these blocks has corank $1$, so that $0$ is an eigenvalue of $A_N$ with multiplicity $2N$. As $\theta$ becomes positive, but smaller than $\theta_c$, we expect that some of the eigenvalues that are the closest to zero will remain exponentially small. The precise asymptotics of $\lambda_k$, for a fixed $k$, would be controlled by the asymptotics of the maximal partition function for $k$ non-intersecting paths in the log-gamma polymer model, which we expect to grow linearly as $-2\Psi(\theta/2)N$ with fluctuations distributed as the $k$-th eigenvalue in the Airy point process. When $\theta$ becomes larger than $\theta_c$ however, the smallest eigenvalues are no longer exponentially small, as proved in Corollary \ref{cor:PM3}. It would be interesting to determine more precise asymptotics when $\theta=\theta_c$ in order to better understand the nature of this transition.

\subsubsection{Eigenvectors and relation to the literature on random operators}

In the context of random operators, an important question is the localization or delocalization of eigenvectors. In our case, the eigenvectors of $A_N$ are related to the singular vectors of $\widetilde A$, which are in turn related to the singular value decomposition of the matrix of partition functions $(Z(S;T))_{S,T\in \llbracket 1,N\rrbracket^2} $ via Proposition \ref{prop:inverseelements}. It would be very interesting to determine if not only the behavior of the smallest positive eigenvalues obeys a transition at $\theta_c$ but also the localization properties of the associated eigenvectors. However, we emphasize that in the model from Definition \ref{def:randomoperator}, disorder is carried by edges of the lattice and not by vertices, so we may have a different behavior than for the more commonly studied models of the type $\Delta +V$ where $\Delta$ is a discrete or continuous Laplacian and  $V$ is a noise potential on vertices. The latter is often referred to as Anderson model or random Schr\"odinger operator.

\medskip

For operators of the type Laplacian plus noise on a two-dimensional lattice, connections to  directed polymers models have been established numerically  in the physics literature, see \cite{prior2005conductance, somoza2007universal, somoza2015unbinding} and the  references therein. The physics literature predicts that  matrix elements of the resolvant of the operator, which  are considered to be a good probe for the conductance of the physical model described by the operator, can be expanded as a series of products \cite{anderson1958absence}, see also the lecture notes \cite{hundertmark2008short} or the book \cite[Chapter 6]{aizenman2015random}. When the disorder is  strong enough, in the localized phase, this series is expected to be well-approximated by the partition function of a directed polymer model (though the  weights in the auxiliary polymer model are not necessarily positive unlike usual polymer models), and hence have the same  fluctuations as for models in the Kardar-Parisi-Zhang universality class, see e.g. the physics lecture notes \cite{scardicchio2017perturbation}.

%
\section{Preliminary results}\label{Section2}
We present various technical results to be used in the proofs of Theorems \ref{PM1}, \ref{PM2} and \ref{PM3}.

%
\subsection{Summary of results from \cite{BCDA} and \cite{BCDB}}\label{Section2.1} We recall some key results from \cite{BCDA} and \cite{BCDB}.
Define the function
\begin{equation}\label{DefLittleg}
g_{\theta}(z) = \frac{\sum_{n =0}^\infty \frac{1}{(n+\theta - z)^2}}{ \sum_{n = 0}^\infty \frac{1}{(n+z)^2}},
\end{equation}
and observe that $g_{\theta}$ is a smooth, increasing bijection from $(0, \theta)$ to $(0, \infty)$. The inverse function $g_{\theta}^{-1}: (0, \infty) \rightarrow (0,\theta)$ is likewise a strictly increasing smooth bijection. For $x \in (0,\infty)$, define
\begin{equation}\label{HDefLLN}
\begin{split}
h_{\theta}(x) = x \cdot  \Psi(g_{\theta}^{-1}(x)) + \Psi( \theta - g^{-1}_{\theta}(x)),
\end{split}
\end{equation}
with  $\Psi(x)$ from \eqref{digammaS1}. Clearly $h_\theta$ is smooth on $(0, \infty)$.
For $x \in (0,\infty)$, define
\begin{equation}\label{DefSigmaS2}
\sigma_{\theta}( x) := \left( \sum_{n = 0}^\infty \frac{x}{(n+g_{\theta}^{-1}(x))^3} +  \sum_{n = 0}^\infty \frac{1}{(n+\theta - g_{\theta}^{-1}(x))^3} \right)^{1/3}.
\end{equation}

Finally, for all $N,M\geq 1$, we define the rescaled free energy
\begin{equation}\label{eq:rescaledpartitionfunction}
\mathcal{F}(M,N) = \frac{\log Z(1,1;M,N) + M h_{\theta}(N/M)}{M^{1/3}\sigma_{\theta}(N/M)}.
\end{equation}

With the above notation in place we can state the two results we require from \cite{BCDA}.
\begin{proposition}\cite[Theorem 1.2]{BCDA}  \label{LGPCT} Let $\theta > 0$, $\delta \in (0,1)$ and $\mathcal F(M,N)$ be as in \eqref{eq:rescaledpartitionfunction}. Assume that we let $M,N$ go to infinity in such a way that the sequence $N/M \in [\delta, \delta^{-1}]$ for all sufficiently large $M$. Then for all $y\in \R$ we have
\begin{equation}\label{S2HLConv}
\lim_{M, N \rightarrow \infty} \mathbb{P} \left( \mathcal F(M,N)  \leq y\right) = F_{\rm GUE}(y).
\end{equation}
\end{proposition}
\begin{proposition}\label{S1LDE1} \cite[Theorem 1.7]{BCDA} Let $\theta > 0$, $\delta \in (0,1)$ and $\mathcal{F}(M,N)$ be as in \eqref{eq:rescaledpartitionfunction}. There exist constants $C_1, C_2,c_1, c_2 > 0$ (all depending on $\theta, \delta$) such that the following holds for all $M, N \in \mathbb{N}$ with $N/M \in [\delta, \delta^{-1}]$ and $x \geq 0$
\begin{equation}\label{S1LDEeq}
\mathbb{P} \left( \mathcal F(M,N)  \geq x\right) \leq C_1 e^{-c_1 M} + C_2 e^{-c_2 x^{3/2}}.
\end{equation}
\end{proposition}

In addition to Propositions \ref{LGPCT} and \ref{S1LDE1} we will require the following quantitative and uniform comparison of the upper tails of $\mathcal{F}(M,N)$ and $F_{\rm GUE}$. Its proof is postponed until Appendix \ref{Section6.1} and relies on exact formulas from \cite{BCDA}.
\begin{proposition}\label{S2LBProp} Let $\theta > 0$, $\delta \in (0, 1]$, $\epsilon_0 \in (0, 1/3)$. There exist positive constants $M_0, C_0$ (depending on $\theta, \delta$) such that for all $M, N \in \mathbb{N}$ with $M \geq N \geq \delta M$, $M \geq M_0$
\begin{equation}\label{S2LowerBound}
\sup_{x \in [1, \infty)} \left| \mathbb{P} \left( \mathcal{F}(M,N) \leq x \right) - F_{\rm GUE}(x)\right| \leq C_0 M^{-\epsilon_0}.
\end{equation}
\end{proposition}

We next recall the result we require from \cite{BCDB}, for which we need the following definition.
\begin{definition}\label{RandomCurve} Fix any  $T >0$, $\theta > 0$,  $r \in (0,\infty)$, and assume $N$ is large enough so $rN \geq 2 + TN^{2/3}$. For each $x \in [-T-N^{-2/3}, T + N^{-2/3}]$ with $xN^{2/3}\in\mathbb{Z}$, define $n = \lfloor rN \rfloor + xN^{2/3}$ and
\begin{equation}\label{LGCurve}
f_N^{LG}(x)=N^{-1/3}\Big(\log  Z(1,1;n,N)+ h_{\theta}(r)N + h_\theta'(r) x N^{2/3}\Big),
\end{equation}
and then extend $f_N^{LG}$ to $x \in [-T, T]$, by linear interpolation.
Clearly $f_N^{LG}$ is a random continuous curve in $(C[-T,T], \mathcal{C})$ -- the space of continuous functions on $[-T, T]$ with the uniform topology and Borel $\sigma$-algebra $\mathcal{C}$ (see e.g. Chapter 7 in \cite{Bill}) -- and we denote its law by $\mathbb{P}_N$.
\end{definition}
\begin{remark}\label{RemarkFNLG} One way to interpret the function $f_N^{LG}(x)$ is that it is a rescaled and perturbed version of $\mathcal{F}( rN + xN^{2/3}, N)$ that has been shifted by a parabola as we explain here. Using the fact that $g_{\theta}^{-1}(1/x)=\theta-g_{\theta}^{-1}(x)$ we see that $Mh_{\theta}(N/M) = N h_{\theta}(M/N)$ and $M^{1/3} \sigma_{\theta}(N/M) = N^{1/3} \sigma_{\theta}(M/N)$. To get to $f_N^{LG}$ from $\mathcal{F}( rN + xN^{2/3}, N)$ as in (\ref{eq:rescaledpartitionfunction}) one needs to do the following:
\begin{enumerate}[leftmargin=*]
\item Replace $Mh_{\theta}(N/M)$ with $N h_{\theta}(M/N)$ and $M^{1/3} \sigma_{\theta}(N/M)$ with $N^{1/3} \sigma_{\theta}(M/N)$.
\item Set $M = rN  + xN^{2/3}$ and Taylor expand $h_{\theta}(M/N)$ and $\sigma_{\theta}(M/N)$ to see that
$$\mathcal{F}(rN + xN^{2/3}, N) =  \frac{\log Z(1,1; rN \hspace{-0.5mm}+ \hspace{-0.5mm}xN^{2/3} ,N) \hspace{-0.5mm}+\hspace{-0.5mm} N h_{\theta}(r) \hspace{-0.5mm}+\hspace{-0.5mm}  h_{\theta}'(r)x N^{2/3} \hspace{-0.5mm}+ \hspace{-0.5mm} (1/2)h_{\theta}''(r)x^2 N^{1/3} \hspace{-0.5mm}+\hspace{-0.5mm}O(1)}{N^{1/3}\sigma_{\theta}(r)} .$$
\item Subtracting the parabola $ \frac{h_{\theta}''(r)x^2}{2 \sigma_{\theta}(r)}$ from the last expression and multiplying the result by $\sigma_{\theta}(r)$ brings us to $f_N^{LG}(x)$ up to an $O(N^{-1/3})$ error.
\end{enumerate}
Since $f_N^{LG}(x)$ and $\mathcal{F}( rN + xN^{2/3}, N)$ are up to a deterministic shift and scalar multiplication $O(N^{-1/3})$ away from each other any asymptotic statement for $f_N^{LG}(x)$ can be rephrased in terms of $\mathcal{F}( rN + xN^{2/3}, N)$. The next result,  from \cite{BCDB}, is phrased in terms of $f_N^{LG}(x)$ instead of $\mathcal{F}( rN + xN^{2/3}, N)$, hence the need to give this definition.
\end{remark}

\begin{proposition}\label{ThmTight}\cite[Theorem 1.10]{BCDB}
Fix any $T, \theta, r >0$. Then the laws $\P_N$ of $f_N^{LG}$, as in Definition \ref{RandomCurve}, form a tight sequence in $N$.
\end{proposition}

We end this section by summarizing various properties of the functions $h_{\theta}$ and $\sigma_{\theta}$ from (\ref{HDefLLN}) and (\ref{DefSigmaS2}), which we use later in the text.
Note that for any $\theta > 0$, since $g_{\theta}$ is a smooth bijection from $(0, \theta)$ to $(0, \infty)$, it follows that $\sigma_{\theta}$ is a smooth positive function on $(0, \infty)$. For any $\theta > 0$, $g_{\theta}(x) = \frac{\Psi'(\theta -x)}{\Psi'(x)}$ and so $h'_{\theta}(x) = \Psi(g_{\theta}^{-1}(x))$. Since $g_{\theta}^{-1}$ and $\Psi$ are strictly increasing on $(0, \infty)$ the same is true for $h'_{\theta}(x) = \Psi(g_{\theta}^{-1}(x))$.  Also note that $ g_{\theta}^{-1}(1) = \theta/2$ and so $h_{\theta}'(1) = \Psi(\theta/2)$.

Observe that $f(y) = h_{\theta}(y^2)$ is differentiable on $(0, 1]$ with a derivative that is bounded by a constant that depends on $\theta$ alone. Indeed, by the smoothness of $g_{\theta}^{-1}$ on $(0, \infty)$ and $\Psi$ on $(0, \theta)$ the latter is clear for $y$ that are bounded away from $0$. On the other hand, when $y \rightarrow 0+$ we have
$g_{\theta}^{-1}(y^2) \sim y$ and $\Psi(y) \sim - y^{-1},$ so that
$h_{\theta}(y^2) - h_{\theta}(0) = y^2 \cdot  \Psi( g_{\theta}^{-1}(y^2)) + \Psi(\theta - g_{\theta}^{-1}(y^2)) - \Psi(\theta)  =O(y).$
 In particular, there is a constant $b_1 > 0$ (depending on $\theta$) such that for all $0 \leq x \leq y \leq 1$
\begin{equation}\label{DerHBound}
\left| h_{\theta}(x) - h_{\theta}( y ) \right| \leq b_1 \cdot y.
\end{equation}

We next note that there exists $\epsilon > 0$ (depending on $\theta$) so that for all $x \in (0, \epsilon]$ we have
\begin{equation}\label{S2H1}
- \Psi(\theta/2) \geq \epsilon  -h_{\theta}(x).
\end{equation}
To see why such a choice of $\epsilon$ is possible note that when $x \rightarrow 0+$ we have
$g_{\theta}^{-1}(x) \sim \sqrt{x}$ and $\Psi(x) \sim - x^{-1},$
so that
$h_{\theta}(x) = x \cdot  \Psi( g_{\theta}^{-1}(x)) + \Psi(\theta - g_{\theta}^{-1}(x))  =O(\sqrt{x}) + \Psi(\theta).$
The latter implies that as $\epsilon \rightarrow 0+$ and $x \in (0, \epsilon]$ we have
$\epsilon  -h_{\theta}(x)= O(\sqrt{\epsilon}) - \Psi(\theta).$
The last equation and the fact that $ \Psi(x)$ is strictly increasing together imply the existence of $\epsilon \in (0, 1]$ so that (\ref{S2H1}) holds.

When $\theta = \theta_c$ as in (\ref{DefThetaC}) we also have
\begin{equation}\label{S2SummH}
\begin{split}
&h_{\theta_c}'(1) = 0 = h_{\theta_c}(1) \mbox{ and } h_{\theta_c}(1 + x) = \frac{h_{\theta_c}''(1)}{2} \cdot x^2 + O(x^3), \mbox{ where } \\
&h_{\theta_c}''(1) = \Psi' \left( g_{\theta_c}^{-1}(1) \right) \cdot \frac{d}{dx} \left( g_{\theta_c}^{-1}(x) \right) \vert_{x= 1} = \Psi' \left( \theta_c/2 \right) \cdot  \frac{d}{dx} \left( g_{\theta_c}^{-1}(x) \right) \vert_{x= 1} > 0,
\end{split}
\end{equation}
and in the last inequality we used that $g_{\theta_c}^{-1}$ and $\Psi$ are strictly increasing on $(0, \infty)$.

As explained above, $h_{\theta_c}'$ is strictly increasing on $(0, \infty)$ and since $h_{\theta_c}'(1) = 0 $, see (\ref{S2SummH}), we get
\begin{equation}\label{S2SummH2}
\begin{split}
&h_{\theta_c}(x) \geq 0 \mbox{ for $x \in [0,1]$} \mbox{ and } h_{\theta_c}(x) \geq h_{\theta_c}(1/2) > 0 \mbox{ for $x \in [0, 1/2]$}.
\end{split}
\end{equation}

%
\subsection{Miscellaneous  probability results}\label{Section2.2} We formulate three results that will be useful later.
\begin{lemma}\label{TL1} For all $a, \theta > 0$ there is $C> 0$ (depending on $\theta$ and $a$), such that for any $n \in \mathbb{N}$
\begin{equation}\label{S2BoundMax}
\mathbb{P}\left(\max_{1 \leq i \leq n} \log X_i \geq \theta^{-1} (1 + a)  \log n\right) \leq C n^{-a},
\end{equation}
where $\{X_i\}_{i = 1}^n$ are i.i.d. random variables with density $f_{\theta}(x)$ as in (\ref{S1invGammaDens}).
\end{lemma}
\begin{proof} Put $f(a) = \theta^{-1} (1 + a) $. Notice that
$$\mathbb{P}\left(\max_{1 \leq i \leq n} \log X_i \geq f(a) \log n\right) = 1 - \mathbb{P}\left(  X_1 < n^{f(a)} \right)^n.$$
In addition, we have
$$ \mathbb{P}\left(  X_1 < n^{f(a)} \right) = 1 - \int_{n^{f(a)}}^{\infty}  \frac{x^{-\theta - 1}}{\Gamma(\theta)} \cdot \exp( - x^{-1})dx \geq 1 -\int_{n^{f(a)}}^{\infty}   \frac{x^{-\theta - 1}}{\Gamma(\theta)}  dx = 1 - \frac{1}{ \theta \Gamma(\theta) n^{\theta f(a)}} .$$
Using that $\log (1- x) \geq -2x$ for all $x \in [0, 1/4]$ we get
$$\exp \left(n \cdot \log \left( 1 - \frac{1}{ \theta \Gamma(\theta) n^{\theta f(a)}} \right) \right) \geq \exp \left(-\frac{2n}{\theta \Gamma(\theta) n^{\theta f(a)}} \right),$$
provided $ \frac{1}{ \theta \Gamma(\theta) n^{\theta f(a)}} \in [0, 1/4]$. Under this condition, the last three inequalities combine to
$$\mathbb{P}\left(\max_{1 \leq i \leq n} \log X_i \geq f(a) \log n\right) \leq 1 -  \exp \left(-\frac{2n}{\theta \Gamma(\theta) n^{\theta f(a)}} \right).$$
Using that $1 - e^{-x} \leq 2x$ for all $x \in [0, \infty)$ and $f(a) = \theta^{-1} (1 + a)$ the last inequality implies
$$\mathbb{P}\left(\max_{1 \leq i \leq n} \log X_i \geq f(a) \log n\right) \leq \frac{4}{\theta \Gamma(\theta) n^{a}}.$$
This implies (\ref{S2BoundMax}) with $C = \frac{4}{\theta \Gamma(\theta) }$ as long as $ \frac{4}{ \theta \Gamma(\theta) } \leq n^{1+a}$.
Setting
$C := \frac{4}{\theta \Gamma(\theta) } + \left( \frac{4}{ \theta \Gamma(\theta) } \right)^{a/(1+a)},$
we see that  (\ref{S2BoundMax}) holds for $\frac{4}{ \theta \Gamma(\theta) } \leq n^{1+a}$ by the previous argument and for $n^{1+a} \leq \frac{4}{ \theta \Gamma(\theta) } $ it holds as the right side of (\ref{S2BoundMax}) is bigger than $1$. Thus (\ref{S2BoundMax})  holds for all $n \geq 1$ with this choice of $C$.
\end{proof}

\begin{proposition}\cite[Theorem 1]{KMT2} \label{KMT} Let $F$ be a cumulative distribution function on $\mathbb{R}$ and $\lambda > 0$ be given. Suppose that for some $t_0 > 0$ and all $|t| \leq t_0$ we have
$$\int_{\mathbb{R}} xF(dx) = 0, \hspace{5mm}  \int_{\mathbb{R}} x^2F(dx) = 1, \hspace{5mm} \int_{\mathbb{R}} e^{tx} F(dx) < \infty.$$
There exist constants $K,C > 0$, depending on $F$ and $\lambda$, a probability space $(\Omega, \mathcal{F}, \mathbb{P})$ and random variables $\{X_i \}_{i = 1}^\infty, \{Y_i \}_{i= 1}^\infty$ defined on this space that satisfy the following properties.
\begin{itemize}
\item Under $\mathbb{P}$ the random variables $\{X_i \}_{i = 1}^\infty$ are i.i.d. with distribution $F$.
\item Under $\mathbb{P}$ the random variables $\{Y_i \}_{i= 1}^\infty$ are i.i.d. with a standard normal distribution.
\item If $S_k = X_1 + \cdots + X_k$ and $T_k = Y_1 + \cdots + Y_k$ for $k \in \mathbb{N}$ then for any $n \in \mathbb{N}$ and $x > 0$
\begin{equation}
\mathbb{P}\left( \max_{1 \leq k \leq n} |S_k - T_k| > C \log n + x  \right) < K e^{ - \lambda x}.
\end{equation}
\end{itemize}
\end{proposition}

We denote by $\Phi(x)$ and $\phi(x)$ the cumulative distribution function and density of a standard normal random variable. The following lemma can be found in \cite[Section 4.2]{MZ}.
\begin{lemma}\label{LemmaI1}
There is a constant $c_0 > 1$ such that for all $x \geq 0$ we have
\begin{equation}\label{LI2}
 \frac{1}{c_0(1+x)} \leq \frac{1 - \Phi(x)}{\phi(x)} \leq \frac{c_0}{1 +x},
\end{equation}
\end{lemma}

%
\subsection{Bounds on point-to-point free energies}\label{Section2.3} In this section we derive two estimates on $\log Z(1, 1; m,n)$. The first, given in Lemma \ref{S2S2L1}, provides estimates on this quantity that are uniform in the entries $(m,n)$ but are sub-optimal. We use this result to control $\log Z(1, 1; m,n)$ when either $m/n$ or $n/m$ are close to $0$. The second, given in Lemma \ref{S2S2L2}, provides useful bounds on $\log Z(1, 1; m,n)$ when $m, n \in \llbracket 1, N \rrbracket$ in terms of $N$. We mention that Lemma \ref{S2S2L1} is formulated in the case $m \geq n$ to ease the notation, which does not lead to any loss of generality by the distributional equality of $\log Z(1, 1; m,n)$ and $\log Z(1, 1; n,m)$.

\begin{lemma}\label{S2S2L1} Let $\theta > 0$ be given. For any $\epsilon \in (0, 1]$, there exist $B,b > 0$ (depending on $ \theta, \epsilon$) such that for $m,n \in \mathbb{N}$, $m \geq n$ one has
\begin{equation}\label{ThinUBPP}
\mathbb{P}\Big( \log Z(1,1; m,n) \geq m \cdot \epsilon - m \cdot h_{\theta}(n/m) \Big) \leq B \exp \left( -b m \right).
\end{equation}
\end{lemma}
\begin{proof} For clarity we split the proof into four steps.

\smallskip
{\bf \raggedleft Step 1.} In this step we introduce a bit of notation that will be used throughout the proof.

We define $b_0 =  |\Psi(\theta)| + 2 + b_1,$ where $b_1$ depends on $\theta$ and is such that (\ref{DerHBound}) holds. We also set $\epsilon_1 = \epsilon/ b_0.$ In the steps below we will prove that there exist $B,b > 0$ (depending on $ \theta, \epsilon$) such that for $m,n \in \mathbb{N}$, $m \geq n$ one has
\begin{equation}\label{ThinUBPPV2}
\mathbb{P}\Big( \log Z(1,1; m,n) \geq m \cdot b_0 \cdot \epsilon_1 - m \cdot h_{\theta}(n/m) \Big) \leq B \exp \left(-bm \right).
\end{equation}
Since (\ref{ThinUBPPV2}) clearly establishes (\ref{ThinUBPP}) it suffices to prove (\ref{ThinUBPPV2}).

\smallskip
{\bf \raggedleft Step 2.} Put $N = \lceil m \epsilon_1 \rceil$. We claim that there exist constants $M_1, M_2 \in \mathbb{N}$, $D_1, D_2, d_1, d_2 > 0$ (depending on $\theta, \epsilon$) such that for all $m \geq M_1$ we have
\begin{equation}\label{S2S2R2}
\mathbb{P}\Big( \log Z(1,1; m, N)  + m  h_{\theta}(\epsilon_1)  \geq m \epsilon_1 \Big) \leq  D_1\exp \left( - d_1 m \right),
\end{equation}
and for all $m \geq M_2$, and $1 \leq n \leq N$ we have
\begin{equation}\label{S2S2R3}
\mathbb{P}\left( - \sum_{ j = n+1}^N \log w_{m,j} \geq m (b_0 - b_1 - 1)   \epsilon_1  \right) \leq D_2\exp \left( - d_2 m\right),
\end{equation}
where we recall that $w_{i,j}$ are i.i.d. random variables with density (\ref{S1invGammaDens}). We will prove equations (\ref{S2S2R2}) and (\ref{S2S2R3}) in the steps below. Here we will assume their validity and establish (\ref{ThinUBPPV2}) by considering the cases $n/m \in (0, \epsilon_1]$ and $n/m \in [\epsilon_1, 1]$ separately.

\smallskip
Using (\ref{PartitionFunct}) we have for all $m, n \in \mathbb{N}$ with $n/m \in  (0, \epsilon_1]$ that
$$Z(1,1, m, N) \geq Z(1,1,m, n) \cdot \prod_{ j = n+1}^N w_{m, j},$$
 and by taking logarithms on both sides we conclude that
$$\log Z(1,1; m, N) \geq \log Z(1,1; m, n)  + \sum_{ j = n+1}^N \log w_{m,j}.$$
The last equation implies that if $m \geq \max(M_1, M_2)$ and $n/m \in (0, \epsilon_1]$ then
\begin{equation*}
\begin{split}
&\mathbb{P}\Big(\log Z(1,1; m, n)  \geq m \cdot  b_0 \cdot  \epsilon_1 - m \cdot h_{\theta}(n/m) \Big) \leq \\
&\mathbb{P}\left( \log Z(1,1; m, N)   - \sum_{ j = n+1}^N \log w_{m,j}  \geq  m \cdot  b_0 \cdot  \epsilon_1 - m \cdot h_{\theta}(n/m) \right) \leq  \\
&\mathbb{P}\left( \log Z(1,1; m, N)   - \sum_{ j = n+1}^N \log w_{m,j}   \geq m \cdot  (b_0 - b_1) \cdot  \epsilon_1  - m \cdot h_{\theta}(\epsilon_1) \right) \leq  \\
&\mathbb{P}\Big( \log Z(1,1; m, N)   +  m  h_{\theta}(\epsilon_1)   \geq m  \epsilon_1 \Big) +\mathbb{P}\left( - \sum_{ j = n+1}^N \log w_{m,j} \geq m (b_0 - b_1 - 1)  \epsilon_1   \right)   \leq \\
& D_1\exp \left( - d_1 m \right) + D_2\exp \left( - d_2 m\right) ,
\end{split}
\end{equation*}
where in going from the second to the third line we used (\ref{DerHBound}) and in going from the third to the fourth line we used the triangle inequality. In the last inequality we used (\ref{S2S2R2}) and (\ref{S2S2R3}). The last inequality implies (\ref{ThinUBPPV2}) for all $n,m \in \mathbb{N},$ $n/m \in (0, \epsilon_1]$ and $m \geq \max(M_1, M_2)$. By increasing $B$ we can ensure that the RHS of (\ref{ThinUBPPV2}) is greater than or equal to $1$ when $m < \max(M_1, M_2)$, which concludes the proof in the case of $n/m \in (0, \epsilon_1]$.
%
%

\smallskip
If $n/m \in [\epsilon_1, 1]$ we have from Proposition \ref{S1LDE1} with $\delta = \epsilon_1$ that there exist constants $C_1, C_2, c_1, c_2$ (depending on $\theta$ and $\epsilon$) such that for all $n/m \in [\epsilon_1, 1]$ and $x \geq 0$ we have
$$ \mathbb{P} \left( \log Z(1,1; m, n) + m h_{\theta} (n/m) \geq m^{1/3} \sigma_{\theta}(n/m) x \right)  \leq C_1 e^{-c_1 m} + C_2 e^{-c_2 x^{3/2}}.$$
Since $\sigma_{\theta}(y)$ is bounded away from $0$ and $\infty$ for $y \in [\epsilon_1, 1]$, we see that the last inequality with $x = \frac{ \epsilon_1 m^{2/3} b_0 }{\sigma_{\theta}(n/m)} $ implies that there are positive constants $C, c > 0$ (depending on $\theta, \epsilon$) such that
$$ \mathbb{P} \Big( \log Z(1,1; m, n)  + m h_{\theta} (n/m) \geq m \cdot b_0 \cdot \epsilon_1 \Big) \leq C e^{-cm},$$
which implies (\ref{ThinUBPPV2}) in the case $n/m \in [\epsilon_1, 1]$.

\smallskip
{\bf \raggedleft Step 3.} In this step we establish (\ref{S2S2R2}). If $m \geq 2 \epsilon_1^{-1}$ note that $\epsilon_1^{-1} \geq N/m \geq \epsilon_1$. Consequently, by Proposition \ref{S1LDE1} (applied to $\delta = \epsilon_1$, $M = m$ and $N, \theta$ as above) we can find positive constants $C, c$ (depending on $\theta, \epsilon$) such that for $m \geq 2 \epsilon_1^{-1}$
$$\mathbb{P}\left( \log Z(1,1; m, N)   + m h_{\theta}(N/m ) \geq  m^{2/3} \sigma_{\theta}(N/m) \right) \leq C e^{-c m}.$$
We now observe that by definition $N/m = \lceil m \epsilon_1 \rceil /m = \epsilon_1 + O(1/m)$ and using the differentiability of $h_\theta$ and $\sigma_\theta$ we have $h_\theta(\epsilon_1) = h_{\theta}(N/m) + O(1/m)$, $\sigma_{\theta}(\epsilon_1) = \sigma_{\theta}(N/m) + O(1/m)$, where the constants in the big $O$ notations depend on $\theta$ and $\epsilon$. In particular, we conclude that there exists $M_1 \in \mathbb{N}$  (depending on $\theta$ and $\epsilon$) such that for $m \geq M_1$ we have
$$\mathbb{P}\left(\log Z(1,1; m, N)   + m  h_{\theta}(\epsilon_1)  \geq m \epsilon_1 \right)  \leq C e^{-c m},$$
which implies (\ref{S2S2R2}).

\smallskip
{\bf \raggedleft Step 4.} In this final step we prove (\ref{S2S2R3}). Since $w_{i,j}$ are i.i.d. random variables with density (\ref{S1invGammaDens}),  $- \log w_{i,j}$ are i.i.d. random variables with a density on $\mathbb{R}$, given by
$\rho_{\theta}(x) =\frac{ e^{\theta x} \exp( - e^x) }{\Gamma(\theta)}.$
If $X$ has density $\rho_{\theta}$ we see by a direct computation that its  moment generating function equals
$$M_X(t) := \mathbb{E} [ e^{tX} ] =  \int_{\mathbb{R}} \frac{ e^{tx} e^{\theta x} \exp( - e^x) }{\Gamma(\theta)}dx = \frac{\Gamma(\theta + t)}{\Gamma(\theta)},$$
whenever $t \in (-\theta, \infty)$. For $t \leq \theta$ we have $M_X(t) = \infty$. From the definition of $\Psi$ in \eqref{digammaS1}
$$\mathbb{E}[X] =  \frac{d M_X(t) }{dt} \Big \vert_{t = 0} = \Psi(\theta) \mbox{ and } \mathbb{E}[X^2] - \mathbb{E}[X]^2 =  \frac{d^2 M_X(t) }{dt^2} \Big \vert_{t = 0} - \Psi(\theta)^2 = \Psi'(\theta).$$

Let us denote $Z_i = (- \log w_{n+ i ,m} - \Psi(\theta) )/ \sqrt{\Psi'(\theta)}$ for $i = 1, \dots, N - n$. Then $Z_i$ are i.i.d., have mean $0$ and variance $1$ and the moment generating function $M_{Z_1}(t)$ is finite for $t \in (-\theta  \sqrt{\Psi'(\theta)}, \infty)$.

By Proposition \ref{KMT}, applied to $\lambda = 1$, there exist constants $C,K > 0$ (depending on $\theta$ alone) and a probability space $(\Omega_1, \mathcal{F}_1, \mathbb{P}_1)$  that supports random variables $\{X_i\}_{i = 1}^{N-n}$ and $\{Y_i \}_{i = 1}^{N-n}$ such that $(X_1, \dots, X_{N-n})$ has the same distribution as $(Z_1, \dots, Z_{N-n})$ and $Y_i$ are i.i.d. standard normals and also for all $x > 0$
\begin{equation}\label{S2S2R5}
\mathbb{P}_1\left( \left|\sum_{i = 1}^{N-n} X_i - \sum_{i = 1}^{N-n} Y_i \right| \geq C \log (N-n) + x  \right) < K e^{ -  x}.
\end{equation}
Furthermore, since $\sum_{i = 1}^{N-n} Y_i $ is a normal random variable with mean $0$ and variance $N-n$ we have
\begin{equation}\label{S2S2R6}
\mathbb{P}_1\left( \left| \sum_{i = 1}^{N-n} Y_i \right| > \frac{m \epsilon_1}{ 2\sqrt{\Psi'(\theta)}}  \right) = 2 \left[1 - \Phi\left(\frac{m \epsilon_1}{2 \sqrt{\Psi'(\theta)(N-n)}} \right)\right] \leq c_0 \exp \left( - \frac{m^2 \epsilon_1^2 }{8\Psi'(\theta)  (N-n)} \right),
\end{equation}
where in the last inequality we used Lemma \ref{LemmaI1} and that $2 \leq \sqrt{2\pi}$. Also $c_0$ is as in Lemma \ref{LemmaI1}.

By the triangle inequality, our definition $b_0 = |\Psi(\theta)| + 2 + b_1$ and the fact that $N - n \leq \epsilon_1 m$,
\begin{equation*}
\begin{split}
& \mathbb{P}\left( - \sum_{ j = n+1}^N \log w_{m,j} \geq m   (b_0 - b_1 -  1)  \epsilon_1  \right) =   \\
& = \mathbb{P}_1\left( \sum_{ i = 1}^{N-n} X_i  \geq \frac{m   (b_0 - b_1 -  1)  \epsilon_1   -  (N-n)\Psi(\theta)  }{\sqrt{\Psi'(\theta)}} \right) \leq \mathbb{P}_1\left( \left| \sum_{ i = 1}^{N-n} X_i \right|  \geq \frac{m  \epsilon_1 }{\sqrt{\Psi'(\theta)}} \right).
\end{split}
\end{equation*}
The last inequality and the triangle inequality imply that for $\frac{m  \epsilon_1 }{2\sqrt{\Psi'(\theta)}}> C \log (N- n) $ we have
\begin{equation*}
\begin{split}
& \mathbb{P}\left( - \sum_{ j = n+1}^N \log w_{m,j} \geq m   (b_0 - b_1 -  1)  \epsilon_1 \right) \leq  \mathbb{P}_1\left( \left|\sum_{i = 1}^{N-n} X_i - \sum_{i = 1}^{N-n} Y_i \right| \geq  \frac{m  \epsilon_1 }{2\sqrt{\Psi'(\theta)}} \right)  + \\
& \mathbb{P}_1\left( \left| \sum_{ i = 1}^{N-n} Y_i \right|  \geq \frac{m  \epsilon_1 }{2\sqrt{\Psi'(\theta)}} \right) \leq K \exp \left(- \frac{m  \epsilon_1 }{2\sqrt{\Psi'(\theta)}} +  C \log (N- n) \right)+  c_0 \exp \left( - \frac{m^2 \epsilon_1^2 }{8\Psi'(\theta)  (N-n)} \right),
\end{split}
\end{equation*}
where in the last inequality we used (\ref{S2S2R5}) and (\ref{S2S2R6}). Since $N- n \leq \epsilon_1 m$ we see that the last inequality implies (\ref{S2S2R3}) and hence the lemma.
\end{proof}

\begin{lemma}\label{S2S2L2} Let $\theta >0$ and $\theta_c$ be as in (\ref{DefThetaC}). If $\theta \in (0, \theta_c]$ and $a > 0$ there exist $M_1, D_1 > 0$ (depending on $\theta$ and $a$) such that for all $M \geq M_1$, $M \geq N \geq 1$ and $m \in \llbracket 1, M \rrbracket$, $n \in \llbracket 1, N \rrbracket$
\begin{equation}\label{UBPP1}
\mathbb{P} \left( \log Z(1,1; m, n) \geq -  \Psi(\theta/2) \cdot (M+N)  + D_1 M^{1/3} (\log M)^{2/3}\right) \leq M^{-a}.
\end{equation}
If $\theta > \theta_c$ and $a > 0$ there exist $M_2, D_2 > 0$ (depending on $\theta$ and $a$) such that for all $M \geq M_2$, $M \geq N \geq 1$ and $m \in \llbracket 1, M \rrbracket$, $n \in \llbracket 1, N \rrbracket$
\begin{equation}\label{UBPP2}
\mathbb{P} \Big( \log Z(1,1; m, n) \geq  D_2  \log M \Big) \leq M^{-a}.
\end{equation}
\end{lemma}
\begin{proof}
For clarity we split the proof into four steps.

\smallskip
{\bf \raggedleft Step 1.} In this step we introduce a bit of notation that will be used throughout the proof and fix the constants $D_1, D_2$ in the statement of the lemma.

In the remainder of the proof we fix $\epsilon > 0$ to be sufficiently small (depending on $\theta$) so that (\ref{S2H1}) holds. Since $\sigma_{\theta}(x)$ is smooth on $(0, \infty)$ we see that there is a constant $d_{\epsilon} > 0$ such that
\begin{equation}\label{WER2V3}
\sigma_{\theta}(x) \leq d_{\epsilon} \mbox{ for all $x \in [\epsilon, 1]$. }
\end{equation}
By Proposition \ref{S1LDE1} (applied to $\theta$ as in the statement of the lemma and $\delta = \epsilon$ as above) we obtain for any $m,n \in \mathbb{N}$ with $ \frac{\min(m,n)}{\max(m,n)} \in [\epsilon, 1]$ and $y \geq 0$ that
\begin{equation}\label{S2H2}
\begin{split}
&\mathbb{P} \left( \log Z(1, 1; m, n) \geq - \max(m,n) h_{\theta} \left( \frac{\min(m,n)}{\max(m,n)} \right) + \max(m,n)^{1/3} \sigma_{\theta} \left( \frac{\min(m,n)}{\max(m,n)} \right) \cdot y \right) \\
& \leq C_1 e^{-c_1 \max(m,n) } + C_2 e^{-c_2 y^{3/2}},
\end{split}
\end{equation}
for some constants $C_1, C_2, c_1, c_2 > 0$. In deriving (\ref{S2H2}) we also used the distributional equality of $Z(1, 1; m, n)$ and $Z(1, 1; n, m)$.

We let $B, b > 0$ be as in Lemma \ref{S2S2L1} for the above choice of $\epsilon$ and $\theta$ as in the statement of the present lemma. We let $R_{\epsilon} > 0$ be sufficiently large (depending on $a, \theta, \epsilon$) so the following hold
\begin{equation}\label{Reps}
\begin{split}
&  b \cdot R_{\epsilon} \geq a + 1, \hspace{3mm}  c_1 \cdot R_{\epsilon} \geq  a + 1, \\
& \mbox{if $\theta > \theta_c$ we also assume $R_{\epsilon}$ is sufficiently large so that $c_2  d_{\epsilon}^{-3/2}  R_{\epsilon}  \Psi(\theta/2)^{3/2} \geq a + 1$.}
\end{split}
\end{equation}
We recall that $\Psi(\theta/2) > 0$ for $\theta > \theta_c$, cf. (\ref{DefThetaC}), so that the choice of $R_{\epsilon}$ in the second line of (\ref{Reps}) is indeed possible.

We finally specify the constants $D_1, D_2$ in the statement of the lemma. If $d_{\epsilon}$ is such that (\ref{WER2V3}) holds and $c_1, c_2$ are such that (\ref{S2H2}) holds we let $D_1, D_2 > 0$ be sufficiently large so that
\begin{equation}\label{Dconst}
c_2 D^{3/2}_1 d_{\epsilon}^{-3/2} \geq a + 1 \mbox{ if $\theta \in (0, \theta_c]$ and } D_2 - {\bf 1} \{ \theta \leq 3 \} \cdot 2 R_{\epsilon} \cdot \log \left( \frac{2}{\theta - 1} \right)  \geq a + 1 \mbox{ if $\theta > \theta_c$}.
\end{equation}

\smallskip
{\bf \raggedleft Step 2.} In this and the next step we prove (\ref{UBPP1}) and (\ref{UBPP2}) when $\max(m,n) \geq R_{\epsilon} \log M$. The case when $\max(m,n) \leq R_{\epsilon} \log M$ is handled in Step 4. We split the proof into the cases when $\frac{\min(m,n)}{\max(m,n)} \leq \epsilon$ and $\frac{\min(m,n)}{\max(m,n)} \in [\epsilon,1]$, with the former handled in this step and the latter in the next.

From the distributional equality of $Z(1, 1; m, n)$ and $Z(1, 1; n, m)$ and Lemma \ref{S2S2L1} we see that
$$\mathbb{P}\left( \log Z(1, 1; n, m) \geq \max(m,n) \cdot \epsilon  - \max(m,n) \cdot h_{\theta}\left( \frac{\min(m,n)}{\max(m,n)}\right) \right) \leq B \exp \left( - b \max(m,n) \right).$$
The latter inequality, (\ref{S2H1}) and $b \max(m,n) \geq b R_{\epsilon} \log M \geq (a+1) \log M$ (see (\ref{Reps})) imply
$$\mathbb{P}\Big( \log Z(1, 1; m, n) \geq  - \Psi(\theta/2) \max(m,n) \Big) \leq B M^{-a - 1}.$$
The latter inequality implies (\ref{UBPP1}) for any $D_1 > 0$ (and thus in particular $D_1$ as in (\ref{Dconst})) since when $\theta \in (0, \theta_c]$ we have $\Psi(\theta/2) \leq 0$, cf. (\ref{DefThetaC}), and $\max(m,n) \leq M + N$. The latter inequality also implies (\ref{UBPP2}) for any $D_2 > 0$ (and thus in particular $D_2$ as in (\ref{Dconst})) since when $\theta > \theta_c$ we have $\Psi(\theta/2) > 0$, cf. (\ref{DefThetaC}).

\smallskip
{\bf \raggedleft Step 3.} In this step we prove (\ref{UBPP1}) and (\ref{UBPP2}) when $\max(m,n) \geq R_{\epsilon} \log M$ and $\frac{\min(m,n)}{\max(m,n)} \in [\epsilon,1]$.
As shown in Section \ref{Section2.1}, $h'_{\theta}(x)$ is strictly increasing on $(0, 1]$, $g_{\theta}^{-1}(1) = \theta/2$ and $h_{\theta}'(1) = \Psi(\theta/2)$ so  $h'_{\theta}(x) \leq \Psi(\theta/2)$ for all $x \in [\epsilon,1]$. This and the mean value theorem imply that for some $c \in [\epsilon, 1]$
\begin{equation}\label{WER1V3}
- h_{\theta}(x) = -h_{\theta}(1) + (1- x) h_{\theta}'(c) \leq  -h_{\theta}(1) + (1- x)\Psi(\theta/2) = - (1+x)\Psi(\theta/2).
\end{equation}
Furthermore, combining (\ref{WER2V3}), (\ref{S2H2}) and (\ref{WER1V3}) we conclude that for all $y \geq 0$
\begin{equation}\label{S2H3}
\mathbb{P} \left( \log Z(1, 1; m, n) \geq - (m+n)\Psi(\theta/2)  + \max(m,n)^{1/3} d_{\epsilon} y \right) \leq C_1 e^{-c_1 \max(m,n)} + C_2 e^{-c_2 y^{3/2}}.
\end{equation}

We now use (\ref{S2H3}) to prove (\ref{UBPP1}) and (\ref{UBPP2}) by considering the cases $\theta \in (0, \theta_c]$ and $\theta > \theta_c$ separately. Suppose that $\theta \in (0, \theta_c]$ and put
\begin{equation}\label{S2H4}
y = \frac{ (m+ n - M - N )\Psi(\theta/2)  + D_1   M^{1/3} (\log M)^{2/3}  }{d_{\epsilon} \max(m,n)^{1/3} }.
\end{equation}
Note that if $\theta \in (0, \theta_c]$ we have $y \geq D_1 d_{\epsilon} ^{-1} (\log M)^{2/3}$ since $\max(m,n)^{1/3} \leq M^{1/3}$, $1 \leq m \leq M$, $1 \leq n \leq N$ and $\Psi(\theta/2) \leq 0$.
Plugging (\ref{S2H4}) into (\ref{S2H3}) and using the fact that $c_1 \max(m,n) \geq c_1 R_{\epsilon} \cdot \log M \geq (a+1) \log M$ (see (\ref{Reps})), while $y \geq  D_1 d_{\epsilon} ^{-1} (\log M)^{2/3}$ we conclude
$$\mathbb{P} \left( \log Z(1, 1; m, n) \geq   D_1 M^{1/3} (\log M)^{2/3} \right) \leq C_1 M^{-a-1} + C_2 e^{-c_2 D_1^{3/2} d_{\epsilon}^{-3/2} \log M},$$
which implies (\ref{UBPP1}), since by (\ref{Dconst}) we have $c_2 D_1^{3/2} d_{\epsilon}^{-3/2} \geq a + 1$.

\smallskip
Suppose now that $\theta > \theta_c$. Put
\begin{equation}\label{S2H5}
y = \frac{ (m+ n)\Psi(\theta/2)  + D_2 \log M  }{d_{\epsilon}\max(m,n)^{1/3}}.
\end{equation}
Note that if $\theta > \theta_c$ we have from  (\ref{DefThetaC}) that $\Psi(\theta/2)   \geq 0$ and so $y \geq d_{\epsilon}^{-1}\max(m,n)^{2/3} \Psi(\theta/2)$. Plugging (\ref{S2H5}) into (\ref{S2H3}) and using the fact that $c_1 \max(m,n) \geq c_1 R_{\epsilon} \cdot \log M \geq (a+1) \log M$ (see (\ref{Reps})), while $y \geq d_{\epsilon}^{-1} \max(m,n)^{2/3} \Psi(\theta/2) \geq d_{\epsilon}^{-1}R^{2/3}_{\epsilon} (\log M)^{2/3} \cdot  \Psi(\theta/2) $ we conclude
$$\mathbb{P} \left( \log Z(1, 1; m, n) \geq   D_2 \log M \right) \leq C_1 M^{-a - 1} + C_2 e^{-c_2 d_{\epsilon}^{-3/2} R_{\epsilon} \cdot \Psi(\theta/2)^{3/2} \cdot \log M}.$$
The latter implies (\ref{UBPP2}), since by (\ref{Reps}) we have $c_2 \cdot d_{\epsilon}^{-3/2} \cdot R_{\epsilon} \cdot \Psi(\theta/2)^{3/2} \geq a + 1 $.

\smallskip
{\bf \raggedleft Step 4.} In this final step we prove (\ref{UBPP1}) and (\ref{UBPP2}) when $\max(m,n) \leq R_{\epsilon} \log M$. As before we consider the cases $\theta \in (0, \theta_c]$ and $\theta > \theta_c$ separately.

\smallskip
Suppose that $\theta \in (0, \theta_c]$. From Lemma \ref{TL1}, applied to $n = M^2$, there exists $C_3 > 0$ such that
\begin{equation}\label{CM2}
\mathbb{P} \left( \max_{1 \leq i,j \leq M} \log w_{i,j} \geq   \theta^{-1} (3 + a)  \log M  \right) \leq C_3 M^{-a - 1},
\end{equation}
where $w_{i,j}$ are the i.i.d. random variables with density (\ref{S1invGammaDens}) underlying the polymer model.

Let $B_M = \{ \max_{1 \leq i,j \leq M} \log w_{i,j} <  \theta^{-1}(3+ a)\log M \}$. By the definition of $Z(1, 1; m, n)$, on $B_M$
\begin{equation*}
\begin{split}
&Z(1, 1; m, n) = \sum_{ \pi \in\Pi(1,1;m,n)} \hspace{-5mm} w(\pi) \leq \sum_{ \pi \in\Pi(1,1;m,n)} \hspace{-5mm}  \exp \left(   ( m + n - 1)  \theta^{-1} (3+a) \log M \right) \\
& \leq 2^{ m + n }\exp \left( ( m + n - 1 ) \cdot \theta^{-1} (3+a)   \log M \right) < \exp \left( M^{1/4} \right),
 \end{split}
\end{equation*}
where in going from the first to the second line we used that $|\Pi(1,1;m,n)| \leq  2^{ m + n }$, and the last inequality holds for all large enough $M$ (here we used our assumption $\max(m,n) \leq R_{\epsilon} \log M$).

Combining the last two equations we see that
$$\mathbb{P} \left( \log Z (1, 1; m, n)  \geq   M^{1/4} \right) \leq C_3 M^{-a - 1}.$$
The last inequality implies (\ref{UBPP1}) with any fixed $D_1 > 0$ (and thus in particular $D_1$ as in (\ref{Dconst})) since $\Psi(\theta/2) \leq 0$ when $\theta \in (0, \theta_c]$.

\smallskip
Suppose that $\theta > \theta_c$. From \cite[page 259]{AS65} we have that $\theta_c \approx 2.92326$. By the definition of $Z(1, 1; m, n)$ and the fact that $\mathbb{E}[w_{1,1}] = (\theta - 1)^{-1}$ we observe that
\begin{equation}\label{S2ExpBound}
\begin{split}
&\mathbb{E} \left[ Z(1, 1; m, n) \right] = \sum_{ \pi \in\Pi(1,1;m,n) } \prod_{(i,j) \in \pi}\mathbb{E} \left[  w_{i,j} \right] = \sum_{ \pi \in\Pi(1,1;m,n)} \frac{1}{(\theta - 1)^{m+n - 1}} \leq \\
& \left(\frac{2}{\theta - 1} \right)^{m+n - 1} \leq \exp \left( {\bf 1} \{\theta \leq 3 \} \cdot 2 R_{\epsilon} \log M \cdot \log \left(\frac{2}{\theta - 1} \right) \right) ,
 \end{split}
\end{equation}
where in going from the first to the second line we used that $|\Pi(1,1;m,n)| \leq  2^{ m + n-1 }$, and in the last inequality we used that $\max(m,n) \leq R_{\epsilon} \log M$.
This and Chebyshev's inequality imply
$$\mathbb{P} \left( \log Z(1, 1; m, n) \geq D_2 \log M \right) \leq \exp \left( {\bf 1} \{\theta \leq 3 \} \cdot 2 R_{\epsilon} \log M \cdot \log \left(\frac{2}{\theta - 1} \right) - D_2 \log M\right) \leq M^{-a-1},$$
where in the last inequality we used that $D_2 - {\bf 1} \{\theta \leq 3 \} \cdot 2 R_{\epsilon} \cdot \log \left(\frac{2}{\theta - 1} \right)  \geq a + 1$, cf. (\ref{Dconst}). The last inequality implies (\ref{UBPP2}). This concludes the proof of the lemma.
\end{proof}

%
\section{Proof of Theorem \ref{PM1}} \label{Section3} In this section we will assume $\theta \in (0, \theta_c)$, where $\theta_c$ is as in (\ref{DefThetaC}). For such values of $\theta$ we construct a sequence of random variables $\MY_N$, which dominate with high probability $\Fe_N$, as in (\ref{Fmax}), and converge to the GUE Tracy-Widom distribution under suitable shifts and scales. The precise result is given as Proposition \ref{MainSub} in Section \ref{Section3.1}. In the same section we  prove Theorem \ref{PM1} by combining Proposition \ref{MainSub} and the complementary bound $\Fe_N\geq \log Z(1,1;N,N)$ along with Proposition \ref{LGPCT}. The construction of $\MY_N$ is accomplished in several steps, which are the content of Sections \ref{Section3.2}, \ref{Section3.3} and \ref{Section3.4}. We continue with the same notation as in Sections \ref{Section1} and \ref{Section2}.

%
\subsection{Upper bounds for $\Fe_N$} \label{Section3.1} In this section we prove Theorem \ref{PM1} using the following proposition, whose proof is postponed until Section \ref{Section3.4}.

\begin{proposition}\label{MainSub} Suppose that $\theta \in (0, \theta_c)$, where $\theta_c$ is as in (\ref{DefThetaC}). There exists a sequence of random variables $\MY_N$ and events $E_N$ such that
\begin{equation}\label{S3MainEq}
\begin{split}
& \Fe_N \leq \MY_N \mbox{ a.s. on $E_N$ for all large enough $N$, $\lim_{N \rightarrow \infty} \mathbb{P}(E_N) = 1$,  and }\\
& \lim_{N \rightarrow \infty} \mathbb{P} \left( \frac{\MY_N + 2 N \Psi(\theta/2)}{\sigma_{\theta} N^{1/3}} \leq y\right) = F_{\rm GUE}(y) \mbox{ for any $y \in \mathbb{R}$}.
\end{split}
\end{equation}
\end{proposition}
The construction of suitable random variables $\MY_N$ and events $E_N$ is performed in three steps, which we briefly explain here.
\begin{figure}[h]
\scalebox{0.50}{\includegraphics{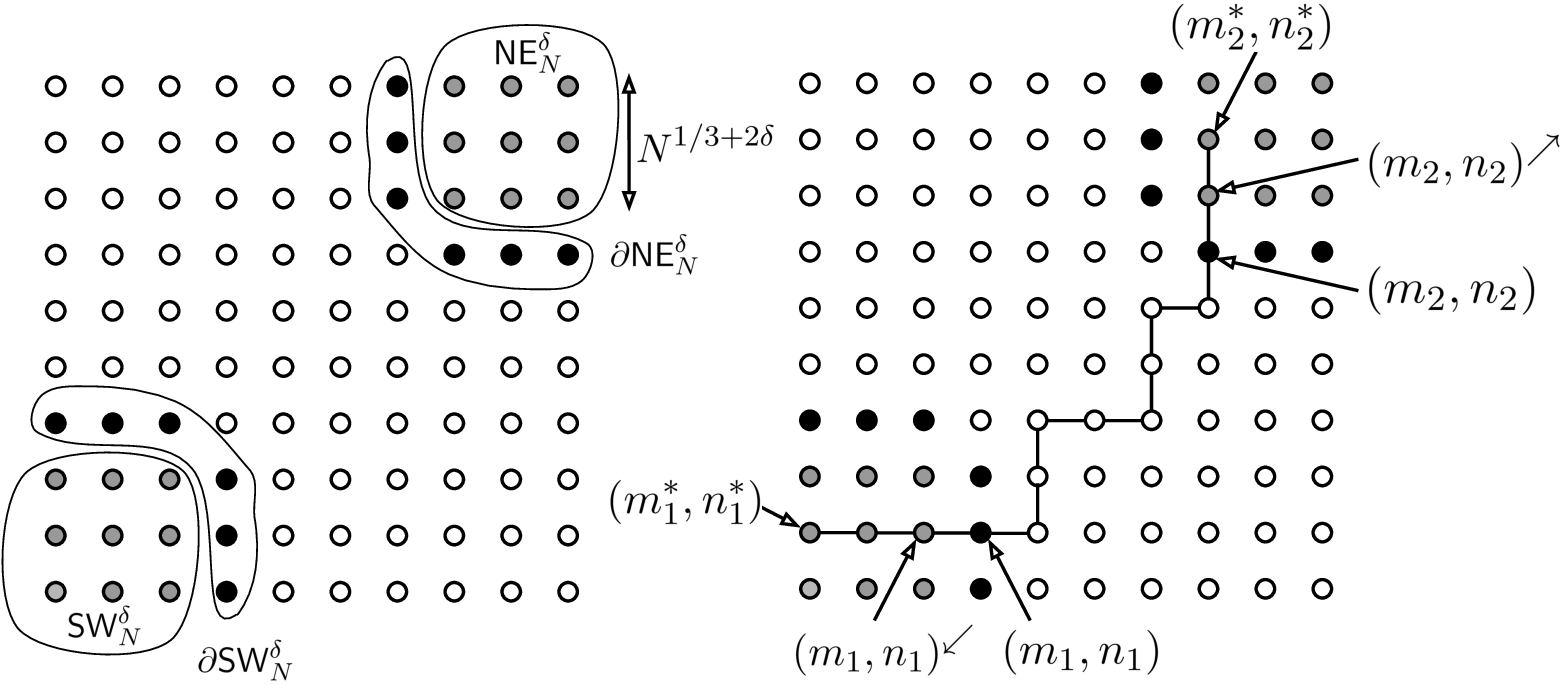}}
\captionsetup{width=\linewidth}
 \caption{The picture on the left depicts the sets $\mathsf{SW}_N^{\delta}$ and $\mathsf{NE}_N^{\delta}$ (these are the grey dots) and the sets $\partial \mathsf{SW}_N^{\delta}$ and $\partial \mathsf{NE}_N^{\delta}$ (these are the black dots). The picture on the right depicts the points $(m^*_1, n^*_1)$ and $(m^*_2, n_2^*)$ that maximize $\log Z(m_1, n_1; m_2,n_2)$ as well as an up-right path $\pi$ that connects these two points. If $(m_1^*, n_1^*) \in  \mathsf{SW}_N^{\delta}$ and $(m_2^*, n_2^*) \in  \mathsf{NE}_N^{\delta}$ the path $\pi$ necessarily passes through $\partial \mathsf{SW}_N^{\delta}$ and $\partial \mathsf{NE}_N^{\delta}$. The point $(m_1, n_1)$ is the first point in $\partial \mathsf{SW}_N^{\delta}$ visited by $\pi$ and $(m_1,n_1)^{\swarrow}$ is the closest point in $\mathsf{SW}_N^{\delta}$ to $(m_1,n_1)$ (it is also visited by $\pi$). The point $(m_2, n_2)$ is the last point in $\partial \mathsf{NE}_N^{\delta}$ visited by $\pi$ and $(m_2,n_2)^{\nearrow}$ is the closest point in $\mathsf{NE}_N^{\delta}$ to $(m_2,n_2)$ (it is also visited by $\pi$).}
\label{S3_1}
\end{figure}

The first step is to show that if $(m^*_1,n^*_1) \leq (m^*_2, n^*_2)$ are maximizers of $Z(m_1, n_1; m_2,n_2)$ then
with high probability $(m^*_1,n^*_1) \in \mathsf{SW}_N^{\delta} := \mathbb{N}^2 \cap[1, N^{1/3 + 2\delta}]^2 $ (the {\em south-west corner} of size $N^{1/3 + 2\delta}$) and $ (m^*_2, n^*_2) \in \mathsf{NE}_N^{\delta} :=  \mathbb{N}^2 \cap[N - N^{1/3 + 2\delta}, N]^2$ (the {\em north-east corner} of size $N^{1/3 + 2\delta}$) for any $\delta \in (0, 1/3)$, see Figure \ref{S3_1}. This step is accomplished in Section \ref{Section3.2}, see Lemma \ref{LCorners}.

The key idea in the proof is to use Lemma \ref{S2S2L2}, which shows that with high probability
$$\log Z(m_1, n_1; m_2,n_2) \leq -2N \Psi(\theta/2) - N^{1/3} \log N$$
when $(m_1, n_1) \not \in  \mathsf{SW}_N^{\delta}$ or $(m_2, n_2) \not \in  \mathsf{NE}_N^{\delta}$. By Proposition \ref{LGPCT} we have that with high probability
$$\Fe_N \geq \log Z(1, 1; N,N) =-2N \Psi(\theta/2) + O(N^{1/3}).$$
The latter two statements imply with high probability $\log Z(m_1, n_1; m_2,n_2)$ cannot compete with $\Fe_N$ for $(m_1, n_1), (m_2,n_2)$ away from the corners, and so $(m^*_1,n^*_1) \in \mathsf{SW}_N^{\delta}$, while $ (m^*_2, n^*_2) \in \mathsf{NE}_N^{\delta}$.

\smallskip
The second step is to show that with high probability
\begin{equation}\label{IS1}
\Fe_N \leq r_N + \max_{\substack{(m_1, n_1) \in \partial \mathsf{SW}_N^{\delta} \\ (m_2,n_2)\in \partial \mathsf{NE}_N^{\delta}}} \hspace{-5mm} \log Z(m_1, n_1; m_2,n_2) + \Psi(\theta/2) \cdot (-m_1 - n_1 + m_2 + n_2 - 2N),
\end{equation}
where $r_N= 2N^{1/9 + 2\delta /3 } \log(N)$ and $\delta \in (0,1/3)$ (all we really need is that $r_N= o(N^{1/3})$ and $r_N$ grows faster than $N^{1/9+ 2\delta /3}$, which is the $1/3$ power of the side-lengths of $\mathsf{NE}_N^{\delta}$ and $\mathsf{SW}_N^{\delta}$), and where the maximum is over $(m_1, n_1) \in \partial \mathsf{SW}_N^{\delta}$ and $(m_2, n_2) \in \partial \mathsf{NE}_N^{\delta}$, which we call the {\em frames} of the corners $\mathsf{SW}_N^{\delta}$ and $\mathsf{NE}_N^{\delta}$. This step is accomplished in Section \ref{Section3.3}, see Lemma \ref{LFrames}. The frames $\partial \mathsf{SW}_N^{\delta}, \partial \mathsf{NE}_N^{\delta}$ are defined in (\ref{DefStrips}), see also Figure \ref{S3_1}.

To prove (\ref{IS1}) we use our first step above (i.e. Lemma \ref{LCorners}), which says that if $(m^*_1,n^*_1) \leq (m^*_2, n^*_2)$ are the maximizers of $Z(m_1, n_1; m_2,n_2)$, then with high probability $(m^*_1,n^*_1) \in \mathsf{SW}_N^{\delta}$ and $ (m^*_2, n^*_2) \in \mathsf{NE}_N^{\delta}$. In addition, from the definition of $Z(m^*_1, n^*_1; m^*_2, n^*_2) $ we have
$$Z(m^*_1, n^*_1; m^*_2, n^*_2) = \sum\nolimits_{ \pi \in\Pi(m^*_1, n^*_1; m^*_2, n^*_2)} \prod_{(i,j) \in \pi} w_{i,j}.$$
If we split the above sum over the first site in $\partial \mathsf{SW}_N^{\delta}$ and the last site in $\partial \mathsf{NE}_N^{\delta}$ that the path $\pi$ passes through we obtain
\begin{equation}\label{IS2}
\begin{split}
Z(m^*_1, n^*_1; m^*_2, n^*_2) = \sum_{ \substack{(m_1, n_1) \in \partial \mathsf{SW}_N^{\delta} \\ (m_1^*, n_1^*) \leq (m_1, n_1)} } \sum_{ \substack{ (m_2, n_2) \in \partial \mathsf{NE}_N^{\delta} \\  (m_2, n_2) \leq (m_2^*, n_2^*) }} & Z\left(m^*_1, n^*_1;(m_1,n_1)^{\swarrow}\right) \times \\
&Z(m_1, n_1;m_2, n_2)Z\left((m_2, n_2)^{\nearrow} ;m_2^*, n_2^* \right),
\end{split}
\end{equation}
where $(m_1,n_1)^{\swarrow}$ is the closest vertex in $\mathsf{SW}_N^{\delta}$ to $(m_1, n_1) \in \partial \mathsf{SW}_N^{\delta}$, and $(m_2,n_2)^{\nearrow}$ is the closest vertex in $\mathsf{NE}_N^{\delta}$ to $(m_2, n_2) \in \partial \mathsf{NE}_N^{\delta}$. The points $(m_1,n_1)^{\swarrow}$, $(m_2,n_2)^{\nearrow}$ are depicted in Figure \ref{S3_1} and a precise definition can be found in equations (\ref{S3Closest1}) and (\ref{S3Closest2}). We have the following trivial upper bounds for the terms on the right side of (\ref{IS2})
\begin{equation}\label{IS3}
\begin{split}
&Z\left(m^*_1, n^*_1;(m_1,n_1)^{\swarrow}\right) \leq Z^{max}((m_1,n_1)^{\swarrow}) := \max_{(1,1) \leq (m,n) \leq (m_1,n_1)^{\swarrow}} Z(m,n; (m_1,n_1)^{\swarrow}), \\
&Z\left((m_2, n_2)^{\nearrow} ;m_2^*, n_2^* \right) \leq Z^{max}((m_2,n_2)^{\nearrow}) := \max_{ (m_2,n_2)^{\nearrow} \leq (m,n) \leq (N,N)} Z( (m_2,n_2)^{\nearrow}; m, n).
\end{split}
\end{equation}
In particular, from (\ref{IS2}) and (\ref{IS3}) we obtain
\begin{equation}\label{IS4}
\begin{split}
&Z(m^*_1, n^*_1; m^*_2, n^*_2) \leq | \partial \mathsf{SW}_N^{\delta}| \cdot | \partial \mathsf{NE}_N^{\delta}| \times \\
& \max_{\substack{ (m_1,n_1) \in \partial \mathsf{SW}_N^{\delta} \\ (m_2,n_2) \in \partial \mathsf{NE}_N^{\delta}}} Z^{max}((m_1,n_1)^{\swarrow})  \cdot Z(m_1, n_1;m_2, n_2) \cdot Z^{max}((m_2,n_2)^{\nearrow}).
\end{split}
\end{equation}
Once we have (\ref{IS4}) in place, obtaining (\ref{IS1}) is relatively straightforward, since we can control the right sides of (\ref{IS3}) using Lemma \ref{S2S2L2}. The terms $| \partial \mathsf{SW}_N^{\delta}|$ and $| \partial \mathsf{NE}_N^{\delta}|$ are benign, since upon taking logarithms they become $O(\log N)$ and get absorbed in the error term $2 N^{1/9 + 2\delta/3} \log N$ in (\ref{IS1}).

\smallskip
The benefit of the above two steps is that we have reduced our problem to finding a sequence $\MY_N$ that (under appropriate shifts and scales) converges to $F_{\rm  GUE}$, and such that $\MY_N$ dominates the right side of (\ref{IS1}) with high probability. The setup of having the arguments $(m_1, n_1)$, $(m_2, n_2)$ of $\log Z(m_1, n_1; m_2,n_2)$ vary over essentially $1$-dimensional segments is important for the application of our tightness result, Proposition \ref{ThmTight}, which is the crux of the third step of constructing $\MY_N$. While the first two steps are fairly intuitive, the third step is a bit involved and so we will postpone its discussion to the main argument in Section \ref{Section3.4} after we have introduced more notation.

\smallskip
In the remainder of this section we use Proposition \ref{MainSub} to prove Theorem \ref{PM1}.
\begin{proof}[Proof of Theorem \ref{PM1}] Let $E_N$ and $\MY_N$ be as in the statement of Proposition \ref{MainSub} and $y \in \mathbb{R}$. It follows from (\ref{S3MainEq}) that
\begin{equation*}
\begin{split}
&\liminf_{N \rightarrow \infty} \mathbb{P} \left( \frac{\Fe_N + 2 N \Psi(\theta/2)}{\sigma_{\theta} N^{1/3}} \leq y\right) \geq \liminf_{N \rightarrow \infty} \mathbb{P} \left( E_N \cap \left\{ \frac{\Fe_N + 2 N \Psi(\theta/2)}{\sigma_{\theta} N^{1/3}} \leq y \right\}\right) \geq \\
&\liminf_{N \rightarrow \infty} \mathbb{P} \left(  \hspace{-1mm} E_N \cap \left\{ \frac{\MY_N + 2 N \Psi(\theta/2)}{\sigma_{\theta} N^{1/3}} \leq y \right\}  \hspace{-1mm} \right) \geq \liminf_{N \rightarrow \infty} \mathbb{P} \left(   \frac{\MY_N + 2 N \Psi(\theta/2)}{\sigma_{\theta} N^{1/3}} \leq y \right) - \mathbb{P}(E_N) = F_{\rm GUE}(y).
\end{split}
\end{equation*}
On the other hand, we have almost surely $\Fe_N \geq \log Z(1,1; N, N)$. The latter combined with the fact that $g_{\theta}^{-1}(1) = \theta/2$ and Proposition \ref{LGPCT} imply
$$\limsup_{N \rightarrow \infty} \mathbb{P} \left( \frac{\Fe_N + 2 N \Psi(\theta/2)}{\sigma_{\theta} N^{1/3}} \leq y\right) \leq \limsup_{N \rightarrow \infty} \mathbb{P} \left( \frac{\log Z(1,1; N,N)+ 2 N \Psi(\theta/2)}{\sigma_{\theta} N^{1/3}} \leq y\right) = F_{\rm GUE}(y).$$
The last two equations imply the statement of the theorem.
\end{proof}

%
\subsection{Localization to corners} \label{Section3.2} For two finite sets $K_1, K_2 \subset \mathbb{N}^2$ we let
\begin{equation}\label{S3MaxF}
\Fe (K_1, K_2) := \max_{\substack{(m_1, n_1) \leq (m_2,n_2)\\  (m_1, n_1) \in K_1, (m_2, n_2) \in K_2}}  \log Z(m_1,n_1;m_2,n_2),
\end{equation}
where we set $\Fe (K_1, K_2) = -\infty$ if the set over which the maximum is taken is empty.

The goal of this section is to show that if $(m^*_1,n^*_1) \leq (m^*_2, n^*_2)$ are maximizers of $Z(m_1, n_1; m_2,n_2)$ then $(m^*_1, n^*_1) \in \mathsf{SW}_N^{\delta}$ and $(m^*_2, n^*_2) \in \mathsf{NE}_N^{\delta} $ for any $\delta \in (0, 1/3)$ with very high probability, where
\begin{equation}\label{DefCorners}
\mathsf{SW}_N^{\delta}: = \mathbb{N}^2 \cap[1, N^{1/3 + 2\delta}]^2 \mbox{ and }\mathsf{NE}_N^{\delta} :=  \mathbb{N}^2 \cap[N - N^{1/3 + 2\delta}, N]^2.
\end{equation}
The precise statement is detailed in the following lemma.
\begin{lemma}\label{LCorners}
 Suppose that $\theta \in (0, \theta_c)$, where $\theta_c$ is as in (\ref{DefThetaC}) and fix $\delta \in (0, 1/3)$. Then
\begin{equation}\label{S3CornEq}
\lim_{N \rightarrow \infty} \mathbb{P} \left( \Fe_N = \Fe (\mathsf{SW}_N^{\delta}, \mathsf{NE}_N^{\delta}) \right) = 1.
\end{equation}
\end{lemma}
\begin{proof} For clarity we split the proof into two steps.

\smallskip
{\bf \raggedleft Step 1.} We claim that there exists a constant $ N_0 > 0$ (depending on $\theta$ and $\delta$) such that for all $N \geq N_0$, $m_1, n_1, m_2, n_2   \in \llbracket 1, N \rrbracket$ with $(m_1, n_1) \leq (m_2, n_2)$ and $((m_1, n_1), (m_2, n_2)) \not \in \mathsf{SW}_N^{\delta} \times \mathsf{NE}_N^{\delta}$
\begin{equation}\label{LC1}
\mathbb{P} \left(\log Z (m_1, n_1; m_2, n_2) \geq -2N \Psi(\theta/2) -  N^{1/3} \log N \right) \leq N^{-5}.
\end{equation}
We will establish (\ref{LC1}) in Step 2 below. Here we assume it and conclude the proof of the lemma.

\smallskip
Let $A_N$ be the event that $\Fe_N = \log Z(m_1,n_1;m_2,n_2)$ for some $m_1,n_1,m_2,n_2 \in \llbracket 1, N \rrbracket$ such that $(m_1,n_1) \leq (m_2,n_2)$ and $((m_1,n_1), (m_2,n_2)) \not \in \mathsf{SW}_N^{\delta} \times \mathsf{NE}_N^{\delta}$. By taking union bound of (\ref{LC1}) (note that there are at most $N^4$ choices for the quadruple $m_1, n_1, m_2, n_2$) we conclude that
$$\mathbb{P} \left( A_N \cap \{ \Fe_N \geq - 2 N \Psi(\theta/2) - N^{1/3} \log N \}  \right) \leq N^{-1}.$$
In addition, we know that $\Fe_N \geq  \log Z(1,1; N,N)$. The latter, combined with the fact that $h_{\theta}(1) =2 \Psi(\theta/2)$ and Proposition \ref{LGPCT}, implies
\begin{equation*}
\lim_{N \rightarrow \infty} \mathbb{P} \left( \Fe_N \geq - 2 N \Psi(\theta/2) - N^{1/3} \log N\right) = 1.
\end{equation*}
The last two equations imply that $\lim_{N \rightarrow \infty} \mathbb{P}(A_N) = 0$, which in turn implies (\ref{S3CornEq}).

\smallskip
{\bf \raggedleft Step 2.} In this step we prove (\ref{LC1}). Since $ Z (m_1, n_1; m_2, n_2)$ has the same distribution as $Z (n_1, m_1; n_2, m_2)$ we may assume that $m_2 - m_1 \geq n_2 - n_1$. In addition, notice that $ Z (m_1 , n_1 ; m_2 , n_2 )$ has the same distribution as $Z (1, 1; m_2 - m_1 + 1, n_2 - n_1 + 1)$. Finally, observe that since we have assumed that $((m_1, n_1), (m_2, n_2)) \not \in \mathsf{SW}_N^{\delta} \times \mathsf{NE}_N^{\delta}$ we have $N \geq m_2 - m_1 + 1 \geq n_2 - n_1 + 1 \geq 1$, and $m_2 - m_1 + 1  \leq N- k + 1$ where $k = \lfloor N^{1/3 + 2 \delta} \rfloor + 1$.

The latter observations and Lemma \ref{S2S2L2} (applied to $a = 5$, $M = N$ and $N = N -k + 1$) together imply that for all large $N$
\begin{equation*}
\begin{split}
&\mathbb{P} \left(\log Z (m_1, n_1; m_2, n_2) \geq - \Psi (\theta/2) \cdot (2 N - k +1) + N^{1/3} \log N \right)  \leq  N^{-5}.
\end{split}
\end{equation*}
The latter inequality implies (\ref{LC1}) since $k = \lfloor N^{1/3 + 2 \delta} \rfloor + 1$ and $\Psi (\theta/2) < 0$.
\end{proof}

%
\subsection{Comparison with the maximum over frames} \label{Section3.3}
For a fixed $\delta \in (0, 1/3)$ and $N \in \mathbb{N}$ we let $\mathsf{SW}_N^{\delta}$ and $\mathsf{NE}_N^{\delta}$ be as in (\ref{DefCorners}). We define the {\em frames} of these two corners to be (see Figure \ref{S3_1})
\begin{equation}\label{DefStrips}
\begin{split}
&\partial \mathsf{SW}_N^{\delta}  := \{ (m,n) \in \mathbb{N}^2 \setminus \mathsf{SW}_N^{\delta} : (m-1,n) \in \mathsf{SW}_N^{\delta} \mbox{ or } (m,n-1) \in \mathsf{SW}_N^{\delta} \}, \\
&\partial \mathsf{NE}_N^{\delta}  := \{ (m,n) \in \mathbb{N}^2 \setminus \mathsf{NE}_N^{\delta} : (m+1,n) \in \mathsf{NE}_N^{\delta} \mbox{ or } (m,n+1) \in \mathsf{NE}_N^{\delta} \}.
\end{split}
\end{equation}
The main result of the section is the following lemma, which substantiates (\ref{IS1}) from Section \ref{Section3.1}.
\begin{lemma}\label{LFrames}
Under the same notation and assumptions as in Lemma \ref{LCorners} we have
\begin{equation}\label{S3FrameEq}
\lim_{N \rightarrow \infty} \mathbb{P} \left(  \Fe (\mathsf{SW}_N^{\delta}, \mathsf{NE}_N^{\delta}) \leq  r_N+ \hspace{-1mm} \max_{\substack{ (m_1,n_1) \in \partial \mathsf{SW}_N^{\delta} \\ (m_2,n_2) \in \partial \mathsf{NE}_N^{\delta}}} \hspace{-1mm} \log Z(m_1, n_1; m_2, n_2) + \Psi(\theta/2) \cdot D_N    \right) = 1,
\end{equation}
where $r_N = 2N^{1/9 + 2\delta /3 } \log(N)$ and $D_N(m_1, n_1; m_2,n_2) =  - m_1 - n_1 + m_2 + n_2 - 2N$.
\end{lemma}
\begin{proof}
For clarity we split the proof into two steps.

\smallskip
{\bf \raggedleft Step 1.} In this step we introduce some notation, which will be required for our arguments. For simplicity we set $k = \lfloor N^{1/3 + 2 \delta} \rfloor + 1$. Let us define the sets
\begin{equation}\label{DefStrips2}
\begin{split}
&\partial \mathsf{SW}_N^{\delta, right }  := \{ (m,n) \in \partial \mathsf{SW}^{\delta}_N : m = k \}, \hspace{2mm} \partial \mathsf{NE}_N^{\delta, left}  :=  \{ (m,n) \in \partial \mathsf{NE}^{\delta}_N : m = N - k  \}  , \\
&\partial \mathsf{SW}_N^{\delta, top}  := \{ (m,n) \in \partial \mathsf{SW}^{\delta}_N : n = k \}, \hspace{2mm} \partial \mathsf{NE}_N^{\delta, bot}  :=  \{ (m,n) \in \partial \mathsf{NE}^{\delta}_N : n = N - k  \},
\end{split}
\end{equation}
where $\partial \mathsf{SW}_N^{\delta}$, $\partial \mathsf{NE}_N^{\delta}$ are as in equation (\ref{DefStrips}). Note that $\partial \mathsf{SW}_N^{\delta} = \partial \mathsf{SW}_N^{\delta, right } \sqcup \partial \mathsf{SW}_N^{\delta, top}$ and $\partial \mathsf{NE}_N^{\delta} = \partial \mathsf{NE}_N^{\delta, left} \sqcup  \partial \mathsf{NE}_N^{\delta, bot}$.
Let $(m^*_1, n^*_1) \in \mathsf{SW}_N^{\delta}$, $(m^*_2, n^*_2) \in \mathsf{NE}_N^{\delta}$ be such that
$$ \Fe (\mathsf{SW}_N^{\delta}, \mathsf{NE}_N^{\delta}) = \log Z(m^*_1, n^*_1; m^*_2, n^*_2),$$
and recall from (\ref{PartitionFunct}) that we have
$$Z(m^*_1, n^*_1; m^*_2, n^*_2) = \sum\nolimits_{ \pi \in\Pi(m^*_1, n^*_1; m^*_2, n^*_2)} \prod_{(i,j) \in \pi} w_{i,j}.$$
Assuming that $N$ is sufficiently large so that $N \geq 2k + 1$ we have that $\mathsf{SW}_N^{\delta}$ and $\mathsf{NE}_N^{\delta}$ are well-separated and so we can split the above sum over the first site in $\partial \mathsf{SW}_N^{\delta}$ and the last site in $\partial \mathsf{NE}_N^{\delta}$ that the path $\pi$ passes through, see Figure \ref{S3_1}. Consequently, we obtain the formula (\ref{IS2}), where the notation $(m_1,n_1)^{\swarrow}$ and $(m_2, n_2)^{\nearrow}$ used therein is explicitly given by
\begin{equation}\label{S3Closest1}
(m_1,n_1)^{\swarrow} = \begin{cases} (m_1 - 1, n_1)  &\mbox{ if }(m_1,n_1) \in \partial \mathsf{SW}_N^{\delta, right },  \\
                                                                      (m_1 , n_1 - 1)  &\mbox{ if }(m_1,n_1) \in \partial \mathsf{SW}_N^{\delta, top }, \end{cases}
\end{equation}
and
\begin{equation}\label{S3Closest2}
(m_2,n_2)^{\nearrow} = \begin{cases} (m_2 + 1, n_2)  &\mbox{ if }(m_2,n_2 ) \in \partial \mathsf{NE}_N^{\delta, left },  \\
                                                                      (m_2 , n_2 + 1)  &\mbox{ if }(m_2,n_2) \in \partial \mathsf{NE}_N^{\delta, bot }. \end{cases}
\end{equation}
For  $\hspace{-0.25mm} (m_1, \hspace{-0.25mm} n_1)\hspace{-1mm} \in \hspace{-1mm}\partial \mathsf{SW}_N^{\delta}$\hspace{-0.25mm} and \hspace{-0.25mm}$(m_2, \hspace{-0.25mm}n_2) \hspace{-1mm} \in \hspace{-1mm}\partial \mathsf{NE}_N^{\delta}$, let $Z^{max}((m_1,n_1)^{\swarrow})$ and $Z^{max}((m_2,n_2)^{\nearrow})$ be as in (\ref{IS3}).

\smallskip
{\bf \raggedleft Step 2.} Note that $Z(m_1, n_1; m_2, n_2)$ has the same distribution as $Z(m_1 + a, n_1 + b; m_2 + a, n_2 + b)$ for any $a,b \in \mathbb{Z}$. The latter implies that for any $ (m_2, n_2) \in \partial \mathsf{NE}_N^{\delta,left}$ and $(m_2,n_2)^{\nearrow} \leq (m,n) \leq (N,N)$ we have that $Z( (m_2,n_2)^{\nearrow}; m, n)$ has the same law as $ Z(1,1; m - m_2, n - n_2 + 1).$
From Lemma \ref{S2S2L2} (applied to $M = k$, $N = N - n_2 + 1$ and $a = 6$) we conclude that for all large $N$,
\begin{equation}\label{S3Zmax0}
\mathbb{P} \left( Z( (m_2,n_2)^{\nearrow}; m, n) \geq  - \Psi (\theta/2) \cdot (k + N -n_2 + 1 ) + k^{1/3} \log k  \right) \leq k^{-6},
\end{equation}
provided that $(m_2, n_2) \in \partial \mathsf{NE}_N^{\delta,left}$ and $(N,N) \geq (m,n) \geq  (m_2,n_2)^{\nearrow}$.

Taking a union bound of (\ref{S3Zmax0}) over $(m,n)$ (note that there are at most $k^2$ possible choices) we get for all large $N$ and any  $(m_2, n_2) \in \partial \mathsf{NE}_N^{\delta,left}$ (notice that $k = N - m_2$)
\begin{equation}\label{S3Zmax1}
\mathbb{P} \left(\log Z^{max}((m_2,n_2)^{\nearrow}) \geq  - \Psi (\theta/2) \cdot (2N - n_2 - m_2 + 1) + k^{1/3} \log k  \right) \leq k^{-4}.
\end{equation}
Using the distributional equality of $Z(m_1, n_1; m_2, n_2)$ and $Z(n_1, m_1; n_2, m_2)$ we conclude that (\ref{S3Zmax1}) holds for $(m_2, n_2) \in \partial \mathsf{NE}_N^{\delta,bot}$ as well. An analogous argument allows us to conclude that for $(m_1, n_1) \in \partial \mathsf{SW}_N^{\delta}$
\begin{equation}\label{S3Zmax2}
\mathbb{P} \left(\log Z^{max}((m_1,n_1)^{\swarrow}) \geq  - \Psi (\theta/2) \cdot (m_1 + n_1 - 1) + k^{1/3} \log k  \right) \leq (k-1)^{-4}.
\end{equation}
We mention that we have $k^{-4}$ in (\ref{S3Zmax1}) and $(k-1)^{-4}$ in (\ref{S3Zmax2}), because $\mathsf{SW}_N^{\delta} = \llbracket 1, k-1\rrbracket \times \llbracket 1, k-1\rrbracket$ (and hence contains $(k-1)^2$ points), while $\mathsf{NE}_N^{\delta} = \llbracket N-k+1, N \rrbracket \times \llbracket N-k+1, N \rrbracket$ (and hence contains $k^2$ points). This discrepancy can be traced to the fact that we are working with the domain $\llbracket 1, N \rrbracket \times \llbracket 1, N \rrbracket$ as opposed to $\llbracket 0, N \rrbracket \times \llbracket 0, N \rrbracket$.

Let $A_N$ denote the event
\begin{equation*}
\begin{split}
A_N:= &\bigcap_{(m_2, n_2) \in \partial \mathsf{NE}_N^{\delta}}  \left\{ \log Z^{max}\left((m_2,n_2)^{\nearrow}\right) <  - \Psi (\theta/2) \cdot (2N - n_2 - m_2 + 1 ) + k^{1/3} \log k \right\} \bigcap  \\
&\bigcap_{(m_1, n_1) \in \partial \mathsf{SW}_N^{\delta}}  \left\{\log Z^{max}\left((m_1,n_1)^{\swarrow}\right) <  - \Psi (\theta/2) \cdot (m_1 + n_1 - 1) + k^{1/3} \log k \right\},
\end{split}
\end{equation*}
 and observe that by a union bound we have $1 - \mathbb{P}(A_N) = O(k^{-2})$ so that $\lim_{N \rightarrow \infty} \mathbb{P}(A_N) = 1$.

We see that on the event $A_N$ we have the following tower of inequalities
\begin{equation*}
\begin{split}
&Z(m^*_1, n^*_1; m^*_2, n^*_2) \leq \hspace{-3mm} \sum_{ \substack{(m_1, n_1) \in \partial \mathsf{SW}_N^{\delta}} }   \sum_{ \substack{ (m_2, n_2) \in \partial \mathsf{NE}_N^{\delta} }} \hspace{-8mm} Z^{max}\left((m_1,n_1)^{\swarrow}\right) Z(m_1, n_1;m_2, n_2)Z^{max}\left((m_2,n_2)^{\nearrow}\right) \leq \\
&4 k^2 \cdot \exp \left( 2k^{1/3} \log k \right) \cdot  \max_{\substack{ (m_1,n_1) \in \partial \mathsf{SW}_N^{\delta} \\ (m_2,n_2) \in \partial \mathsf{NE}_N^{\delta}}}  Z(m_1, n_1;m_2, n_2) \cdot \exp \left( - \Psi(\theta/2) [2N - n_2 - m_2 + m_1 + n_1 ] \right),
\end{split}
\end{equation*}
where in the first inequality we used (\ref{IS2}) and $Z^{max}\left((m_1,n_1)^{\swarrow}\right)  \geq Z\left(m^*_1, n^*_1;(m_1,n_1)^{\swarrow}\right) $, \\ $Z^{max}\left((m_2,n_2)^{\nearrow}\right) \geq Z\left((m_2, n_2)^{\nearrow} ;m_2^*, n_2^* \right)$;  in the second inequality we used the definition of $A_N$ and the fact that $|\partial \mathsf{SW}_N^{\delta}| = 2k-2$, $|\partial \mathsf{NE}_N^{\delta}| = 2k$. The last inequality, the fact that $k = \lfloor N^{1/3 + 2 \delta} \rfloor + 1$ and $\lim_{N \rightarrow \infty} \mathbb{P}(A_N) = 1$ together imply (\ref{S3FrameEq}).
\end{proof}

%
\subsection{Proof of Proposition \ref{MainSub}} \label{Section3.4} In this section we present the proof of Proposition \ref{MainSub}. For clarity we split the proof into three steps.

\smallskip
{\bf \raggedleft Step 1.} In this step we define $\MY_N$ and $E_N$ as in the statement of the proposition.

Set $\delta = 1/12$, $K = \lfloor N^{4/5} \rfloor$ and define the random variables
\begin{equation}\label{S3RVS}
\begin{split}
&\MY_N^{\mathsf{SW}} = \min_{(m,n) \in \partial \mathsf{SW}_N^{\delta}} \log Z(-K,-K; m, n) +  \Psi( \theta/2) (2K+m + n) ,\\
&\MY_N^{\mathsf{NE}} = \min_{(m,n) \in \partial \mathsf{NE}_N^{\delta}} \log Z(m,n; K+N, K+N) +  \Psi( \theta/2) (2K + 2N - m - n),
\end{split}
\end{equation}
where $\partial \mathsf{SW}_N^{\delta}$ and $\partial \mathsf{NE}_N^{\delta}$ are as in (\ref{DefStrips}). We then define
\begin{equation}\label{S3YDef}
\begin{split}
&\MY_N = 2N^{1/4} + \log Z(-K,-K; N+K, N+K) + 4 K \Psi(\theta/2) - \MY_N^{\mathsf{SW}}  -  \MY_N^{\mathsf{NE}} .
\end{split}
\end{equation}
This specifies the variables $\MY_N$. We mention that the $N^{1/4}$ in (\ref{S3YDef}) is a convenient choice -- for this term we could take any $N^{\rho}$ with $1/6 < \rho < 1/3$. The bound $\rho < 1/3$ is dictated from our desire to have this modification become irrelevant in the $N^{1/3}$ scale. The bound $\rho > 1/6$ is dictated from our desire to have $r_N \leq N^{\rho}$, where $r_N = 2N^{1/6} \log(N)$ is as in Lemma \ref{LFrames} for $\delta = 1/12$.

We next define the events $E_N=E_N^1 \cap E_N^2 \cap E_N^{3}$ where
\begin{equation}\label{Event2}
\begin{split}
&E_N^1 = \left\{ \Fe_N = \Fe (\mathsf{SW}_N^{\delta}, \mathsf{NE}_N^{\delta}) \right\}, \hspace{5mm} E_N^2 = \left\{ \max_{1 \leq i,j \leq N}\log w_{i,j}\leq 4 \theta^{-1} \log N  \right\},\\
&E_N^3 = \left\{ \Fe (\mathsf{SW}_N^{\delta}, \mathsf{NE}_N^{\delta}) \leq  N^{1/4} +  \max_{\substack{ (m_1,n_1) \in \partial \mathsf{SW}_N^{\delta} \\ (m_2,n_2) \in \partial \mathsf{NE}_N^{\delta}}} \log Z(m_1, n_1; m_2, n_2) + \Psi(\theta/2) \cdot D_N \right\},
\end{split}
\end{equation}
and we recall from Lemma \ref{LFrames} that $D_N(m_1, n_1; m_2,n_2) =  - m_1 - n_1 + m_2 + n_2 - 2N$.

\smallskip
{\bf \raggedleft Step 2.} Recall from the statement of the proposition that we seek to prove that
\begin{equation}\label{QW1}
\begin{split}
&(I) \hspace{2mm} \Fe_N \leq \MY_N \mbox{ a.s. on $E_N$ for all large enough $N$, }
\qquad (II) \hspace{2mm} \lim_{N \rightarrow \infty} \mathbb{P}(E_N) = 1, \\
&(III) \hspace{2mm} \lim_{N \rightarrow \infty} \mathbb{P} \left( \frac{\MY_N + 2 N \Psi(\theta/2)}{\sigma_{\theta} N^{1/3}} \leq y\right) = F_{\rm GUE}(y) \mbox{ for any $y \in \mathbb{R}$}.
\end{split}
\end{equation}
In this step we prove $(III)$.

From the definition of $\MY_N$ in (\ref{S3YDef}) we have
\begin{equation}\label{S3YNSE}
\begin{split}
\frac{\MY_N + 2 N \Psi(\theta/2)}{\sigma_{\theta} N^{1/3}}  = &\frac{\log Z(-K,-K; N+K, N+K) + (4 K + 2N) \Psi(\theta/2)}{\sigma_{\theta} (N+K)^{1/3}} \cdot \frac{(N+K)^{1/3}}{N^{1/3}} + \\
& \frac{2N^{1/4} - \MY_N^{\mathsf{SW}} - \MY_N^{\mathsf{NE}}}{\sigma_{\theta} N^{1/3}}.
\end{split}
\end{equation}
From Proposition \ref{LGPCT} we have that the first term on the right weakly converges to $F_{\rm GUE}$ as $N \rightarrow \infty$ (here we used that $K =  \lfloor N^{4/5} \rfloor$ so that $N^{-1/3} (N+K)^{1/3} \rightarrow 1$ as $N \rightarrow \infty$). Thus all we need to show is that $N^{-1/3} (\MY_N^{\mathsf{SW}} + \MY_N^{\mathsf{NE}})$ converges weakly to zero.

Focusing on $Y^{\mathsf{SW}}_N$, we note that $\MY^{\mathsf{SW}}_N = \min \left( \MY^{top, \mathsf{SW}}_N, \MY^{right, \mathsf{SW}}_N \right)$, where
$$\MY^{top, \mathsf{SW}}_N = \min_{(m,n) \in \partial \mathsf{SW}_N^{\delta, top }}  \log Z(-K,-K; m, n) +  \Psi( \theta/2) (2K+m + n),$$
$$\MY^{right, \mathsf{SW}}_N = \min_{(m,n) \in \partial \mathsf{SW}_N^{\delta, right}}  \log Z(-K,-K; m, n) +  \Psi( \theta/2) (2K+m + n),$$
and we recall that $\partial \mathsf{SW}_N^{\delta, top }, \partial \mathsf{SW}_N^{\delta, right}$ were defined in (\ref{DefStrips2}).

In words, $\MY^{top, \mathsf{SW}}_N$ is the (properly centered) minimal point-to-point free energy from the point $(-K, -K)$ to $\partial \mathsf{SW}_N^{\delta, top }$. This set of points is drawn on Figure \ref{S3_2} and is a horizontal segment connecting $(1,k)$ and $(k,k)$. Similarly, $\MY^{right, \mathsf{SW}}_N$ is the (properly centered) minimal point-to-point free energy from the point $(-K, -K)$ to $\partial \mathsf{SW}_N^{\delta, right }$.
\begin{figure}[h]
\scalebox{0.5}{\includegraphics{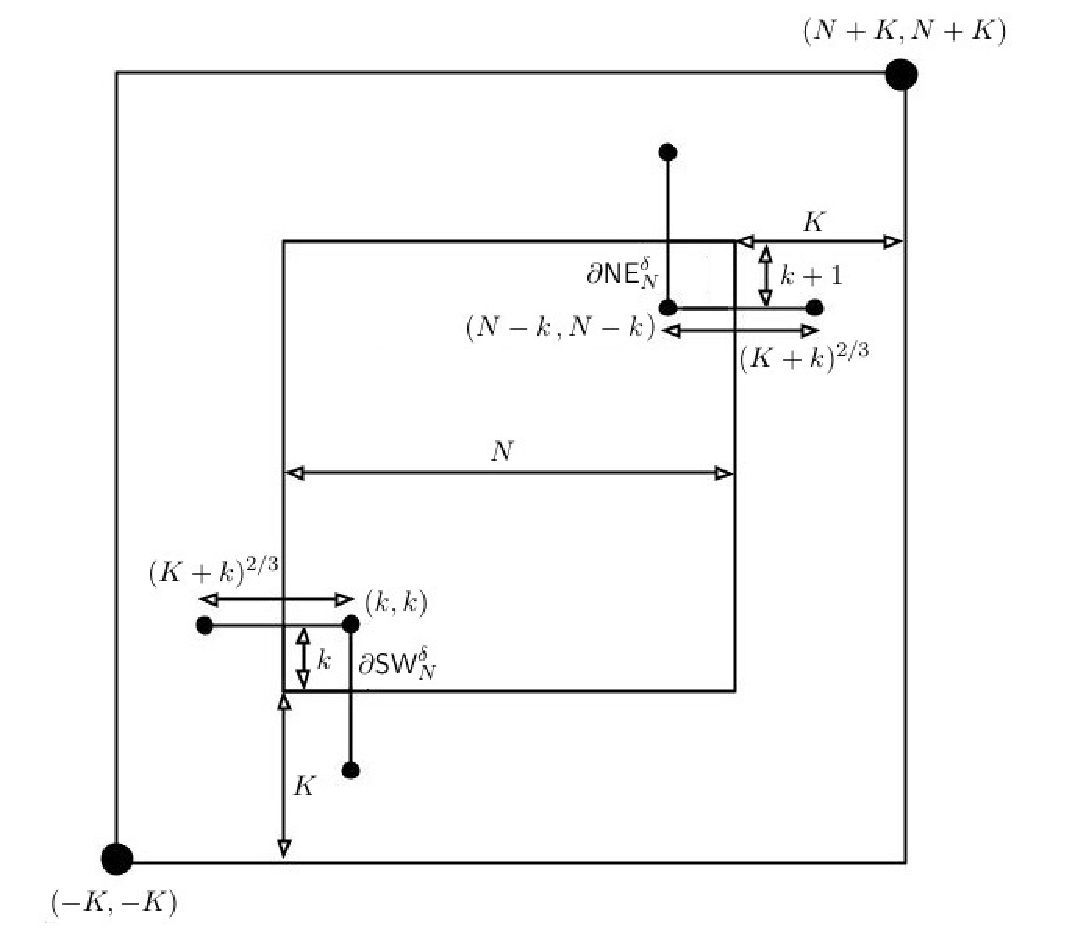}}
\captionsetup{width=\linewidth}
 \caption{The picture depicts the geometric meaning of the quantities $N$, $k = \lfloor N^{1/2} \rfloor + 1$, $K = \lfloor N^{4/5} \rfloor$, and $(K+k)^{2/3}$ that appear in our arguments.}
\label{S3_2}
\end{figure}

Let us set $k = \lfloor N^{1/2} \rfloor + 1$ and for each $x \in [-1 - (K+k)^{-2/3}, 1 + (K+k)^{2/3}]$, such that $x (K+k)^{2/3}$ is an integer, define $s = s(x) = K + k + x(K+k)^{2/3}$ and
$$f_N(x)=(K + k)^{-1/3}\Big(\log  Z(-K,-K; s- K, k)+ h_{\theta}(1)(K + k) + h_\theta'(1) x (K + k)^{2/3}\Big),$$
and then extend $f_N$ to all points $x \in [-1,1]$ by linear interpolation. The function $f_N$ records the (centered and scaled) point-to-point free energy from the point $(-K, -K)$ to the horizontal segment of length $2 (K+k)^{2/3}$ centered at the point $(k,k)$ (half of this segment is drawn in Figure \ref{S3_2}). We have the following relationship between $f_N$ and $\MY^{top, \mathsf{SW}}_N$
$$(K + k)^{-1/3} \MY^{top, \mathsf{SW}}_N = \min_{x \in [-(k-1)(K+k)^{-2/3}, - (K+k)^{-2/3} ]} f_N(x).$$

From Proposition \ref{ThmTight} we know that $f_N$ form a tight $N$-indexed sequence of random continuous functions on $[-1,1]$ (here we also used that $Z(1,1; \cdot , \cdot)$ has the same distribution as $Z(a + 1, b + 1; a+ \cdot , b+ \cdot)$ for any $a, b \in \mathbb{Z}$). Consequently, $(K + k)^{-1/3} \MY^{top, \mathsf{SW}}_N$ also forms a tight $N$-indexed sequence. Here it is important that $K^{2/3}$ (which is the transversal scale in $f_N$) is larger than $k$ so that $k(K+k)^{-2/3} \leq 1$. The latter is true since $K \sim N^{4/5}$, while $k \sim N^{1/2}$. The fact that $(K + k)^{-1/3} \MY^{top, \mathsf{SW}}_N$ is tight implies that $N^{-1/3} \MY^{top, \mathsf{SW}}_N$ weakly converges to zero, since $N^{-1/3} (K + k)^{1/3} \rightarrow 0$. Here it is important that $K/N \rightarrow 0$, which is true as $K \sim N^{4/5}$.

Using that $Z(m_1, n_1; m_2, n_2)$ has the same finite dimensional distribution as $Z(n_1, m_1; n_2, m_2)$ we analogously conclude that $N^{-1/3} \MY^{right, \mathsf{SW}}_N$ weakly converges to $0$ and then so does $N^{-1/3} \MY_N^{\mathsf{SW}}$. Using that $Z(m_1, n_1; m_2, n_2)$ has the same finite dimensional distribution as $Z(-m_2, -n_2; -m_1, -n_1)$ we can repeat the above argument to conclude that $N^{-1/3}\MY_N^{\mathsf{NE}}$ weakly converges to zero. This means that $N^{-1/3} (\MY_N^{\mathsf{SW}} + \MY_N^{\mathsf{NE}})$ converges weakly to zero, which is what was needed to complete $(III)$.

In plain words, the consequence of Proposition \ref{ThmTight} is that both $\MY_N^{\mathsf{SW}}$ and $\MY_N^{\mathsf{NE}}$ are of order $O((K+k)^{1/3})$ and thus do not affect the asymptotics of $\MY_N$ in (\ref{S3YNSE}) because of the $N^{1/3}$ in the denominator. Thus the fluctuations of $\MY_N$ are entirely governed by the first term on the right of (\ref{S3YNSE}), which weakly converges to $F_{\rm GUE}$ as $N \rightarrow \infty$ by Proposition \ref{LGPCT}.

\smallskip
{\bf \raggedleft Step 3.} In this step we prove $(I)$ and $(II)$ in (\ref{QW1}).
From Lemma \ref{LCorners} with $\delta = 1/12$ we have $\lim_{N \rightarrow \infty} \mathbb{P}(E_N^1) = 1$; from Lemma \ref{TL1} (with $a = 1$ and $n = N^2$) we have $\lim_{N \rightarrow \infty} \mathbb{P}(E_N^2) = 1$; and  from Lemma \ref{LFrames} with $\delta = 1/12$ we have $\lim_{N \rightarrow \infty} \mathbb{P}(E_N^3) = 1$. This proves $(II)$.

\smallskip
On the event $E_N$ we know that $E_N^1 \cap E_N^3$ also holds, and thus to prove $(I)$ in (\ref{QW1}) it suffices to prove that for all large enough $N$ we have
\begin{equation}\label{S3RF}
\MY_N \geq N^{1/4} + \max_{\substack{ (m_1,n_1) \in \partial \mathsf{SW}_N^{\delta} \\ (m_2,n_2) \in \partial \mathsf{NE}_N^{\delta}}} \log Z(m_1, n_1; m_2, n_2) + \Psi(\theta/2) \cdot (- m_1 - n_1 + m_2 + n_2 - 2N) .
\end{equation}

Let us fix $(m^*_1,n^*_1) \in \partial \mathsf{SW}_N^{\delta}$ and $(m^*_2,n^*_2) \in \partial \mathsf{NE}_N^{\delta}$ that achieve the maximum in (\ref{S3RF}). By the definition of $Z(-K,-K; N+K, N+K) $ from (\ref{PartitionFunct}) we have
$$Z(-K,-K; N+K, N+K) = \sum\nolimits_{ \pi \in\Pi(-K,-K;N+K,N+K)} \prod_{(i,j) \in \pi} w_{i,j},$$
and by restricting the above sum to paths $\pi$ that pass through $(m^*_1, n^*_1)$ and $(m^*_2, n^*_2)$ we obtain
\begin{equation}\label{ASR}
\begin{split}
& Z(-K,-K; N+K, N+K) \geq Z(-K,-K; m^*_1, n^*_1) \times \\
&Z(m^*_1, n^*_1; m^*_2, n^*_2)  \cdot Z(m^*_2, n^*_2; N+K, N+K) \cdot w_{n^*_1, m^*_1}^{-1} \cdot  w_{n^*_2, m^*_2}^{-1}.
\end{split}
\end{equation}

We now observe the following inequalities, which hold for large enough $N$:
\begin{equation*}
\begin{split}
&N^{1/4} + \log Z(m^*_1, n^*_1; m^*_2, n^*_2) + \Psi(\theta/2) \cdot (- m^*_1 - n^*_1 +  m^*_2 + n^*_2 - 2N) \leq \\
&N^{1/4} + \log Z(-K,-K; N+K, N+K) + \Psi(\theta/2) \cdot (- m^*_1 - n^*_1 +  m^*_2 + n^*_2 - 2N) - \\
& \log  Z(-K, -K; m^*_1, n^*_1)  - \log Z(m^*_2, n^*_2; N+K, N+K) + \log w_{n^*_1, m^*_1} + \log w_{n^*_2, m^*_2} \leq \\
&N^{1/4} + 8 \theta^{-1} \log N +  \log Z(-K,-K; N+K, N+K) + 4K \Psi(\theta/2)   - \MY_N^{\mathsf{SW}}  - \MY_N^{\mathsf{NE}} \leq \MY_N,
\end{split}
\end{equation*}
where in the first inequality we used (\ref{ASR}), in the second inequality we used the definitions of $\MY_N^{\mathsf{NE}}, \MY_N^{\mathsf{SW}}$ from (\ref{S3RVS}) and also that $\log w_{i,j} \leq 4 \theta^{-1} \log N$ since on the event $E_N$, we have $E_N^2$ as well. In the last inequality we used the definition of $\MY_N$ from (\ref{S3YDef}) and the fact that $8 \theta^{-1} \log N  \leq N^{1/4}$ for all large $N$. This establishes (\ref{S3RF}) and hence the proposition.

%
\section{Proof of Theorem \ref{PM2}} \label{Section4} In this section we work in the case $\theta = \theta_c$, where $\theta_c$ is as in (\ref{DefThetaC}). For this choice of $\theta$ we construct a sequence of random variables $\MY_N $ such that $\MY_N \leq \Fe_N$, as in (\ref{Fmax}), and $\MY_N$ are, with high probability, bounded away from zero on scale $N^{1/3} (\log N)^{2/3}$. The precise result is given as Proposition \ref{MainCrit} in Section \ref{Section4.1}. In the same section we use Proposition \ref{MainCrit} along with Lemma \ref{S2S2L2}(which gives an upper bound) to prove Theorem \ref{PM2}. The construction of $\MY_N$ is accomplished in Section \ref{Section4.2} and relies on two key lemmas, Lemma \ref{CloseZ} and Lemma \ref{CloseTW}, whose proofs are given in Sections \ref{Section4.3} and \ref{Section6}, respectively. We continue with the same notation as in Sections \ref{Section1} and \ref{Section2}.

%
\subsection{Lower bounds for $\Fe_N$.} \label{Section4.1} We prove Theorem \ref{PM2} via the following (proved in Section \ref{Section4.2}).
\begin{proposition}\label{MainCrit} Suppose that $\theta = \theta_c$. There exists a sequence of random variables $\MY_N$ such that
\begin{equation}\label{S4MainEq}
\begin{split}
& \MY_N \leq \Fe_N \mbox{ a.s. for all large $N$, and  $\lim_{N \rightarrow \infty} \mathbb{P}\left(\MY_N \geq 20^{-1} \cdot \sigma_{\theta_c}(1) \cdot N^{1/3} (\log N)^{2/3} \right) = 1$, }
\end{split}
\end{equation}
where we recall that $\sigma_{\theta_c}(x)$ is as in (\ref{DefSigmaS2}).
\end{proposition}

\begin{proof}[Proof of Theorem \ref{PM2}] From Proposition \ref{S4MainEq},
$\lim_{N \rightarrow \infty} \mathbb{P} \left( \Fe_N \geq 20^{-1} \cdot  \sigma_{\theta_c}(1) \cdot    N^{1/3} (\log N)^{2/3} \right) = 1,$
which establishes the lower bound in Theorem \ref{PM2} with $c_1 = 20^{-1} \cdot  \sigma_{\theta_c}(1)$.

On the other hand, we have by Lemma \ref{S2S2L2} (applied to $a = 5$, $M = N$ and $\theta = \theta_c$) that there exists $D_1 > 0$ such that for all large $N$,
$
\mathbb{P} \left(\log Z (1, 1; m, n) \geq D_1 N^{1/3} (\log N)^{2/3} \right) \leq N^{-5}.
$
Since $Z(m_1,n_1; m_2,n_2)$ has the same distribution as $Z(a + m_1,b + n_1; a+ m_2, b +n_2)$ for any $a,b \in \mathbb{Z}$, we conclude the same with $(1,1;m,n)$ replaced by $(m_1,n_1;m_2,n_2)$ for all $m_1, n_1, m_2, n_2 \in \llbracket 1, N \rrbracket$ with $(m_1,n_1) \leq (m_2, n_2)$
Taking a union bound of these inequalities (note that we have at most $N^4$ possible choices for the quadruple of $m_1, n_1, m_2, n_2$) we conclude that
$$\mathbb{P} \left( \Fe_N \geq D_1 N^{1/3} (\log N)^{2/3}\right) \leq N^{-1},$$
which establishes the upper bound in Theorem \ref{PM2} with $c_2 = D_1$.
\end{proof}

%
\subsection{Proof of Proposition \ref{MainCrit}}\label{Section4.2} In this section we present the proof of Proposition \ref{MainCrit}. We begin this section by introducing some notation and explicitly constructing the variables $\MY_N$ in the statement of Proposition \ref{MainCrit}. Afterwards we explain the main ideas behind our construction, and state two lemmas, which will be used in our analysis. Finally, we show that the $\MY_N$ we construct satisfy the conditions in Proposition \ref{MainCrit}.

\begin{definition}\label{DefHex}
For $P = (x_1, y_1), Q = (x_2, y_2) \in \mathbb{R}^2$ with $x_1 \leq x_2$ and $y_1 \leq y_2$ and $a \in [0, \infty)$ satisfying $\min(x_2 - x_1, y_2 - y_1) \geq a$ we let $H(P;Q;a)$ denote the hexagonal domain in $\mathbb{R}^2$, whose boundary is obtained by straight segments connecting the points
$$(x_1, y_1) \rightarrow (x_1 + a, y_1) \rightarrow (x_2, y_2 - a) \rightarrow (x_2, y_2) \rightarrow (x_2 - a, y_2) \rightarrow (x_1, y_1 + a) \rightarrow (x_1, y_1).$$
The hexagon $H(P;Q;a)$ is illustrated in Figure \ref{S4_1}.
\end{definition}

\begin{figure}[h]
\scalebox{0.55}{\includegraphics{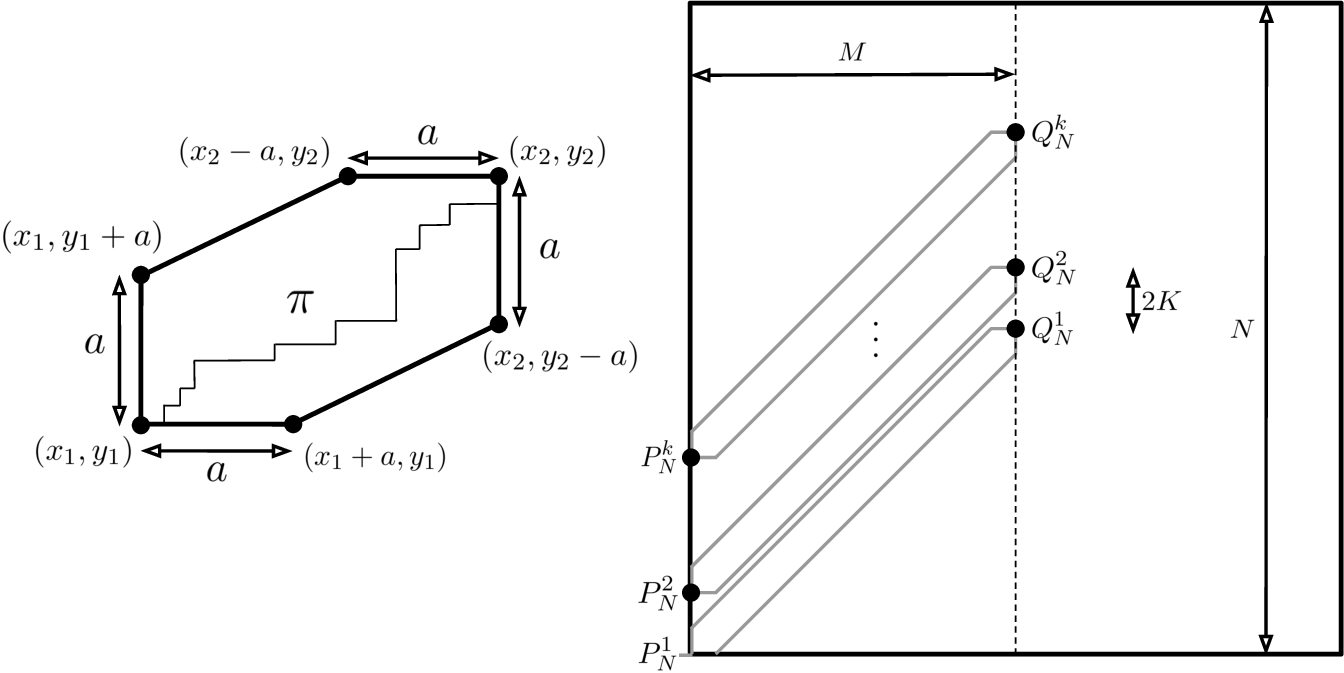}}
\captionsetup{width=\linewidth}
 \caption{The left picture represents the hexagonal domain $H(P;Q;a)$ with $P = (x_1, y_1)$ and $Q = (x_2, y_2)$ and a path $\pi \in \Pi(P; Q; a)$. The right figure depicts the hexagonal domains $H_N^i = H(P_N^i; Q_N^i; K-1)$ for $i = 1, \dots, k$, where $M = \lfloor N/2 \rfloor$, $K =  \lfloor M^{3/4} \rfloor $ and $k = \lfloor N^{1/8} \rfloor$.}
\label{S4_1}
\end{figure}

If $P = (m_1, n_1), Q = (m_2, n_2) \in \mathbb{Z}^2$ are such that $(m_1, n_1) \leq (m_2, n_2)$ and $a \in [0, \infty)$ we let
\begin{equation}\label{PathHex}
\Pi(P; Q; a) = \Pi(m_1, n_1; m_2, n_2; a) := \{ \pi \in \Pi(m_1, n_1; m_2, n_2): V(\pi) \subset H(P;Q; a) \},
\end{equation}
where we recall that $\Pi(m_1, n_1; m_2, n_2)$ was defined in Section \ref{Section1} and was the set of up-right paths connecting $(m_1, n_1)$ and $(m_2, n_2)$, and we have written $V(\pi)$ for the set of vertices in $\mathbb{Z}^2$ that are visited by $\pi$. We also define the {\em modified partition function}
 \begin{equation}\label{PFHex}
Z(P; Q; a) = Z(m_1, n_1; m_2, n_2; a) := \sum\nolimits_{ \pi \in\Pi(m_1,n_1;m_2,n_2; a)} \prod\nolimits_{(i,j) \in \pi} w_{i,j}.
\end{equation}
If $\Pi(m_1,n_1;m_2,n_2; a) = \emptyset$ we adopt the convention that $Z(P; Q; a)  = 0$.

With the above notation in place we can proceed with the construction of $\MY_N$. For $N \in \mathbb{N}$ we denote for simplicity $M = \lfloor N/2 \rfloor$, $K =  \lfloor M^{3/4} \rfloor $ and $k = \lfloor N^{1/8} \rfloor$. We also define the points $P_N^i, Q_N^i$ and hexagonal domains $H_N^i$ for $i = 1, \dots, k$ via
\begin{equation}\label{S4DefPoints}
P_N^i = (1, 1 + 2 K \cdot (i-1)), \hspace{3mm} Q_N^i = (M, M + 2K \cdot (i-1)), \mbox{ and } H_N^i = H(P_N^i; Q_N^i; K-1).
\end{equation}
The latter domains are depicted in Figure \ref{S4_1}. Finally, we define
\begin{equation}\label{S4DefYN}
\MY_N = \max_{i = 1, \dots, k} \log Z_N^i , \mbox{ with }Z_N^i = Z(P_N^i; Q_N^i; K-1).
\end{equation}
This specifies the choice of the $\MY_N$, for which we will establish Proposition \ref{MainCrit}.

\smallskip
Let us explain some of the ideas behind the above choice of $\MY_N$. The idea is to find a large collection of ${k}$  i.i.d. positive random variables ${Z}_N^i$ (${k} \geq N^{\alpha}$ for some $\alpha > 0$ will suffice), such that
\begin{enumerate}
\item $\log {Z}_N^i$ are of order $N^{1/3}$;
\item for large $x$ and some $\gamma > 0$ we have $\mathbb{P}( N^{-1/3} \log {Z}_N^i \geq x) \geq e^{-\gamma x^{3/2}}$;
\item $\Fe_N \geq \log {Z}_{N}^i$ a.s. for all $i = 1, \dots, k$.
\end{enumerate}
Once we have found such a collection, we should be able to show that $\mathsf{Y}_N := \max_{i =1, \dots, k}\log {Z}_{N}^i$ is a lower bound for $\Fe_N$ and $\mathsf{Y}_N$ is bounded away from zero by $c N^{1/3} (\log N)^{2/3}$ with high probability where the constant $c$ would depend on $\alpha$ and $\gamma$ above. The improvement from $N^{1/3}$ to $N^{1/3} (\log N)^{2/3}$ comes from the independence of the ${Z}_i^N$, the fact that ${k} \geq N^{\alpha}$ and the upper tail lower bound in condition (2) above.

Since $\exp(\Fe_N)$ is the maximal point-to-point partition function, it makes sense to have our choice of ${Z}_{N}^i$ be related to some type of partition functions of the polymer. A na\"ive guess is to choose ${Z}_N^i = Z(P_N^i; Q_N^i)$ for $i = 1, \dots, {k}$, where we recall that $Z(P_N^i; Q_N^i)$ are the partition functions from $P_N^i$ to $Q_N^i$. Such a choice clearly satisfies condition (3). In view of Proposition \ref{LGPCT}, we have that $\frac{ \log Z(P_N^i; Q_N^i)}{M^{1/3} \sigma_{\theta_c}(1)}$ weakly converge to $F_{\rm GUE}$ and so condition (1) is also satisfied. One also expects (although it does not follow from weak convergence alone) that condition (2) is satisfied in light of the tail behavior of $F_{\rm GUE}$. Specifically, from \cite[Theorem 1]{DV13} we have that there exists a constant $\lambda > 0$ such that for all $x \geq 1$
\begin{equation}\label{TWLB}
\exp \left( -\frac{4}{3} x^{3/2} +  \lambda \log (x+1) \right) \geq 1 - F_{\rm GUE}(x) \geq \exp \left( -\frac{4}{3} x^{3/2} -  \lambda \log (x+1) \right),
\end{equation}
which suggests that one might be able to establish a similar upper tail exponent of $3/2$ for $\log Z(P_N^i; Q_N^i)$, which would yield (2).

The main issue with the above proposal for ${Z}_N^i$ is that the resulting variables are not independent. Indeed, the $Z(P_N^i; Q_N^i)$ and $Z(P_N^j; Q_N^j)$ for $i \neq j$ are summations over paths that may pass through the same vertices and so the same $w_{m,n}$ appear in the formulas for both partition functions. The key idea is then to restrict the summation from all up-right paths that connect $P_N^i$ to $Q_N^i$ to ones that belong in narrow enough diagonal strips -- these are the $H_N^i$'s. The resulting functions are precisely $Z(P_N^i; Q_N^i; K-1)$ and we take ${Z}_N^i =Z(P_N^i; Q_N^i; K-1)$. Since ${Z}_N^i$ involve $w_{m,n}$ with $(m,n) \in H_N^i$ and the latter are pairwise disjoint we indeed have that ${Z}_N^i$ are i.i.d.

At this point there are two main issues that need to be addressed. In order for our strategy to be successful we must show that the following conditions hold:
\begin{enumerate}[label=\Roman*.]
\item  $Z(P_N^1; Q_N^1; K-1)$ is {\em close} to $Z(P_N^1; Q_N^1)$;
\item  $\log Z(P_N^1; Q_N^1)$ has an upper tail exponent of $3/2$.
\end{enumerate}

Based on a version of the fact that the transversal fluctuations for the log-gamma polymer are of order $N^{2/3}$, we expect that statement I is true provided that $K$ is much bigger than $N^{2/3}$ and that is why we choose $K = \lfloor M^{3/4} \rfloor \sim N^{3/4}$. A precise formulation of statement I is given in Lemma \ref{CloseZ} below. The form of statement II that we establish is given in Lemma \ref{CloseTW} .

\begin{lemma}\label{CloseZ} Let $\theta = \theta_c$ as in (\ref{DefThetaC}). There exists $M_0 \in \mathbb{N}$ such that for $M \geq M_0$ and $K =  \lfloor M^{3/4} \rfloor$
$$\mathbb{P}\left( Z(1,1; M, M ) - Z(1,1; M, M; K-1) \geq e^{-M^{1/2} h''_{\theta_c}(1)/4} \right) \leq  e^{-M^{1/6}},$$
where we recall that $h_{\theta}(x)$ is as in (\ref{HDefLLN}).
\end{lemma}

\begin{lemma}\label{CloseTW} Let $\theta = \theta_c$ as in (\ref{DefThetaC}). For any $\alpha \in (0, 1/3)$ there is $M_0 \in \mathbb{N}$ such that for $M \geq M_0$
$$\mathbb{P}\left(    \log Z(1,1; M, M) \geq \sigma_{\theta_c}(1) \cdot M^{1/3} \cdot (\alpha \log M)^{2/3}\right) \geq M^{-\alpha},$$
where we recall that $\sigma_{\theta_c}(x)$ is as in (\ref{DefSigmaS2}).
\end{lemma}
\begin{remark}
Neither Lemma \ref{CloseZ}, nor Lemma \ref{CloseTW} are optimal though they will suffice for  proving Proposition \ref{MainCrit}. Lemma \ref{CloseZ} is proved in Section \ref{Section4.3} and is essentially a consequence of Proposition \ref{S1LDE1} and Lemma \ref{S2S2L1}. Lemma \ref{CloseTW} is proved below by utilizing Proposition \ref{S2LBProp}.
\end{remark}
\begin{proof}[Proof of Lemma \ref{CloseTW}] Let us fix $\alpha \in (0, 1/3)$ and $\epsilon_0 \in (\alpha, 1/3)$. From (\ref{eq:rescaledpartitionfunction}) and (\ref{S2SummH}) we have
$$ \sigma_{\theta_c}(1) \cdot M^{1/3} \cdot \mathcal{F}(M,M) = \log Z(1,1;M,M),$$
where we used that  $\theta = \theta_c$. The latter and Proposition \ref{S2LBProp} applied to $\theta = \theta_c$, $\delta = 1$ and $\epsilon_0$ as above implies that there exists $M_1, C_1 > 0$ such that for $M \geq M_1$ we have
$$\left| \mathbb{P}\left(  \log Z(1,1; M, M) \geq  \sigma_{\theta_c}(1) \cdot M^{1/3} \cdot (\alpha \log M)^{2/3}\right ) - (1 - F_{\rm GUE}\left( (\alpha \log M)^{2/3}) \right) \right| \leq C_1 M^{-\epsilon_0}.$$
On the other hand, by (\ref{TWLB}) we have for some $\lambda > 0$ and all large $M$
$$1 - F_{\rm GUE}\left( (\alpha \log M)^{2/3}) \right)\geq \exp \left( -\frac{4\alpha}{3}  \log M -  \lambda \log \left( (\alpha  \log M)^{2/3} + 1\right) \right) \geq \frac{2}{M^{\alpha}}.$$
Combining the last two inequalities we see that for all large $M$
$$\mathbb{P}\left(   \log Z(1,1; M, M) \geq  \sigma_{\theta_c}(1) \cdot M^{1/3} \cdot (\alpha \log M)^{2/3} \right )  \geq  2M^{-\alpha} - C_1 M^{-\epsilon_0},$$
which implies the lemma as $\epsilon_0 > \alpha$.
\end{proof}

The above discussion aimed to explain the origin of the choice of $\MY_N$ in (\ref{S4DefYN}). The only quantity we did not discuss is $k = \lfloor N^{1/8} \rfloor$. We make this choice to ensure that $H_N^i$ all belong to $\llbracket 1, N \rrbracket^2$ for $i = 1, \dots, k$. We mention that one could make different choices for $K$ and $k$, which might produce a higher constant than $ 20^{-1} \cdot \sigma_{\theta_c}(1)$ in (\ref{S4MainEq}) -- we have not attempted to optimize over such choices.

\smallskip
We now proceed to show that the $\MY_N$ as in (\ref{S4DefYN}) satisfy Proposition \ref{MainCrit}.
\begin{proof}[Proof of Proposition \ref{MainCrit}]
We prove the proposition with $\MY_N$ as in (\ref{S4DefYN}). We also continue with the same notation as in the beginning of the section. By definition, we have for all large $N$ that $P_N^i, Q_N^i \in \llbracket 1, N \rrbracket^2$ for $i =1, \dots, k$ and so
$$\Fe_N \geq \max_{i = 1, \dots, k }  \log  Z(P_N^i; Q_N^i) \geq \max_{i = 1, \dots, k }  \log  Z(P_N^i; Q_N^i; K-1) = \MY_N,$$
which implies the first part of the proposition.

Let us set $R = 20^{-1} \cdot \sigma_{\theta_c}(1)$ for simplicity. Since $Z(P_N^i; Q_N^i; K-1)$ depends only on $w_{m,n}$ for $(m,n) \in H_N^i$ as in (\ref{S4DefPoints}), which are disjoint for all large $N$, we conclude that $\log  Z(P_N^i; Q_N^i; K-1)$ are i.i.d. random variables with the same distribution as $\log  Z(P_N^1; Q_N^1; K-1)$. The latter implies
\begin{equation}\label{S4DB1}
\begin{split}
&\mathbb{P} \left( \MY_N \geq R N^{1/3} ( \log N)^{2/3} \right) = \mathbb{P} \left( \max_{i = 1, \dots, k }  \log  Z(P_N^i; Q_N^i; K-1) \geq  R N^{1/3} ( \log N)^{2/3}  \right) =\\
& 1 - \mathbb{P} \left( \log  Z(P_N^1; Q_N^1; K-1) <  R N^{1/3}  (\log N)^{2/3}  \right)^k.
\end{split}
\end{equation}
From Lemma \ref{CloseTW} (applied to $\alpha = 1/10$ and $M = \lfloor N/2 \rfloor$) we have for all large $N$ that
$$\mathbb{P} \left( \log  Z(P_N^1; Q_N^1) < 2 R N^{1/3}     (\log N)^{2/3}  \right) \leq 1 - M^{-1/10}.$$
The latter inequality and Lemma \ref{CloseZ} imply that for all large $N$
$$\mathbb{P} \left( \log  Z(P_N^1; Q_N^1; K-1) < R N^{1/3}     (\log N)^{2/3}  \right) \leq 1 - (1/2)M^{-1/10}.$$

Substituting the last inequality into (\ref{S4DB1}) we obtain
$$\mathbb{P} \left( \MY_N \geq R N^{1/3} ( \log N)^{2/3} \right) \geq 1 - \left(1- (1/2) M^{-1/10} \right)^k,$$
which implies the second part of (\ref{S4MainEq}) since $k = \lfloor N^{1/8} \rfloor$ and $M = \lfloor N/2 \rfloor$.

\end{proof}

%
\subsection{Proof of Lemma \ref{CloseZ}}\label{Section4.3} For clarity we split the proof of the lemma into three steps.

\smallskip
{\bf \raggedleft Step 1.} We begin by introducing some notation. Let $H$ denote the hexagonal domain $H(P;Q;K-1)$ for $P = (1,1)$ and $Q = (M, M)$ as in Definition \ref{DefHex}. For $M \geq 2$ we define
\begin{equation*}
\begin{split}
&\partial H^{top} := \{ (m,n) \in \mathbb{Z}^2: (m, n-1) \in H \mbox{ and } n-m = K \},  \\
&\partial H^{bot} := \{ (m,n) \in \mathbb{Z}^2: (m-1, n) \in H \mbox{ and } m - n = K \}.
\end{split}
\end{equation*}
In words, $\partial H^{top}$ is the set of points $\mathbb{Z}^2$ that lie on the segment connecting $(1, K+1)$ and $(M-K, M)$, and $\partial H^{bot}$ is the set of points $\mathbb{Z}^2$ that lie on the segment connecting $(K+1, 1)$ and $(M, M - K)$.

For every $(m,n) \in \partial H^{top} \sqcup \partial H^{bot}$ we let $(m^*, n^*)$ denote $(m, n -1)$ if $(m,n) \in \partial H^{top}$ and $(m-1, n)$ if $(m,n) \in \partial H^{bot}$. We claim that there exist $M_1, D_1, d_1 > 0$ such that if $M \geq M_1$ and $(m,n) \in \partial H^{top} \sqcup \partial H^{bot}$ we have
\begin{equation}\label{S4MR1}
\begin{split}
&\mathbb{P} \left(  E_M \right) \leq D_1 \exp \left( - d_1 M^{3/16} \right), \mbox{ with } \\
& E_M = \left \{ Z(1,1; m^*, n^*) \cdot  Z(m,n; M, M) \geq e^{-M^{1/2} h''_{\theta_c}(1)/3}  \right\}.
\end{split}
\end{equation}
We prove (\ref{S4MR1}) in the steps below. Here we assume its validity and conclude the proof of the lemma.

\smallskip
From the definition of $Z(1,1; M, M; K-1) $ in (\ref{PFHex}) we have that
$$Z(1,1; M, M ) - Z(1,1; M, M; K-1) = \sum_{\pi \in \Pi(m_1,n_1;m_2,n_2) \setminus \Pi(m_1,n_1;m_2,n_2; K-1) } \prod\nolimits_{(i,j) \in \pi} w_{i,j}.$$
Notice that each of the paths $\pi \in \Pi(m_1,n_1;m_2,n_2) \setminus \Pi(m_1,n_1;m_2,n_2; K-1) $ needs to contain at least one vertex in $\partial H^{top} \sqcup \partial H^{bot}$. Splitting the sum over the first vertex $(m,n)$ in $\pi$ that belongs to $\partial H^{top} \sqcup \partial H^{bot}$ we obtain the upper bound
$$Z(1,1; M, M ) - Z(1,1; M, M; K) \leq \sum_{(m,n) \in \partial H^{top} \sqcup \partial H^{bot}} Z(1,1; m^*, n^*) \cdot  Z(m,n; M, M) .$$
The latter inequality and a union bound of (\ref{S4MR1}) gives
\begin{equation*}
\begin{split}
&\mathbb{P} \left( Z(1,1; M, M ) - Z(1,1; M, M; K-1) \geq 2M  e^{-M^{1/2} h''_{\theta_c}(1)/3} \right) \leq 2M D_1 \exp \left( - d_1 M^{3/16} \right),
\end{split}
\end{equation*}
where we used that $|\partial H^{top} \sqcup \partial H^{bot}| \leq 2M$. The latter inequality implies the statement of the lemma, since $h''_{\theta_c}(1) > 0$ by (\ref{S2SummH}) and $3/16 > 1/6$.

\smallskip
{\bf \raggedleft Step 2.} By the distributional equality of $Z(m_1, n_1; m_2, n_2)$ and $Z(n_1, m_1; n_2, m_2)$ it suffices to prove (\ref{S4MR1}) when $(m,n) \in \partial H^{bot}$. Furthermore, by the distributional equality of $Z(m_1, n_1; m_2, n_2)$ and $Z(-m_2, -n_2; -m_1, -n_1)$ we only need to consider the case when $m = K + n$ and $1 \leq n \leq \lfloor M/2 \rfloor - K + 1$ (i.e. the left half of $\partial H^{bot}$). In particular, we have $m^* = m-1$ and $n^* = n$.

In this step we prove (\ref{S4MR1}) when $m = K + n$ and $1 \leq n \leq \lfloor M/2 \rfloor - K + 1$ and $n/(m-1) \geq 1/2$. In this case we have from Proposition \ref{S1LDE1} (applied to $\delta = 1/2$, $x = M^{1/8}$) that there are constants $C_1, c_1 > 0$ such that for all large $M$ we have
\begin{equation}\label{S4PR1}
\begin{split}
&\mathbb{P}(E_M^1) \leq C_1 e^{-c_1 M^{3/16} },  \mbox{ with }  \\
&E_M^1 = \left\{ \frac{\log Z(1,1; m^*, n^*) + (m-1)h_{\theta_c}(n/(m-1)) }{(m-1)^{1/3} \sigma_{\theta_c}(n/(m-1)) }\geq M^{1/8}  \right\},
\end{split}
\end{equation}
where we used that $m^* = m-1$, $n^* = n$, $m \geq K$ and $K  = \lfloor M^{3/4} \rfloor$. By the same proposition we also know that there constants $C_2, c_2 > 0$ such that for all large $M$ we have
\begin{equation}\label{S4PR2}
\begin{split}
&\mathbb{P}(E_M^2) \leq C_2 e^{-c_2 M^{3/16} }, \mbox{ with } \\
& E_M^2 =\left\{ \frac{\log Z(m,n; M, M) + (M - n+1)h_{\theta_c}\left( \frac{M-m+1}{M-n+1} \right) }{(M - n+1)^{1/3} \sigma_{\theta_c}\left( \frac{M-m+1}{M-n+1} \right)  }\geq M^{1/8}  \right\} ,
\end{split}
\end{equation}
where we also used the distributional equality of $Z(m_1, n_1; m_2, n_2)$ and $Z(n_1, m_1; n_2, m_2)$ .

In the remainder of this step we prove that for all $M$ sufficiently large we have that $E_M \subset E_M^1 \cup E_M^2$, which in view of (\ref{S4PR1}) and (\ref{S4PR2}), would imply (\ref{S4MR1}).

\smallskip
We will show that $E_M \subset E_M^1 \cup E_M^2$ by demonstrating that $ ( E_M^1)^c  \cap (E_M^2)^c \subset E^c_M$. The latter statement would follow if we can show that for all large $M$
\begin{equation}\label{S4MR2}
\begin{split}
&M^{1/8} (m-1)^{1/3} \sigma_{\theta_c}(n/(m-1)) - (m-1)h_{\theta_c}(n/(m-1))  + \\
&M^{1/8} (M - n+1)^{1/3} \sigma_{\theta_c}\left( \frac{M-m+1}{M-n+1} \right)- (M - n+1)h_{\theta_c}\left( \frac{M-m+1}{M-n+1} \right) \leq - \frac{M^{1/2} h''_{\theta_c}(1)}{3}.
\end{split}
\end{equation}

We recall from Section \ref{Section2.1} that $\sigma_{\theta_c}(x)$ is a smooth positive function on $(0, \infty)$ and both $\frac{n}{m-1}$ and $\frac{M-m + 1}{M - n +1}$ belong to $[1/2, 1]$. This means that we can find a constant $A_1 > 0$ such that
\begin{equation}\label{FDS1}
\begin{split}
&M^{1/8} (m-1)^{1/3} \sigma_{\theta_c}(n/(m-1)) \leq A_1 M^{11/24} \mbox{ and } \\
&M^{1/8} (M - n+1)^{1/3} \sigma_{\theta_c}\left( \frac{M-m+1}{M-n+1} \right) \leq A_1 M^{11/24},
\end{split}
\end{equation}
where we also used that $m \leq M$ and $M-n + 1 \leq M$.

Using that $m = n + K$, (\ref{S2SummH}) and (\ref{S2SummH2}) we get
\begin{equation}\label{FDS2}
(M - n+1)h_{\theta_c}\left( \frac{M-m+1}{M-n+1} \right) =   \frac{h''_{\theta_c}(1) K^2}{2(M-n+1)} + O(M^{1/4})  \mbox{ and } (m-1)h_{\theta_c}\left(\frac{n}{m-1} \right)  \geq 0.
\end{equation}

Combining (\ref{FDS1}) and (\ref{FDS2}) we get for all large $M$
\begin{equation*}
\begin{split}
&M^{1/8} (m-1)^{1/3} \sigma_{\theta_c}(n/(m-1)) - (m-1)h_{\theta_c}(n/(m-1))  + M^{1/8} (M - n+1)^{1/3} \sigma_{\theta_c}\left( \frac{M-m+1}{M-n+1} \right) \\
&- (M - n+1)h_{\theta_c}\left( \frac{M-m+1}{M-n+1} \right) \leq (2A_1+1) M^{11/24} -  \frac{h''_{\theta_c}(1) K^2}{2(M-n+1)} .
\end{split}
\end{equation*}
Since $K = \lfloor M^{3/4} \rfloor$ and $h''_{\theta_c}(1) > 0$ we see that the last inequality implies (\ref{S4MR2}).

\smallskip
{\bf \raggedleft Step 3.} In this step we prove (\ref{S4MR1}) when $m = K + n$ and $1 \leq n \leq \lfloor M/2 \rfloor - K + 1$ and $n/(m-1) \in (0, 1/2]$. Let us set $\epsilon = \min (1, h_{\theta_c}(1/2))$, which is positive in view of (\ref{S2SummH2}). From Lemma \ref{S2S2L1} we have that there are positive constants $B_1, b_1 > 0$ such that
\begin{equation}\label{S4PR3}
\mathbb{P}(E_M^3) \hspace{-0.5mm} \leq \hspace{-0.5mm} B_1 e^{-b_1 M^{3/4} },  \hspace{-0.5mm}  \mbox{ with } \hspace{-0.5mm} E_M^3 = \left\{ \log Z(1,1; m^*, n^*) + (m-1)\left(h_{\theta_c}\left(\frac{n}{m-1}\right) - \epsilon\right) \geq 0  \right\} \hspace{-0.5mm}.
\end{equation}
In addition, since $1 \leq n \leq \lfloor M/2 \rfloor - K + 1$, we have that (\ref{S4PR2}) also holds for all large enough $M$.

In the remainder of this step we prove that for all $M$ sufficiently large we have that $E_M \subset E_M^2 \cup E_M^3$, which in view of (\ref{S4PR2}) and (\ref{S4PR3}), would imply (\ref{S4MR1}).

\smallskip
We will show that $E_M \subset E_M^2 \cup E_M^3$ by demonstrating that $ ( E_M^2)^c  \cap (E_M^3)^c \subset E^c_M$. The latter statement would follow if we can show that
\begin{equation}\label{S4MR3}
\begin{split}
&- (m-1)(h_{\theta_c}(n/(m-1)) - \epsilon)  + M^{1/8} (M - n+1)^{1/3} \sigma_{\theta_c}\left( \frac{M-m+1}{M-n+1} \right) - \\
& (M - n+1)h_{\theta_c}\left( \frac{M-m+1}{M-n+1} \right) \leq - \frac{M^{1/2} h''_{\theta_c}(1)}{3}.
\end{split}
\end{equation}
In view of (\ref{S2SummH2}) and the definition of $\epsilon$ we have
$$h_{\theta_c}(n/(m-1)) - \epsilon \geq 0.$$
Combining the latter with the second line of (\ref{FDS1}) and the first equality in (\ref{FDS2}) we conclude
\begin{equation*}
\begin{split}
&- (m-1)(h_{\theta_c}(n/(m-1)) - \epsilon)  + M^{1/8} (M - n+1)^{1/3} \sigma_{\theta_c}\left( \frac{M-m+1}{M-n+1} \right) - \\
& (M - n+1)h_{\theta_c}\left( \frac{M-m+1}{M-n+1} \right) \leq (A_1+1) M^{11/24} - \frac{h''_{\theta_c}(1) }{2} \cdot \frac{K^2}{(M-n+1)},
\end{split}
\end{equation*}
which implies (\ref{S4MR3}) as $K = \lfloor M^{3/4} \rfloor$ and $h''_{\theta_c}(1) > 0$. This concludes the proof of the lemma.

%
\section{Proof of Theorem \ref{PM3}} \label{Section5}

In this section we prove Theorem \ref{PM3}. The theorem follows immediately from Lemma \ref{S5LB} (which provides a lower bound) and Lemma \ref{S5UB} (which provides an upper bound). The lower bound relies on the simple observation that $\Fe_N$, as in (\ref{Fmax}), cannot be smaller than the maximum of the variables $\log w_{i,j}$ for $i,j\in \llbracket 1, N\rrbracket$. To prove the upper bound, we will use Lemma \ref{S2S2L2}. In this section we continue with the notation of Sections \ref{Section1} and \ref{Section2}.

\subsection{Lower bound} The following result proves the lower bound in Theorem \ref{PM3} for any $c_1 < 2/\theta$.
\begin{lemma}\label{S5LB}
Fix $\theta>0$. For any $\eps>0$,
$$ \lim_{N\to\infty} \mathbb P\left( \Fe_N \geq \left( 2/\theta  -\eps  \right) \log N\right) =1 $$
\end{lemma}
\begin{proof} Recall the definition of $\Fe_N$ in \eqref{Fmax}.
By restricting the maximum over $(m_1,n_1)=(m_2,n_2)$, we have that  $\Fe_N \geq \log w^{\max}_N$ where $w^{\max}_N=\max_{1\leq m,n\leq N} w_{m,n}$.  The $w_{m,n}$ form a collection of $N^2$ i.i.d. inverse-gamma random variables. Thus, for any $c>0$, we have
\begin{equation}
\mathbb P\left( \Fe_N \geq c \log N\right) \geq \mathbb P\left( \log w^{\max}_N \geq c\log N \right)  =1-  \mathbb P\left(  w_{1,1} \leq  N^c \right)^{N^2} . \label{eq:trivialbound}
\end{equation}
We have
\begin{align*}
\mathbb P\left(  w_{1,1} \leq  N^c \right) &= 1-\frac{1}{\Gamma(\theta)}\int_{N^c}^{\infty} x^{-\theta-1}\exp(-x^{-1})\mathrm dx \leq 1-\frac{\exp\left( -N^{-c}\right)N^{-c\theta}}{\theta \Gamma(\theta)}.
\end{align*}
Using the bound $\log(1-x)\leq - x$ for $x \in (0,1)$, we find that for any $0<c< 2/ \theta$,
$$ \mathbb P\left(  w_{1,1} \leq  N^c \right)^{N^2}\leq \exp\left(- \frac{N^{2-c\theta} e^{-N^{-c}}}{\theta \Gamma(\theta)} \right) \xrightarrow[N\to\infty]{} 0.$$
Combining the last statement with \eqref{eq:trivialbound}, we find that for any $0 < c <2/ \theta$,
$$ \lim_{N\to\infty} \mathbb P\left( \Fe_N \geq c \log N\right) =1,$$
which proves the lemma.
\end{proof}


\subsection{Upper bound} The following result proves the upper bound in Theorem \ref{PM3}.
\begin{lemma}\label{S5UB}
Let $\theta_c$ be as in \eqref{DefThetaC} and fix $\theta > \theta_c$. There exists a constant $c_2$ (depending on $\theta$) with
$$ \lim_{N\to\infty} \mathbb P\left( \Fe_N \leq c_2\log N\right) =1.$$
\end{lemma}
\begin{proof} From Lemma \ref{S2S2L2} with $a = 5$, $M = N$ and $\theta$ as in the statement of the present lemma, we know that there are constants $M_2, D_2 > 0$ such that for $N \geq M_2$ and $m,n \in \llbracket 1, N \rrbracket$ we have
\begin{equation*}
\mathbb{P}( \log Z(1,1; m,n) \geq D_2 \log N) \leq N^{-5}.
\end{equation*}
Using the distributional equality of $Z(m_1, n_1; m_2, n_2)$ and $Z(m_1 + a , n_1 + b; m_2 + a, n_2 + b)$ for any $a,b \in \mathbb{Z}$ we conclude that for any $(m_1, n_1), (m_2, n_2) \in \llbracket 1, N \rrbracket^2$ with $(m_1, n_1) \leq (m_2, n_2)$ we have
\begin{equation*}
\mathbb{P}( \log Z(m_1,n_1; m_2,n_2) \geq D_2 \log N) \leq N^{-5}.
\end{equation*}
Taking a union bound of the above (there are at most $N$ choices for $m_1, n_1, m_2, n_2$) we get
$$\mathbb{P}(  \Fe_N \geq D_2 \log N) \leq N^{-1}.$$
The latter implies the statement of the lemma with $c_2 = D_2$.
\end{proof}
\begin{remark} In the case $\theta>3$, the constant $c_2$ in Lemma \ref{S5UB} can be chosen as any constant $c_2>2$. Indeed,
from the definition of $\Fe_N$ and (\ref{S2ExpBound}) we have
	\begin{equation*}
	\begin{split}
	\mathbb E \left[  e^{\Fe_N}\right] \leq   \sum_{\substack{(m_1, n_1) \leq (m_2,n_2)\\  m_1, n_1, m_2, n_2 \in \llbracket 1, N \rrbracket}} \mathbb{E} \left[ Z(m_1, n_1; m_2, n_2) \right] \leq \sum_{\substack{(m_1, n_1) \leq (m_2,n_2)\\  m_1, n_1, m_2, n_2 \in \llbracket 1, N \rrbracket}}  \left( \frac{2}{\theta - 1} \right)^{m_2 + n_2 - m_1 - n_1 +1} .
	\end{split}
	\end{equation*}
	The last inequality implies
	\begin{equation*}
	\begin{split}
	\mathbb E \left[  e^{\Fe_N}\right] \leq  \sum_{m_1 = 1}^N \sum_{n_1 = 1}^N \sum_{m_2 = m_1}^\infty \sum_{n_2 = n_1}^\infty\left( \frac{2}{\theta - 1} \right)^{m_2 + n_2 - m_1 - n_1 +1} = \frac{2( \theta - 1) N^2 }{(\theta - 3)^2},
	\end{split}
	\end{equation*}
	where the last equality used the formula for a geometric series, and it is here we use our assumption that $\theta > 3$. Thus, by Chebyshev's inequality, for any $\eps > 0$ we have
	$$ P\left( \Fe_N \geq (2+\eps)\log N\right) =  \mathbb P\left( e^{\Fe_N} \geq N^{2+\eps} \right) \leq  N^{-2 - \eps} \cdot \mathbb E \left[  e^{\Fe_N} \right] \xrightarrow[N\to\infty]{} 0,$$
	which proves the Lemma \ref{S5UB} for any $c_2 > 2$.
	\label{rem:betterc2}
\end{remark}

%
\section{Proofs of Propositions \ref{PMP1}, \ref{PMP2} and \ref{PMP3}} \label{Section7} In this section we prove Propositions \ref{PMP1}, \ref{PMP2} and \ref{PMP3} about the log-gamma polymer measure from Section \ref{SectionS.1}. Throughout this section we use freely notation from Sections \ref{Section1.1}, \ref{SectionS.1} and \ref{Section2} above.

%
\subsection{Proof of Proposition \ref{PMP1}} Throughout this section we assume that $\theta \in (0, \theta_c)$. 

We recall from (\ref{DefCorners}) for $\delta \in (0,1/3)$ the following corner sets 
\begin{equation}\label{S7DefCorners}
\mathsf{SW}_N^{\delta}: = \mathbb{N}^2 \cap[1, N^{1/3 + 2\delta}]^2 \mbox{ and }\mathsf{NE}_N^{\delta} :=  \mathbb{N}^2 \cap[N - N^{1/3 + 2\delta}, N]^2.
\end{equation}
We claim that for each $\delta \in (0,1/3)$ we have that 
\begin{equation}\label{Red1}
\lim_{N \rightarrow \infty} \mathbb{P} \left( (\XS_N,\YS_N) \in  \mathsf{SW}_N^{\delta}  \right) = 1 \mbox{ and } \lim_{N \rightarrow \infty} \mathbb{P} \left( (\XE_N,\YE_N) \in  \mathsf{NE}_N^{\delta}  \right) = 1.
\end{equation}
We prove (\ref{Red1}) below. For now, we assume its validity and conclude the proof of Proposition \ref{PMP1}. \\

Fix $\epsilon >0$ as in the statement of the proposition. From (\ref{Red1}) with $\delta = \min(1/6,  \epsilon/3)$ we immediately conclude that the first four sequences of variables in (\ref{SubCritAll}) converge to $0$ in probability as $N \rightarrow \infty$. To illustrate our point, let us focus on $\frac{\XS_N }{N^{1/3 + \epsilon}}$ and notice that for any $a > 0$ and $N$ large enough so that $aN^{\epsilon} > N^{2\delta}$
$$\mathbb{P}\left( \frac{|\XS_N| }{N^{1/3 + \epsilon}}  > a \right) = \mathbb{P}\left( \XS_N  > aN^{1/3 + \epsilon} \right) \leq \mathbb{P}\left( \XS_N  > N^{1/3 + 2\delta} \right) \leq  1 - \mathbb{P} \left( (\XS_N,\YS_N) \in  \mathsf{SW}_N^{\delta}  \right).$$
Combining the latter with (\ref{Red1}) and taking $N \rightarrow \infty$ we conclude that $\frac{\XS_N }{N^{1/3 + \epsilon}}$ converges to $0$ in probability. An analogous argument shows that $ \frac{\YS_N }{N^{1/3 + \epsilon}}, \frac{\XE_N - N }{N^{1/3 + \epsilon}}, \hspace{1mm} \frac{\YE_N - N }{N^{1/3 + \epsilon}}$ all converge to $0$ in probability. Since 
$\LP_N  = \XE_N + \YE_N - \YS_N - \YS_N$ we also conclude that 
$$\frac{\LP_N - 2N }{N^{1/3 + \epsilon}} =\frac{\XE_N - N }{N^{1/3 + \epsilon}} +  \frac{\YE_N - N }{N^{1/3 + \epsilon}} -  \frac{\XS_N }{N^{1/3 + \epsilon}} - \frac{\YS_N }{N^{1/3 + \epsilon}}$$
also converges to $0$ in probability. Finally, as $\SP_N = \frac{\YE_N - \YS_N + 1}{\XE_N - \XS_N + 1}$, we have 
$$ \frac{N (\SP_N - 1)}{ N^{1/3 + \epsilon}}  = \frac{(\YE_N - N) - \YS_N - (\XE_N-N) + \XS_N }{N^{1/3 + \epsilon}} \cdot \frac{N^{2/3 - \delta} }{N^{2/3 -\delta}  +  N^{-1/3 - \delta} \left[ (\XE_N - N) -  \XS_N +1 \right]}.$$
From our work so far, the first term in the above product converges to $0$ in probability and the second one (from our work with $\epsilon$ replaced with $\delta$) converges to $1$, and so the product converges to $0$ in probability as desired. This proves that all sequences of variables in (\ref{SubCritAll}) converge to $0$ in probability as $N \rightarrow \infty$.\\

In the remainder of this section we prove (\ref{Red1}). From Definitions \ref{Def1} and \ref{Def2} we know that for each $(m_1, n_1), (m_2, n_2) \in \llbracket 1,N \rrbracket^2$ we have 
\begin{equation}\label{LC0Copy}
\mathbb{P} \left(\XS_N = m_1, \YS_N = n_1, \XE_N = m_2, \YE_N = n_2 \right) = \mathbb{E} \left[ \frac{Z(m_1, n_1; m_2, n_2)}{\ZM} \right].
\end{equation}
Let $S^{\delta}_N$ be the set of quadruples $m_1, n_1, m_2, n_2   \in \llbracket 1, N \rrbracket$ such that $(m_1, n_1) \leq (m_2, n_2)$ and $((m_1, n_1), (m_2, n_2)) \not \in \mathsf{SW}_N^{\delta} \times \mathsf{NE}_N^{\delta}$ and note $|S^{\delta}_N| \leq N^{4}$. To prove (\ref{Red1}) it suffices to show that 
\begin{equation}\label{Red2}
\lim_{N \rightarrow \infty} \mathbb{P} \left( (\XS_N,\YS_N, \XE_N, \YE_N) \in S^{\delta}_N  \right) = 0.
\end{equation}

From (\ref{LC1}) we know that there exists $N_0 > 0$ (depending on $\theta$ and $\delta$) such that for all $N \geq N_0$ and $(m_1, n_1, m_2, n_2) \in S^{\delta}_N$ we have 
\begin{equation}\label{LC1Copy}
\mathbb{P} \left(\log Z (m_1, n_1; m_2, n_2) \geq -2N \Psi(\theta/2) -  N^{1/3} \log N \right) \leq N^{-5}.
\end{equation}
In addition, from Proposition \ref{LGPCT} we know that 
\begin{equation}\label{LC2Copy}
\lim_{N \rightarrow \infty} \mathbb{P} \left(E_N  \right) = 0, \mbox{ where } E_N = \{ \log Z (1, 1; N, N) \leq -2N \Psi(\theta/2) - (1/2) N^{1/3} \log N \}.
\end{equation}
Let us introduce the event
$$F_N = \cup_{(m_1,n_1, m_2,n_2) \in  S^{\delta}_N} \{ \log Z (m_1, n_1; m_2, n_2) \geq -2N \Psi(\theta/2) -  N^{1/3} \log N\}$$
and note that by union bound and (\ref{LC1Copy}) we have for all $N \geq N_0$
\begin{equation}\label{LC3Copy}
\mathbb{P}(F_N) \leq |S^{\delta}_N| \cdot N^{-5} \leq N^{-1}.
\end{equation}

Combining (\ref{LC0Copy}), (\ref{LC2Copy}) and (\ref{LC3Copy}), we get
\begin{equation}\label{LC4Copy}
\begin{split}
&\limsup_{N \rightarrow \infty} \mathbb{P} \hspace{-1mm}\left( (\XS_N,\YS_N, \XE_N, \YE_N) \in S^{\delta}_N  \right) \leq \\
& \limsup_{N \rightarrow \infty} \mathbb{P} \hspace{-1mm} \left( \left\{ (\XS_N,\YS_N, \XE_N, \YE_N) \in S^{\delta}_N \right\} \cap E^c_N \cap F^c_N \right) +  \limsup_{N \rightarrow \infty} \mathbb{P}(E_N) + \limsup_{N \rightarrow \infty} \mathbb{P}(F_N) = \\
& \limsup_{N \rightarrow \infty} \mathbb{P} \left( \left\{ (\XS_N,\YS_N, \XE_N, \YE_N) \in S^{\delta}_N \right\} \cap E^c_N \cap F^c_N \right) =  \\
&\limsup_{N \rightarrow \infty} \mathbb{E} \left[  \frac{\sum_{(m_1,n_1,m_2,n_2) \in S_N^{\delta} } Z(m_1,n_1;m_2;n_2) }{Z_N^w} \cdot {\bf 1}_{E^c_N} \cdot {\bf 1}_{F^c_N}\right].
\end{split}
\end{equation}
On the event $E^c_N$ we have 
$$Z_N^w \geq Z(1,1;N,N) \geq \exp( -2N \Psi(\theta/2) - (1/2)N^{1/3} \log N).$$
and on the event $F_N^c$ we have for $(m_1,n_1,m_2,n_2) \in S_N^{\delta}$
$$Z(m_1,n_1;m_2;n_2)  \leq \exp( -2N \Psi(\theta/2) - N^{1/3} \log N).$$
Combining the latter two statements, we see 
\begin{equation*}
\begin{split}
&\limsup_{N \rightarrow \infty} \mathbb{E} \left[  \frac{\sum_{(m_1,n_1,m_2,n_2) \in S_N^{\delta} } Z(m_1,n_1;m_2;n_2) }{Z_N^w} \cdot {\bf 1}_{E^c_N} \cdot {\bf 1}_{F^c_N}\right] \leq \\
& \limsup_{N \rightarrow \infty} \mathbb{E} \left[  \frac{|S_N^{\delta}| \cdot  \exp( -2N \Psi(\theta/2) - N^{1/3} \log N) }{\exp( -2N \Psi(\theta/2) - (1/2)N^{1/3} \log N)} \cdot {\bf 1}_{E^c_N} \cdot {\bf 1}_{F^c_N}\right] \leq \limsup_{N \rightarrow \infty} N^4 \cdot e^{-(1/2) N^{1/3} \log N} = 0,
\end{split}
\end{equation*}
which together with (\ref{LC4Copy}) establishes (\ref{Red2}).

%
\subsection{Proof of Proposition \ref{PMP2}} Throughout this section we assume that $\theta = \theta_c$. 

From Theorem \ref{PM2} we know that there exists a constant $\rho > 0$ such that 
\begin{equation}\label{OP1}
\lim_{N \rightarrow \infty} \mathbb{P}(A_N) = 0 \mbox{ where } A_N = \{ \mathsf{F}_N \leq \rho N^{1/3} (\log N)^{2/3}\},
\end{equation}
where we recall that $\mathsf{F}_N$ is the maximal point-to-point free energy as in (\ref{Fmax}).

Let $M_1,D_1 > 0$ be as in Lemma \ref{S2S2L2} for $\theta = \theta_c$ and $a = 5$. In addition, suppose that $\delta \in (0,1)$ is small enough so that $\delta^{1/3} D_1 < \rho/4$. We first show that
\begin{equation}\label{OP2}
\limsup_{N \rightarrow \infty} \mathbb{P}(   \LP_N \leq \delta N ) = 0,
\end{equation}
which would imply (\ref{CritLength}) with $c_1 = \delta$ once we recall that $\LP_N \leq 2N-2$ by definition. 

Let $M_2 \in \mathbb{N}$ be sufficiently large so that $N_\delta := \lfloor \delta N \rfloor +1 \geq \max(2,M_1)$ for all $N \geq M_2$. From Lemma \ref{S2S2L2} we have for $N \geq M_2$ and $(m, n) \in \llbracket 1, N_\delta \rrbracket^2$ that 
\begin{equation*}
 \mathbb{P}\left(\log Z(1,1;m,n) \geq D_1 N_{\delta}^{1/3} \cdot (\log N_{\delta})^{2/3}\right)  \leq N_\delta^{-5}.
\end{equation*}
The latter and the tower of inequalities
$$ D_1 N_{\delta}^{1/3} \cdot (\log N_{\delta})^{2/3} \leq 2D_1 \delta^{1/3} N^{1/3} \cdot (\log N_{\delta})^{2/3} \leq (\rho/2) N^{1/3} (\log N_{\delta})^{2/3} \leq (\rho/2) N^{1/3} (\log N)^{2/3},$$
together imply that for $N \geq M_2$ and $(m, n) \in \llbracket 1, N_\delta \rrbracket^2$
\begin{equation}\label{OP3}
 \mathbb{P}\left(\log Z(1,1;m,n) \geq (\rho/2) N^{1/3} \cdot (\log N)^{2/3}\right)  \leq N_\delta^{-5}.
\end{equation}
Let $T^{\delta}_N$ be the set of quadruples $m_1, n_1, m_2, n_2   \in \llbracket 1, N \rrbracket$ such that $(m_1, n_1) \leq (m_2, n_2)$ and $m_2 - m_1 + n_2 - n_1 \leq \lfloor \delta N \rfloor$. Note that $Z(m_1,n_1; m_2, n_2)$ has the same distribution as $Z(m_1 + a,n_1 + b; m_2 + a, n_2 + b)$ for any $a,b \in \mathbb{Z}$ and also that if $(m_1, n_1, m_2, n_2) \in  T^{\delta}_N$, then $(1, 1; m_2 - m_1 +1, n_2 - n_1 + 1) \in \llbracket 1, N_\delta \rrbracket^2$. The latter, (\ref{OP3}) and a union bound imply for $N \geq M_2$
\begin{equation}\label{OP4}
\begin{split}
&\mathbb{P}\left(B_N \right)  \leq N^4 N_\delta^{-5}, \mbox{ and consequently } \lim_{N \rightarrow \infty} \mathbb{P}(B_N) = 0,  \\
&\mbox{ where } B_N = \cup_{(m_1,n_1,m_2, n_2) \in T_N^{\delta}} \left\{\log Z(m_1,n_2;m_2,n_2) \geq (\rho/2) N^{1/3} \cdot (\log N)^{2/3} \right\}.
\end{split}
\end{equation}
We mention that in the first inequality in (\ref{OP4}) we used that $|T_N^{\delta}| \leq N^4$. 

Using the equality of events $\{\LP_N \leq \delta N\} = \{(\XS_N ,\YS_N, \XE_N,\YE_N) \in T_N^{\delta} \}$ and (\ref{LC0Copy}) we see that 
\begin{equation}\label{OP5}
\begin{split}
&\limsup_{N \rightarrow \infty} \mathbb{P}(   \LP_N \leq \delta N ) \leq \limsup_{N \rightarrow \infty} \mathbb{P}(   \{ \LP_N \leq \delta N \} \cap A_N^c \cap B_N^c) +\limsup_{N \rightarrow \infty} \mathbb{P}(A_N) + \mathbb{P}(B_N) =   \\
& \limsup_{N \rightarrow \infty} \mathbb{E} \left[  \frac{\sum_{(m_1,n_1,m_2,n_2) \in T_N^{\delta} } Z(m_1,n_1;m_2;n_2) }{Z_N^w} \cdot {\bf 1}_{A^c_N} \cdot {\bf 1}_{B^c_N}\right].
\end{split}
\end{equation}
In the last equality we used $ \limsup_{N \rightarrow \infty} \mathbb{P}(A_N) + \mathbb{P}(B_N)  = 0$, which holds from (\ref{OP1}) and (\ref{OP4}).

On the event $A^c_N$ we have 
$$Z_N^w \geq Z(1,1;N,N) \geq \exp( \mathsf{F}_N) \geq \exp \left( \rho N^{1/3} (\log N)^{2/3} \right),$$
and on the event $B_N^c$ we have for $(m_1,n_1,m_2,n_2) \in T_N^{\delta}$
$$Z(m_1,n_1;m_2;n_2)  \leq \exp\left( (\rho/2)N^{1/3} (\log N)^{2/3} \right).$$
Combining the latter two statements, we see 
\begin{equation*}
\begin{split}
&\limsup_{N \rightarrow \infty} \mathbb{E} \left[  \frac{\sum_{(m_1,n_1,m_2,n_2) \in T_N^{\delta} } Z(m_1,n_1;m_2;n_2) }{Z_N^w} \cdot {\bf 1}_{A^c_N} \cdot {\bf 1}_{B^c_N}\right] \leq \\
& \limsup_{N \rightarrow \infty} \mathbb{E} \left[  \frac{|T_N^{\delta}| \cdot  \exp\left( (\rho/2)N^{1/3} (\log N)^{2/3} \right) }{\exp \left( \rho N^{1/3} (\log N)^{2/3} \right)} \cdot {\bf 1}_{A^c_N} \cdot {\bf 1}_{B^c_N}\right] \leq \limsup_{N \rightarrow \infty} N^4 \cdot e^{-(\rho/2) N^{1/3} (\log N)^{2/3}} = 0,
\end{split}
\end{equation*}
which together with (\ref{OP5}) establishes (\ref{OP2}).\\

In the remainder of this section we prove that $\frac{N (\SP_N - 1)}{ N^{2/3 + \epsilon}}$ converges to $0$ in probability as $N \rightarrow \infty$. We will prove that for any $\delta_1 \in (0,1/3)$
\begin{equation}\label{UP1}
\lim_{N \rightarrow \infty} \mathbb{P}\left(  S_N \leq \exp( - N^{-1/3 + \delta_1})\right) = 0 \mbox{ and } \lim_{N \rightarrow \infty} \mathbb{P} \left(  S_N \geq \exp(  N^{-1/3 + \delta_1}) \right) = 0.
\end{equation}
Notice that the last two statements with $\delta_1 = \min(\epsilon/2, 1/4)$ imply the weak convergence of $\frac{N (\SP_N - 1)}{ N^{2/3 + \epsilon}}$ to $0$ as $N \rightarrow \infty$. Moreover, since $\SP_N$ and $1/\SP_N$ have the same distribution (this follows from the distributional invariance of the log-gamma polymer upon exchanging the horizontal and vertical axes) it suffices to show the first inequality in (\ref{UP1}). The first inequality in (\ref{UP1}) follows from equations (\ref{UP2}) and (\ref{UP3}), which we state and prove in the following two steps.\\

{\bf \raggedleft Step 1.} In this step we prove that 
\begin{equation}\label{UP2}
\lim_{N \rightarrow \infty} \mathbb{P}\left(  S_N \leq 1/2 \right) = 0.
\end{equation}
Let $\epsilon = \min\left(1,  h_{\theta_c}(1/2)\right)$, which in view of (\ref{S2SummH2}) is a positive number. From Lemma \ref{S2S2L1} we can find positive constants $B, b > 0$ such that for $m, n \in \mathbb{N}$, $m \geq n$
\begin{equation*}
\mathbb{P}(\log Z(1,1;m ,n) \geq m\cdot \epsilon - m \cdot h_{\theta_c}(n/m)) \leq B \exp(-bm).
\end{equation*}
The last equation combined with the definition of $\epsilon$, and (\ref{S2SummH2}) allows us to conclude that for $m \geq n$ and $n/m \in (0,1/2]$
\begin{equation}\label{FK1}
\begin{split}
&B \exp(-bm) \geq \mathbb{P}(\log  Z(1,1;m ,n) \geq m \cdot \epsilon - m\cdot h_{\theta_c}(n/m)) \geq \\
&  \mathbb{P}(\log  Z(1,1;m ,n) \geq m \cdot \epsilon - m \cdot h_{\theta_c}(1/2)) \geq  \mathbb{P}(  Z(1,1;m ,n) \geq 1).
\end{split}
\end{equation}

Let $U_N$ be the set of quadruples $m_1, n_1, m_2, n_2 \in \llbracket 1, N \rrbracket$ such that $(m_1, n_1) \leq (m_2, n_2)$ and $\frac{n_2 - n_1 + 1}{m_2 - m_1 + 1} \in (0,1/2]$. Let $\delta > 0$ be such that (\ref{OP2}) holds (recall that we proved (\ref{OP2}) earlier in this section), and let $T_N^{\delta}$ be as in the paragraph above (\ref{OP4}).
Finally, let $\rho > 0$ and the events $ A_N$ be as in (\ref{OP1}).

Note that $Z(m_1,n_1; m_2, n_2)$ has the same distribution as $Z(m_1 + a,n_1 + b; m_2 + a, n_2 + b)$ for any $a,b \in \mathbb{Z}$ and also that if $(m_1, n_1, m_2, n_2) \in  U_N$, then $m_2 - m_1 + 1 \geq n_2 - n_1 + 1$, $(n_2- n_1+1)/(m_2 - m_1+1) \in (0,1/2]$. The latter and a union bound of (\ref{FK1}) implies that
\begin{equation}\label{FK2}
\begin{split}
&\mathbb{P}(C_N) \leq B N^4 \exp(-\delta b N/2 ) \mbox{, and consequently $\lim_{N \rightarrow \infty} \mathbb{P}(C_N) = 0$,}\\
& \mbox{ where } C_N = \cup_{(m_1,n_1,m_2,n_2) \in U_N \setminus T_N^{\delta}} \{ Z(m_1, n_1; m_2, n_2) \geq 1\}.
\end{split}
\end{equation}
We mention that in the first inequality in (\ref{FK2}) we used that $|U_N\setminus T_N^{\delta}| \leq N^4$ and that for $ (m_1,n_1,m_2,n_2) \in U_N \setminus T_N^{\delta}$ we have $m_2 - m_1 +1 \geq \delta N/2$. 

Using the equality of events $\{S_N \leq 1/2  \} = \{(\XS_N ,\YS_N, \XE_N,\YE_N) \in U_N \}$ and (\ref{LC0Copy}), we see that 
\begin{equation}\label{FK3}
\begin{split}
&\limsup_{N \rightarrow \infty} \mathbb{P}(   S_N \leq 1/2 ) = \limsup_{N \rightarrow \infty} \mathbb{P}(   \{ S_N \leq 1/2 \} \cap A_N^c \cap C_N^c)  =   \\
&   \limsup_{N \rightarrow \infty} \mathbb{E} \left[  \frac{\sum_{(m_1,n_1,m_2,n_2) \in U_N } Z(m_1,n_1;m_2;n_2) }{Z_N^w} \cdot {\bf 1}_{A^c_N} \cdot {\bf 1}_{C^c_N}\right].
\end{split}
\end{equation}
In the first equality we used $ \lim_{N \rightarrow \infty} \mathbb{P}(A_N) +  \mathbb{P}(C_N) = 0$, which holds from (\ref{OP1}) and (\ref{FK2}).

We also know from (\ref{LC0Copy}) and (\ref{OP2}) that 
\begin{equation*}
\begin{split}
&   \lim_{N \rightarrow \infty} \mathbb{E} \left[  \frac{\sum_{(m_1,n_1,m_2,n_2) \in T_N^{\delta} } Z(m_1,n_1;m_2;n_2) }{Z_N^w}\right] =\lim_{N \rightarrow \infty} \mathbb{P}(\LP_N \leq \delta N ) = 0.
\end{split}
\end{equation*}
Combining the above two statements, we get
\begin{equation}\label{FK4}
\begin{split}
&\limsup_{N \rightarrow \infty} \mathbb{P}(   S_N \leq 1/2 ) =   \limsup_{N \rightarrow \infty} \mathbb{E} \left[  \frac{\sum_{(m_1,n_1,m_2,n_2) \in U_N\setminus T_N^{\delta} } Z(m_1,n_1;m_2;n_2) }{Z_N^w} \cdot {\bf 1}_{A^c_N} \cdot {\bf 1}_{C^c_N}\right].
\end{split}
\end{equation}
On the event $A^c_N$ we have 
$$Z_N^w \geq Z(1,1;N,N) \geq \exp( \mathsf{F}_N) \geq \exp \left( \rho N^{1/3} (\log N)^{2/3} \right),$$
and on the event $C_N^c$ we have for $(m_1,n_1,m_2,n_2) \in U_N \setminus T_N^{\delta}$
$$Z(m_1,n_1;m_2;n_2)  \leq 1.$$
Combining the latter statements, we see 
\begin{equation*}
\begin{split}
& \limsup_{N \rightarrow \infty} \mathbb{E} \left[  \frac{\sum_{(m_1,n_1,m_2,n_2) \in U_N \setminus T_N^{\delta} } Z(m_1,n_1;m_2;n_2) }{Z_N^w} \cdot {\bf 1}_{A^c_N} \cdot {\bf 1}_{C^c_N}\right] \leq \\
& \limsup_{N \rightarrow \infty} \mathbb{E} \left[  \frac{|U_N\setminus T_N^{\delta}| }{\exp \left( \rho N^{1/3} (\log N)^{2/3} \right)} \cdot {\bf 1}_{A^c_N} \cdot {\bf 1}_{C^c_N}\right] \leq \limsup_{N \rightarrow \infty} N^4 \cdot e^{-\rho N^{1/3} (\log N)^{2/3}} = 0,
\end{split}
\end{equation*}
which together with (\ref{FK4}) establishes (\ref{UP2}).\\

{\bf \raggedleft Step 2.} In this step we prove that for $\delta_1 \in (0,1/3)$ we have
\begin{equation}\label{UP3}
\lim_{N \rightarrow \infty} \mathbb{P}\left(  1/2 < S_N \leq \exp( - N^{-1/3 + \delta_1}) \right) = 0.
\end{equation}

From Proposition \ref{S1LDE1} (applied to $\delta = 1/2$) there exist constants $D_1, d_1, D_2, d_2 > 0$ such that for all $m,n \in \mathbb{N}$ with $m \geq n$, $n/m \in [1/2,1]$ and $x \geq 0$
\begin{equation}\label{GK1}
\mathbb{P} \left(\log Z(1,1,;m,n) + m h_{\theta_c} (n/m) \geq x m^{1/3} \sigma_{\theta_c}(n/m) \right) \leq D_1 e^{-d_1 m} + D_2 e^{-d_2x^{3/2}}.
\end{equation}

As explained in the line above (\ref{S2SummH2}), we know that $h_{\theta_c}$ is strictly decreasing on $[0, 1]$ and $h_{\theta_c}(1) = 0$. In addition, as explained in Section \ref{Section2.1} (see the discussion after Proposition \ref{ThmTight}), we know that $\sigma_{\theta_c}(y) $ is positive and smooth on $(0,\infty)$. In particular, we can find $\alpha \in (0,1)$ such that 
$$\alpha^{-1} \geq \sigma_{\theta_c}(y) \geq \alpha \mbox{ if } y \in [1/2,1].$$
Setting $x = \frac{m^{2/3} \cdot h_{\theta_c} (n/m) }{\sigma_{\theta_c}(n/m)}$ in (\ref{GK1}), we conclude that for all $m,n \in \mathbb{N}$ with $m \geq n$, $n/m \in [1/2,1]$
\begin{equation}\label{GK2}
\mathbb{P} \left( Z(1,1;m,n) \geq 1  \right) \leq D_1 e^{-d_1 m} + D_2 e^{-d_2 \alpha^{3/2} m  [h_{\theta_c} (n/m)]^{3/2}}.
\end{equation}

Let $\delta > 0$ be such that (\ref{OP2}) holds (recall that we proved (\ref{OP2}) earlier in this section), and let $T_N^{\delta}$ be as in the paragraph above (\ref{OP4}). In addition, let $U_N^{\delta_1}$ be the set of quadruples $m_1, n_1, m_2, n_2   \in \llbracket 1, N \rrbracket$ such that $(m_1, n_1) \leq (m_2, n_2)$ and $\frac{n_2 - n_1 + 1}{m_2 - m_1 + 1} \in \left(1/2,\exp\left(-N^{-1/3 + \delta_1}\right) \right]$. Notice that if $(m_1,n_1, m_2, n_2) \in U_N^{\delta_1}$ we have that $m = m_2 - m_1 + 1$, $n = n_2 - n_1 +1$ satisfy the conditions above (\ref{GK2}). Since $Z(m_1, n_1;m_2, n_2)$ has the same distribution as $Z(1,1;m,n)$ we conclude that (\ref{GK2}) holds with $Z(1,1;m,n)$ replaced with $Z(m_1, n_1;m_2, n_2)$, and $m = m_2 - m_1 + 1$, $n = n_2 - n_1 +1$. Furthermore as $h_{\theta_c}$ is strictly decreasing on $[0, 1]$ we have 
$$h_{\theta_c} (n/m) \geq h_{\theta_c} \left (\exp\left(-N^{-1/3 + \delta_1}\right)  \right) =  \frac{h_{\theta_c}''(1)}{2} \cdot N^{-2/3 + 2\delta_1} + O(N^{-1 + 3\delta_1}),$$
where in the last equality we used (\ref{S2SummH}). This gives for $(m_1, n_2, m_2, n_2) \in U_N^{\delta_1}$ 
\begin{equation}\label{GK3}
\begin{split}
&\mathbb{P} \left( Z(m_1,n_1; m_2,n_2) \geq 1 \right) \leq D_1 e^{-d_1 (m_2 - m_1 + 1)} +  \\
&D_2 \exp \left( -d_2 \alpha^{3/2}  (m_2 - m_1 + 1) \left( h_{\theta_c}''(1) /2  \right)^{3/2} \cdot N^{-1 + 3\delta_1}  + O( (m_2 - m_1 + 1) N^{-4/3 + 4\delta_1} )\right).
\end{split}
\end{equation}
From (\ref{GK3}) we can conclude the existence of constants $D, d> 0$ such that  if $(m_1, n_1,m_2, n_2) \in U_N^{\delta_1} \setminus T_N^{\delta}$  (which implies $N \geq m_2 - m_1 + 1 \geq \delta N/2$) we have
\begin{equation}\label{GK4}
\begin{split}
&\mathbb{P} \left( Z(m_1,n_1; m_2,n_2) \geq 1  \right) \leq D e^{-d N^{3\delta_1}}.
\end{split}
\end{equation}
The latter and a union bound gives 
\begin{equation}\label{GK5}
\begin{split}
&\mathbb{P}(C^{\delta_1}_N) \leq D N^4 \exp(- d N^{3\delta_1} ) \mbox{, and consequently $\lim_{N \rightarrow \infty} \mathbb{P}(C^{\delta_1}_N) = 0$,}\\
& \mbox{ where } C^{\delta_1}_N = \cup_{(m_1,n_1,m_2,n_2) \in U^{\delta_1}_N \setminus T_N^{\delta}} \{ Z(m_1, n_1; m_2, n_2) \geq 1\}.
\end{split}
\end{equation}
From here we can repeat verbatim the argument from Step 1 above, starting from (\ref{FK3}), to conclude (\ref{UP3}). In that argument we only need to replace the events $\{S_N \leq 1/2  \}$ with $\{1/2 < S_N  \leq \exp( - N^{-1/3 + \delta_1} ) \} =  \{(\XS_N ,\YS_N, \XE_N,\YE_N) \in U^{\delta_1}_N \}$, the sets $U_N$ with $U_N^{\delta_1}$, and the events $C_N$ with $C^{\delta_1}_N$.

%
\subsection{Proof of Proposition \ref{PMP3}} In this section we prove Proposition \ref{PMP3}. Proposition \ref{PMP3} contains two types of statements: about the asymptotic distribution of the length $\LP_N$ of a path sampled according to the log-gamma polymer measure, and about the asymptotic distribution of the starting point $(\XS_N, \YS_N)$ of such a path. We will split the proofs of these two types of statements into Sections \ref{PMP3Proof1} and \ref{PMP3Proof2}.  Throughout this section we assume that $\theta > \theta_c$.

%
\subsubsection{Asymptotics of $\LP_N$}\label{PMP3Proof1} In this section we prove (\ref{SuperCritLength}) for all $\theta > \theta_c$ and that $\LP_N$ is tight for $\theta > 3$. For clarity we split the proof into three steps.\\

{\bf \raggedleft Step 1.} In this step we summarize the technical inputs we will need in the other steps and specify the value of $c_2$ in the statement of the proposition. 

From our discussion at the end of Section \ref{Section2.1} we know that $g_{\theta}^{-1}$ is an increasing bijection from $(0,\infty)$ to $(0, \theta)$. Since $\theta > \theta_c$, we can find a (unique) $\phi_c$ (depending on $\theta$) such that $g_{\theta}(\phi_c) = \theta_c/2$. Moreover, since $g^{-1}_{\theta}(1) = \theta/2 > \theta_c/2$, we can conclude that $\phi_c \in (0,1)$. If $h_{\theta}$ is as in (\ref{HDefLLN}), we know from Section \ref{Section2.1} that $h_{\theta}'(x) = \Psi(g_{\theta}^{-1}(x))$ and is a strictly increasing function on $(0,\infty)$. We also know that $h_{\theta}'(\phi_c) = \Psi(\theta_c/2) = 0$ and so $h_{\theta}(x)$ is decreasing on $[0, \phi_c]$ and increasing on $[\phi_c,1]$ so that it attains its minimum on $[0,1]$ at $\phi_c$. We also know that 
\begin{equation}\label{TR1}
\kappa := \min_{[0,1]} h_{\theta}(x) = h_{\theta}(\phi_c) = (\phi_c) \cdot \Psi(\theta_c/2) + \Psi(\theta -\theta_c/2) =  \Psi(\theta -\theta_c/2) > 0,
\end{equation}
where we used (\ref{HDefLLN}) and the fact that $\Psi(\theta_c/2) = 0$, while $\Psi(x) > 0$ for $x > \theta_c/2$.

Let $B, b > 0$ be as in Lemma \ref{S2S2L1} for $\epsilon = \min(\kappa/2,1)$. With this choice of parameters, we define 
\begin{equation}\label{TR2}
c_2 : = \max(8/\kappa, 10/b).
\end{equation}
This is the choice of $c_2$ with which we will prove (\ref{SuperCritLength}) in the next step.\\

The last thing we do in this step is record the following statement
\begin{equation}\label{TR3}
\lim_{N \rightarrow \infty} \mathbb{P} \left( A_N \right) = 0 \mbox{, where } A_N = \left\{ \frac{1}{N^2}\sum_{i,j = 1}^N w_{i,j} \leq \frac{1}{2(\theta -1)}  \right\},
\end{equation}
and $w_{i,j}$ are as in Definition \ref{Def1}. Equation (\ref{TR3}) follows from the Law of large numbers and the fact that $\mathbb{E}[w_{1,1}] = (\theta - 1)^{-1}$.\\

{\bf \raggedleft Step 2.} In this step we prove (\ref{SuperCritLength}). Recall from Step 1, that $B, b > 0$ are as in Lemma \ref{S2S2L1}  for $\epsilon = \min(\kappa/2,1)$. Suppose that $m \geq n \geq 1$ and $m \geq (c_2/2) \cdot \log N$. From Lemma \ref{S2S2L1} we know 
$$B \exp(-bm) \geq \mathbb{P}(\log  Z(1,1;m ,n) \geq m \cdot \epsilon - m\cdot h_{\theta}(n/m)).$$
Since $h_{\theta}(n/m) \geq \kappa$ as in (\ref{TR1}) and $\epsilon \leq \kappa/2$, we conclude that 
$$B \exp(-bm) \geq \mathbb{P}(\log  Z(1,1;m ,n) \geq  - (\kappa/2) m ).$$
Since $m \geq (c_2/2) \cdot \log N$ and $c_2/2 \geq  4/\kappa$ and $c_2/2 \geq 5/b$ by (\ref{TR2}) we conclude that
\begin{equation}\label{TR4}
BN^{-5} \geq \mathbb{P}( Z(1,1;m ,n) \geq  N^{-4}).
\end{equation}
Let $T_N$  be the set of quadruples $m_1, n_1, m_2, n_2   \in \llbracket 1, N \rrbracket$ such that $(m_1, n_1) \leq (m_2, n_2)$ and $m_2 - m_1 + n_2 - n_1 \geq c_2 \log N$. Using that $Z(m_1,n_1; m_2, n_2)$ has the same distribution as $Z(m_1 + a,n_1 + b; m_2 + a, n_2 + b)$ for any $a,b \in \mathbb{Z}$ and also that $Z(1,1;m,n)$ has the same distribution as $Z(1,1;n,m)$, we conclude from (\ref{TR4}) that for $(m_1, n_1, m_2, n_2) \in T_N$ we have
\begin{equation*}
BN^{-5} \geq \mathbb{P}( Z(m_1, n_1; m_2, n_2) \geq  N^{-4}).
\end{equation*}
In particular, taking a union bound and using that $|T_N| \leq N^4$ we conclude 
\begin{equation}\label{TR5}
\lim_{N \rightarrow \infty} \mathbb{P}( B_N) = 0 \mbox{, where } B_N = \cup_{(m_1, n_1, m_2, n_2) \in T_N}  \{  Z(m_1, n_1; m_2, n_2) \geq  N^{-4} \}.
\end{equation}

Using the equality of events $\{\LP_N \geq c_2 \log N\} = \{(\XS_N ,\YS_N, \XE_N,\YE_N) \in T_N\}$ and (\ref{LC0Copy}) we see that 
\begin{equation}\label{TR6}
\begin{split}
&\limsup_{N \rightarrow \infty} \mathbb{P}(   \LP_N \geq c_2 \log N ) = \limsup_{N \rightarrow \infty} \mathbb{P}(   \{ \LP_N \geq c_2 \log N \} \cap A_N^c \cap B_N^c) =   \\
& \limsup_{N \rightarrow \infty} \mathbb{E} \left[  \frac{\sum_{(m_1,n_1,m_2,n_2) \in T_N } Z(m_1,n_1;m_2;n_2) }{Z_N^w} \cdot {\bf 1}_{A^c_N} \cdot {\bf 1}_{B^c_N}\right].
\end{split}
\end{equation}
In the first equality we used $ \lim_{N \rightarrow \infty} \mathbb{P}(A_N) + \mathbb{P}(B_N)  = 0$, which holds from (\ref{TR3}) and (\ref{TR5}).

On the event $A^c_N$ we have 
$$Z_N^w \geq \sum_{i,j = 1}^N w_{i,j} \geq \frac{N^2}{2(\theta -1)},$$
and on the event $B_N^c$ we have for $(m_1,n_1,m_2,n_2) \in  T_N$
$$Z(m_1,n_1;m_2;n_2)  \leq N^{-4}.$$
Combining the latter statements, we see 
\begin{equation*}
\begin{split}
& \limsup_{N \rightarrow \infty} \mathbb{E} \left[  \frac{\sum_{(m_1,n_1,m_2,n_2) \in T_N } Z(m_1,n_1;m_2;n_2) }{Z_N^w} \cdot {\bf 1}_{A^c_N} \cdot {\bf 1}_{B^c_N}\right] \leq \\
& \limsup_{N \rightarrow \infty} \mathbb{E} \left[  \frac{| T_N| \cdot N^{-4}\cdot 2 (\theta-1) }{N^2} \cdot {\bf 1}_{A^c_N} \cdot {\bf 1}_{B^c_N}\right] \leq \limsup_{N \rightarrow \infty} \frac{2 (\theta-1)}{N^2}  = 0,
\end{split}
\end{equation*}
which together with (\ref{TR6}) establishes (\ref{SuperCritLength}).\\

{\bf \raggedleft Step 3.} In this last step we show that $\LP_N$ is tight when $\theta > 3$. \\

Fix $(m_1,n_1) \leq (m_2, n_2)$ with $m_1, n_1, m_2, n_2 \in \mathbb{N}$. Using that $Z(m_1,n_1; m_2, n_2)$ has the same distribution as $Z(m_1 + a,n_1 + b; m_2 + a, n_2 + b)$ for any $a,b \in \mathbb{Z}$ and (\ref{S2ExpBound}) we conclude 
\begin{equation*}
\mathbb{E}\left[ Z(m_1,n_1; m_2, n_2)  \right] \leq \left(\frac{2}{\theta-1} \right)^{m_2 - m_1 + n_2 - n_1 + 1}.
\end{equation*}
The last equation implies that 
\begin{equation}\label{TR7}
\begin{split}
 &\sum_{\substack{(m_1, n_1) \leq (m_2,n_2)\\  m_1, n_1, m_2, n_2 \in \llbracket 1, N \rrbracket}} \mathbb{E}\left[ (m_2 - m_1 + n_2 - n_1) Z(m_1,n_1; m_2, n_2)  \right] \leq \\
& \sum_{\substack{(m_1, n_1) \leq (m_2,n_2)\\  m_1, n_1, m_2, n_2 \in \llbracket 1, N \rrbracket}} (m_2 - m_1 + n_2 - n_1) \left(\frac{2}{\theta-1} \right)^{m_2 - m_1 + n_2 - n_1 + 1} \leq  \lambda_{\theta} N^2, \mbox{ where }\\
& \lambda_{\theta} =  \sum^{\infty}_{a,b = 0} (a+b) \cdot \left(\frac{2}{\theta-1} \right)^{a + b + 1} \in (0,\infty).
\end{split}
\end{equation}
We mention that $\lambda_{\theta} \in (0,\infty)$, since $\theta > 3$ and hence $2 (\theta - 1)^{-1} \in (0,1)$.\\

Let $\epsilon > 0$ be given. We can find $M \in \mathbb{N}$ such that 
\begin{equation}\label{TR8}
\frac{2\lambda_{\theta}(\theta - 1)}{M} < \epsilon.
\end{equation}
Let $T_N^M$ be the set of quadruples $m_1, n_1, m_2, n_2   \in \llbracket 1, N \rrbracket$ such that $(m_1, n_1) \leq (m_2, n_2)$ and $m_2 - m_1 + n_2 - n_1 \geq M$. Using the equality of events $\{\LP_N \geq M \} = \{(\XS_N ,\YS_N, \XE_N,\YE_N) \in T^M_N\}$ and (\ref{LC0Copy}) we see that 
\begin{equation}\label{TR9}
\begin{split}
&\limsup_{N \rightarrow \infty} \mathbb{P}(   \LP_N \geq M ) = \limsup_{N \rightarrow \infty} \mathbb{P}(   \{ \LP_N \geq M \} \cap A_N^c) =   \\
& \limsup_{N \rightarrow \infty} \mathbb{E} \left[  \frac{\sum_{(m_1,n_1,m_2,n_2) \in T^M_N } Z(m_1,n_1;m_2;n_2) }{Z_N^w} \cdot {\bf 1}_{A^c_N} \right].
\end{split}
\end{equation}
where $A_N$ is as in (\ref{TR3}) and we used from the same equation that $ \lim_{N \rightarrow \infty} \mathbb{P}(A_N) = 0$.

On the event $A^c_N$ we have 
$$Z_N^w \geq \sum_{i,j = 1}^N w_{i,j} \geq \frac{N^2}{2(\theta -1)},$$
and so from (\ref{TR9}) we get
$$\limsup_{N \rightarrow \infty} \mathbb{P}(   \LP_N \geq M ) \leq \limsup_{N \rightarrow \infty} \frac{2(\theta -1)}{N^2} \cdot \mathbb{E} \left[ \sum_{(m_1,n_1,m_2,n_2) \in T^M_N} Z(m_1,n_1;m_2;n_2) \right].$$
For $(m_1,n_1,m_2,n_2) \in T^M_N$ we have $m_2 - m_1 + n_2 - n_1 \geq M$ and so 
\begin{equation}\label{TR10}
\begin{split}
&\limsup_{N \rightarrow \infty} \mathbb{P}(   \LP_N \geq M ) \leq \\
& \limsup_{N \rightarrow \infty} \frac{2(\theta -1)}{N^{2} M} \cdot  \sum_{(m_1,n_1,m_2,n_2) \in T^M_N}   \mathbb{E} \left[ (m_2 - m_1 + n_2 - n_1) Z(m_1,n_1;m_2;n_2) \right] \leq \\
& \limsup_{N \rightarrow \infty} \frac{2(\theta -1)}{N^2 M} \cdot   \lambda_{\theta} N^2 = \frac{2\lambda_{\theta}(\theta - 1)}{M} < \epsilon.
\end{split}
\end{equation}
We mention that the second inequality in (\ref{TR10}) used (\ref{TR7}) and in the last inequality we used (\ref{TR8}). As $\epsilon > 0$ was arbitrary we see that (\ref{TR10}) implies that $\LP_N$ is tight.

%
\subsubsection{Asymptotics of $(\XS_N, \YS_N)$}\label{PMP3Proof2}  In this section we prove that for $\theta > \theta_c$ the sequence of random vectors $(N^{-1} \XS_N, N^{-1} \YS_N)$ converges weakly to the uniform measure on $(0,1)^2$. For clarity we split the proof into two steps.\\

{\bf \raggedleft Step 1.} We claim that if $x,y \in [0,1]$ and $\Delta_x, \Delta_y \in (0,1/4]$ are such that $[x, x+ \Delta_x] \times [y, y + \Delta_y] \subseteq [0,1]^2$, then 
\begin{equation}\label{WY1}
\limsup_{N \rightarrow \infty} \mathbb{P} \left( x \leq N^{-1} \XS_N \leq x+ \Delta_x, y \leq N^{-1} \YS_N \leq y+ \Delta_y  \right) \leq \frac{1}{ \left( \lfloor \Delta_x^{-1} \rfloor -3 \right) \cdot \left(\lfloor \Delta_y^{-1} \rfloor -3 \right)}.
\end{equation}
We prove (\ref{WY1}) in the second step. Here we assume its validity and conclude the proof of the proposition.\\

Suppose that $[a, a + \Delta_a] \times [b, b + \Delta_b] \subseteq [0,1]$ and $\Delta_a, \Delta_b > 0$. We fix $A, B \in \mathbb{N}$ such that $A, B \geq 2$ and divide $[a, a + \Delta_a] \times [b, b + \Delta_b]$ into $2^{A+B}$ equal rectangles of side lengths $2^{-A} \Delta_a$ and $2^{-B} \Delta_b$. Explicitly, these rectangles are given by 
$$R_{i,j} = [x_i, x_i + 2^{-A} \Delta_a] \times [y_j, y_j+ 2^{-B} \Delta_b] \mbox{, where } x_i = a + (i-1) \cdot 2^{-A} \Delta_a, \hspace{2mm}  y_j = b + (j-1) \cdot 2^{-B} \Delta_b$$
for $i = 1, \dots, 2^A$ and $j = 1, \dots, 2^B$.
We then have 
\begin{equation}\label{WY2}
\begin{split}
&\limsup_{N \rightarrow \infty} \mathbb{P} \left( (N^{-1} \XS_N, N^{-1} \YS_N) \in  [a, a + \Delta_a] \times [b, b + \Delta_b]  \right) \leq \\
& \sum_{i = 1}^{2^A} \sum_{j = 1}^{2^B} \limsup_{N \rightarrow \infty} \mathbb{P} \left( (N^{-1} \XS_N, N^{-1} \YS_N) \in  R_{i,j} \right) \leq \frac{2^{A+B}}{\left( \lfloor 2^A \Delta_a^{-1} \rfloor -3 \right) \cdot \left(\lfloor 2^B \Delta_b^{-1} \rfloor -3 \right)},
\end{split}
\end{equation}
where in the first inequality we used subadditivity and in the second we used (\ref{WY1}) for the rectangles $R_{i,j}$ (which by construction have side-lengths in $(0,1/4])$.

As (\ref{WY2}) holds for all $A, B \geq 2$ we can let $A, B \rightarrow \infty$ to conclude that 
\begin{equation}\label{WY3}
\begin{split}
&\limsup_{N \rightarrow \infty} \mathbb{P} \left( (N^{-1} \XS_N, N^{-1} \YS_N) \in  [a, a + \Delta_a] \times [b, b + \Delta_b]  \right) \leq \Delta_a \cdot \Delta_b.
\end{split}
\end{equation}

To prove that $(N^{-1} \XS_N, N^{-1} \YS_N)$ converges weakly to the uniform measure on $(0,1)^2$ it suffices to show that for all $(s,t) \in (0,1)^2$ we have 
\begin{equation}\label{WY4}
\begin{split}
&\lim_{N \rightarrow \infty} \mathbb{P} \left( N^{-1} \XS_N \leq s, N^{-1} \YS_N \leq t  \right) = s \cdot t.
\end{split}
\end{equation}
From (\ref{WY3}) applied to $a = 0, \Delta_a = s$, $b = 0, \Delta_b = t$ we see that 
\begin{equation}\label{WY5}
\begin{split}
&\limsup_{N \rightarrow \infty} \mathbb{P} \left( N^{-1} \XS_N \leq s, N^{-1} \YS_N \leq t  \right) \leq s \cdot t.
\end{split}
\end{equation}
On the other hand, 
\begin{equation*}
\begin{split}
&1 - \mathbb{P} \left( N^{-1} \XS_N \leq s, N^{-1} \XS_N \leq t  \right)  \leq  \\
&\mathbb{P} \left( (N^{-1} \XS_N, N^{-1} \YS_N) \in [s, 1] \times [0,1]  \right) +  \mathbb{P} \left( (N^{-1} \XS_N, N^{-1} \YS_N) \in [0, s] \times [t,1]  \right) .
\end{split}
\end{equation*}
Taking $\limsup_{N \rightarrow \infty}$ on both sides and using (\ref{WY5}) we conclude 
\begin{equation}\label{WY6}
\begin{split}
&1 - \liminf_{N \rightarrow \infty} \mathbb{P} \left( N^{-1} \XS_N \leq s, N^{-1} \YS_N \leq t  \right) \leq (1- s) + (1-t) s  \iff \\
& \liminf_{N \rightarrow \infty} \mathbb{P} \left( N^{-1} \XS_N \leq s, N^{-1} \XS_N \leq t  \right)  \geq s\cdot t.
\end{split}
\end{equation}
Equations (\ref{WY5}) and (\ref{WY6}) together imply (\ref{WY4}), which concludes the proof of the proposition.\\

{\bf \raggedleft Step 2.} In this step we prove (\ref{WY1}). We begin by introducing some useful notation. We define
$$A_N^x = \lceil N x \rceil, \hspace{2mm} B_N^x  = \lfloor N (x + \Delta_x) \rfloor, \hspace{2mm} A_N^y = \lceil N y \rceil, \hspace{2mm} B_N^y = \lfloor N(y + \Delta_y) \rfloor,$$
$$ \bar{B}_N^x = \min(B_N^x + \lfloor N^{1/4} \rfloor, N), \hspace{1mm} \bar{B}_N^y = \min(B_N^y + \lfloor N^{1/4} \rfloor, N), \hspace{1mm} D_N^x = \bar{B}_N^x - A_N^x+1, \hspace{1mm}D_N^y = \bar{B}_N^y- A_N^y+1 .$$
We further define $S_N^x$ and $S_N^y$ through
$$ S^{x/y}_N =\{ z \in \mathbb{Z}: N \geq A_N^{x/y} + z \cdot D_N^{x/y} \geq 1 \mbox{ and }  N \geq \bar{B}_N^{x/y} + z \cdot D_N^{x/y} \geq 1 \}.$$
Observe that both $S_N^x$ and $S_N^y$ contain $0$ and constitute a finite sequence of consecutive integers. Denoting the smallest and largest elements in $S^{x/y}_N$ by $P^{x/y}_N$ and $Q^{x/y}_N$, respectively, we have 
$$S_N^{x/y} = \llbracket P^{x/y}_N , Q^{x/y}_N \rrbracket.$$
For $(p,q) \in S_N^x \times S_N^y$ we let 
$$R^{p,q}_N = \{(m,n) \in \mathbb{N}: \bar{B}_N^x + p D_N^x \geq m \geq A_N^x + p D_N^x \mbox{ and } \bar{B}_N^y + q D_N^y \geq n \geq A_N^y + q D_N^y \}.$$

We summarize the following observations, which immediately follow from the above definition
\begin{enumerate}
\item We have the equality of events: 
$$\left\{ (N^{-1} \XS_N, N^{-1} \YS_N) \in  [x, x + \Delta_x] \times [y, y + \Delta_y] \right\} = \left\{(\XS_N, \YS_N) \in \llbracket A_N^x, B_N^x \rrbracket \times \llbracket A_N^y, B_N^y \rrbracket \right\}.$$
\item $R_N^{p,q}$ for $(p,q) \in S_N^x \times S_N^y$ are pairwise disjoint.
\item We have the inclusion of events:
$$ \left\{(\XS_N, \YS_N) \in \llbracket A_N^x, B_N^x \rrbracket \times \llbracket A_N^y, B_N^y \rrbracket , (\XE_N, \YE_N)  \not \in R_N^{0,0} \right\} \subseteq \{\LP_N \geq \lfloor N^{1/4} \rfloor \}.$$
\item We have the inequalities
$$(Q_N^{x/y} + 1) \cdot D_N^{x/y} + \bar{B}_N^{x/y} \geq N+1 \mbox{ and } (P_N^{x/y} - 1) \cdot D_N^{x/y} + A_N^{x/y} \leq 0.$$
\end{enumerate}
We also note that for all large $N$ we have
\begin{equation}\label{WY7}
|S_N^x| \geq \lfloor \Delta_x^{-1} \rfloor - 3 \mbox{ and }|S_N^y| \geq \lfloor \Delta_y^{-1} \rfloor - 3.
\end{equation}
Let us prove the latter briefly, and as the two inequalities are analogous focus only on the first. By definition, we have $|S_N^{x}| = Q_N^x - P_N^x + 1$ and so the fourth observation above gives
$$|S_N^x| \geq \left(\frac{N + 1 - \bar{B}_N^x}{D_N^x} - 1 \right) - \left(1 -  \frac{A_N^x}{ D_N^x} \right) + 1 =  \frac{N+1 - \bar{B}_N^x + A_N^x}{D_N^x} - 1.$$
By definition, we have that 
$$\lim_{N \rightarrow \infty} \frac{N+1 - \bar{B}_N^x + A_N^x}{D_N^x} = \frac{1 - (x+\Delta_x) + x}{\Delta_x} = \frac{1}{\Delta_x} - 1,$$
which together with the last inequality implies the first inequality in (\ref{WY7}) for all large $N$.\\

We now turn to proving (\ref{WY1}). Using the first and third observations above we have
\begin{equation}\label{WY8}
\begin{split}
&\limsup_{N \rightarrow \infty}\mathbb{P} \left( x \leq N^{-1} \XS_N \leq x+ \Delta_x, y \leq N^{-1} \YS_N \leq y+ \Delta_y  \right) = \\
&\limsup_{N \rightarrow \infty}\mathbb{P} \left( (\XS_N, \YS_N) \in \llbracket A_N^x, B_N^x \rrbracket \times \llbracket A_N^y, B_N^y \rrbracket  \right) = \\
&\limsup_{N \rightarrow \infty}\mathbb{P} \left( (\XS_N, \YS_N) \in \llbracket A_N^x, B_N^x \rrbracket \times \llbracket A_N^y, B_N^y \rrbracket , (\XE_N, \YE_N) \in R_N^{0,0}  \right) \leq \\
&\limsup_{N \rightarrow \infty}\mathbb{P} \left( (\XS_N, \YS_N) \in R_N^{0,0} , (\XE_N, \YE_N) \in R_N^{0,0}  \right).
\end{split}
\end{equation}
We mention that in going from the second to the third line we used the third observation above and the fact that 
$$\lim_{N \rightarrow \infty} \mathbb{P}(\LP_N \geq \lfloor N^{1/4} \rfloor) = 0$$
as follows from (\ref{SuperCritLength}), which we proved in the previous section.

We next have from (\ref{LC0Copy}) that
\begin{equation}\label{WY9}
\begin{split}
&\mathbb{P} \left( (\XS_N, \YS_N) \in R_N^{0,0} , (\XE_N, \YE_N) \in R_N^{0,0}  \right)=   \mathbb{E} \left[ \frac{1}{Z_N^w}  \sum_{\substack{(m_1, n_1) \leq (m_2,n_2)\\  (m_1, n_1), (m_2, n_2) \in R_N^{0,0} }} Z(m_1,n_1;m_2;n_2) \right].
\end{split}
\end{equation}

Let us set for $p \in S_N^x$ and $q \in S_N^y$
$$Z^{p,q}_N = \sum_{\substack{(m_1, n_1) \leq (m_2,n_2)\\  (m_1, n_1), (m_2, n_2) \in R_N^{p,q} }} Z(m_1,n_1;m_2;n_2).$$
From the second observation above we know that $R_N^{p,q}$ are pairwise disjoint and so $Z^{p,q}_N$ are independent, but they are also identically distributed as $R_N^{p,q}$ are translates of the same rectangle $R^{0,0}_N$. In addition, by the disjointness of $R_N^{p,q}$ we have that 
$$Z_N^w \geq \sum_{p \in S_N^x, q \in S_N^y} Z_N^{p,q}.$$
Combining the last two observations and (\ref{WY9}) gives 
\begin{equation}\label{WY10}
\begin{split}
&\mathbb{P} \left( (\XS_N, \YS_N) \in R_N^{0,0} , (\XE_N, \YE_N) \in R_N^{0,0}  \right) =   \mathbb{E} \left[ \frac{Z_N^{0,0}}{Z_N^w}  \right] \leq \mathbb{E} \left[ \frac{Z_N^{0,0}}{\sum_{p \in S_N^x, q \in S_N^y} Z_N^{p,q}}  \right]  = \frac{1}{|S_N^x| \cdot |S_N^y|}.
\end{split}
\end{equation}
Combining (\ref{WY7}), (\ref{WY8}) and (\ref{WY10}) proves (\ref{WY1}).

\begin{appendix}
%
\section{Upper tail estimates} \label{Section6} Here we prove Proposition \ref{S2LBProp}. We continue with the same notation as in Sections \ref{Section1} and \ref{Section2.1}.

%
\subsection{Estimates on Fredholm determinants} \label{Section6.1}
 In this section we recall a result from \cite{BCDA}, see Proposition \ref{LGPT1}, which expresses the Laplace transform of $Z(1,1;M,N)$ from Section \ref{Section1} in terms of a {\em Fredholm determinant}. In Proposition \ref{S6Main} we show that this Laplace transform is close to the GUE Tracy-Widom distribution and we use this statement to establish Proposition \ref{S2LBProp}. We record here that for a measurable function $K(x,y)$ on $X \times X$, where $(X, \mu)$ is a measure space, we define
\begin{equation}\label{fredholmDefS6}
\det(I + K)_{L^2(X)} := 1 + \sum_{n = 1}^\infty \frac{1}{n!} \int_X \cdots \int_X \det \left[ K(x_i, x_j)\right]_{i,j = 1}^n \prod_{i = 1}^nd\mu(x_i),
\end{equation}
provided that the above integrals and corresponding sum are absolutely convergent. In this paper we will exclusively work in the case when $X$ is a piecewise smooth contour in $\mathbb{C}$ and $\mu(dz) = \frac{dz}{2\pi \i}$ with $dz$ denoting the usual complex integration. We refer to \cite[Section 2.1]{BCDA} for some basic background on Fredholm determinants.

We begin by introducing some useful notation, see also Figure \ref{S6_1}.
\begin{figure}[h]
\scalebox{0.5}{\includegraphics{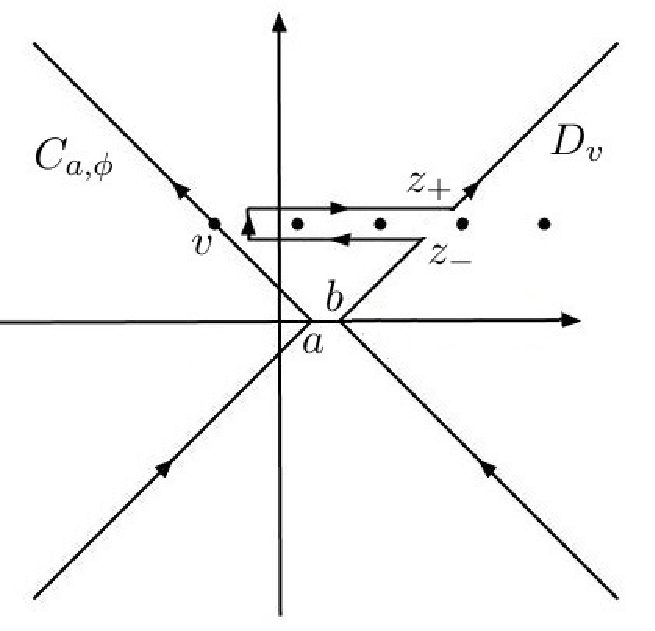}}
\captionsetup{width=\linewidth}
 \caption{The contour $C_{a, \phi}$ for $\phi = 3 \pi/ 4$ and $D_v(b,\pi/4, d)$ with $a, b \in \mathbb{R}$ and $b > a$. The black dots denote the points $v, v+1, v+2, \dots$. }
\label{S6_1}
\end{figure}

\begin{definition}\label{ContV}
For $a \in \mathbb{C}$ and $\phi \in (0, \pi)$ we define the contour $C_{a,\phi}$ to be the union of $\{a + ye^{-\i\phi} \}_{y \in \mathbb{R}^+}$ and $\{a + ye^{\i \phi} \}_{y \in \mathbb{R}^+}$ oriented to have increasing imaginary part.
\end{definition}

\begin{definition}\label{ContE}
Suppose that $a \in \mathbb{C}$, $\phi \in (0, \pi)$, $d > 0$ and $v \in \mathbb{C}$. From this data we construct a contour that consists of two parts. The first, called $D_v^1$, is $C_{a,\phi} \setminus \{z: \Im(z) \in [\Im(v) - d, \Im (v) + d] \}.$ Let $z_-, z_+$ be the points on $C_{a,\phi} $ that have imaginary parts $\Im(v) - d$ and $\Im(v) + d$ respectively. The second part of the contour, called $D_v^2$, consists of straight oriented segments that connect $z_-$ to $v + 2d- \i d$ to $v + 2d + \i d$ to $z_+$. We will denote the resulting contour by $D_{v}(a,\phi, d)$ or just $D_v$ when the other parameters are clear from the context.
\end{definition}

The first major result we require from \cite{BCDA} is the following Laplace transform formula for the partition function $Z(1,1; M, N)$.
\begin{proposition}\cite[Theorem 2.12]{BCDA} \label{LGPT1}
Fix integers $N \geq 9, M \geq 1$ and $\theta > 0$. Let $\theta > b > a > 0 $, $d \in \left(0 , \min \left( 1/4, (b-a)/4\right) \right)$ and $u \in \mathbb{C}$ with $\Re(u)> 0$. Then
\begin{equation}\label{LGPTe1}
\mathbb{E} \left[ e^{-u Z(1,1; M, N)} \right] = \det \left( I + K_u \right)_{L^2(C_{a,3\pi/4})},
\end{equation}
where the operator $K_u$ is defined in terms of its integral kernel
\begin{equation}\label{kernelMain}
K_u(v,v') = \frac{1}{2\pi \i} \int_{D_v}  \frac{\pi }{\sin(\pi (v-w))}  \left( \frac{\Gamma(v)}{\Gamma(w)} \right)^N  \left( \frac{\Gamma(\theta - w)}{\Gamma(\theta - v)} \right)^M \frac{u^{w-v} dw}{w - v'},
\end{equation}
where $C_{a,3\pi/4}$ and $D_v = D_v(b, \pi/4, d)$ are as in Definitions \ref{ContV} and \ref{ContE} .
\end{proposition}
\begin{remark} Part of the statement of Proposition \ref{LGPT1} is that $\det \left( I + K_u \right)_{L^2(C_{a,3\pi/4})}$ is well-defined in the sense that the integral in (\ref{kernelMain}) is well-defined for each $v,v' \in C_{a,3\pi/4}$, the resulting function is measurable and the series in (\ref{fredholmDefS6}) is absolutely convergent for $X = C_{a,3\pi/4}$ and $\mu(dz) = \frac{dz}{2\pi \i}$ with $dz$ denoting the usual complex integration.
\end{remark}

We next summarize a certain scaling of $M, N$ and introduce several quantities of interest, which will appear in our analysis.
\begin{definition}\label{DefScaleS6}
Recall that $\Psi(x)$ denotes the digamma function, see (\ref{digammaS1}). Suppose that $M, N \geq 1$ and $\theta > 0$ are given. We let $z_{c}(M,N)$ denote the maximizer of
$$W_{M,N}(x) := N \Psi(x) + M \Psi(\theta - x)$$
on the interval $(0, \theta)$.
 Notice that the above expression converges to $- \infty$ as $x \rightarrow 0+$ or $x \rightarrow \theta-$ and also the function $W_{M,N}(x)$ is strictly concave, hence the maximum exists and is unique. We let $W_{M,N}$ denote $W_{M,N}(z_c)$. It is also worth noting that $z_c=g_{\theta}^{-1}(N/M)$.

For any $\alpha > 0$ we define
$$\tilde\sigma_\alpha := \left( \sum_{n = 0}^\infty \frac{\alpha}{(n+z_c)^3} +  \sum_{n = 0}^\infty \frac{1}{(n+\theta - z_c)^3} \right)^{1/3}.$$

The way we scale the parameters is as follows. We will let $M,N \rightarrow \infty$ while $M\geq N$ and $\alpha = \alpha(M,N) = N/M \geq \delta$, where $\delta > 0$ is fixed.
For a given $x \in \mathbb{R}$ we set
$$u = u(x,M,N):= e^{W_{M,N} - M^{1/3} \tilde\sigma_{\alpha} x}.$$
Notice that if $\alpha \in [\delta, 1]$ and $\theta > 0$ is fixed then $\tilde\sigma_\alpha$ is positive and bounded away from $0$ and $\infty$, see e.g. \cite[(5.4)]{BCDA}. We finally let $K_u$ be as in Proposition \ref{LGPT1} for the parameter $u$ as above, $a = z_c $ and $b,d$ arbitrarily chosen so as to satisfy the conditions of that proposition.
\end{definition}

The main technical result of this section is contained in the following strong comparison of the Fredholm determinant in Proposition \ref{LGPT1} and $F_{\rm GUE}$.
\begin{proposition} \label{S6Main} Assume the same notation and scaling as in Definition \ref{DefScaleS6}. Then for any $\epsilon_0 \in (0,1/3)$ we have
\begin{equation}\label{S6MainEq}
\limsup_{M \rightarrow \infty} \sup_{x \in [0, \infty)} M^{\epsilon_0} \left| \det \left( I + K_u \right)_{L^2(C_{z_c,3\pi/4})} - F_{\rm GUE}(x) \right| = 0.
\end{equation}
\end{proposition}

In the remainder of this section we use Proposition \ref{S6Main} to prove Proposition \ref{S2LBProp}.

\begin{proof}[Proof of Proposition \ref{S2LBProp}] Assume the same notation as in Proposition \ref{LGPT1}, where $a,b,d, u$ are as in Definition \ref{DefScaleS6}. It follows from Proposition \ref{S6Main} that there exists $M_0$ sufficiently large so that for $M \geq M_0$ and $M \geq N \geq \delta M$ and $ y \in [0, \infty)$ we have
\begin{equation*}
\left| \det \left( I + K_u \right)_{L^2(C_{z_c,3\pi/4})} - F_{\rm GUE}(y) \right| \leq M^{-\epsilon_0},
\end{equation*}
where $u =  e^{W_{M,N} - M^{1/3} \tilde\sigma_{\alpha} y }.$ In particular, if for any $x \in [1, \infty)$ we set $x^{\pm} = x \pm M^{-1/3} (\log M)^2$, and $u^{\pm} = e^{W_{M,N} - M^{1/3} \tilde\sigma_{\alpha} x^{\pm}}$ we see that by possibly enlarging $M_0$ (so that $x^- \geq 0$) we obtain
\begin{equation}\label{SRE1}
\left| \det \left( I + K_{u^{\pm}} \right)_{L^2(C_{z_c,3\pi/4})} - F_{\rm GUE}(x^{\pm}) \right| \leq M^{-\epsilon_0}.
\end{equation}

 If we set $f_M(z) = \exp(-e^{\tilde\sigma_{\alpha}  M^{1/3}z}),$ then by Proposition \ref{LGPT1} we know that for all large enough $M$
\begin{equation}\label{SRE2}
\mathbb{E} \left[f_M(\mathcal{F}(M,N) - x^{\pm} )\right] = \mathbb{E} \left[ e^{-u^{\pm} \cdot Z(1,1; M , N)} \right] = \det \left( I + K_{u^{\pm}} \right)_{L^2(C_{z_c,3\pi/4})}.
\end{equation}

We note that by (\ref{SRE1}) and (\ref{SRE2}) we have for all large $M$
\begin{equation*}
\begin{split}
&\mathbb{P} \left( \mathcal{F}(M,N) \leq x  \right)   = \mathbb{E} \left[ {\bf 1} \{ \mathcal{F}(M,N) \leq x \} \right] \geq \mathbb{E} \left[ e^{-e^{\tilde\sigma_{\alpha}  M^{1/3}(\mathcal{F}(M,N) - x^-)}} \cdot {\bf 1} \{ \mathcal{F}(M,N) \leq x \} \right] \geq \\
&  \mathbb{E} \left[ e^{-e^{\tilde\sigma_{\alpha}  M^{1/3}(\mathcal{F}(M,N) - x^- )}} \right] -  e^{-e^{\tilde\sigma_{\alpha}  (\log M)^2}} \geq \det \left( I + K_{u^{-}} \right)_{L^2(C_{z_c,3\pi/4})} - M^{-\epsilon_0}  \geq F_{\rm GUE}(x^{-}) - 2M^{-\epsilon_0}.
\end{split}
\end{equation*}
By the same equations we also have
\begin{equation*}
\begin{split}
&\mathbb{P} \left( \mathcal{F}(M,N) > x  \right)   = \mathbb{E} \left[ {\bf 1} \{ \mathcal{F}(M,N) > x \} \right] \geq \mathbb{E} \left[ \left(1 - e^{-e^{\tilde\sigma_{\alpha}  M^{1/3}(\mathcal{F}(M,N) - x^+)}} \right) \cdot {\bf 1} \{ \mathcal{F}(M,N) > x \} \right] \geq \\
&  1 -  \det \left( I + K_{u^{+}} \right)_{L^2(C_{z_c,3\pi/4})} - \left(1 - e^{-e^{-\tilde\sigma_{\alpha} (\log M)^2}}  \right)  \geq 1 - F_{\rm GUE}(x^{+}) - 2M^{-\epsilon_0}.
\end{split}
\end{equation*}
Combining the last two inequalities we see that
$$ F_{\rm GUE}(x^{-}) - 2M^{-\epsilon_0} \leq \mathbb{P} \left( \mathcal{F}(M,N) \leq x  \right) \leq F_{\rm GUE}(x^{+}) + 2M^{-\epsilon_0},$$
which implies (\ref{S2LowerBound}) once we utilize the fact that $F_{\rm GUE}(x^{\pm}) - F_{\rm GUE}(x) = O(M^{-1/3} (\log M)^2).$ The latter equality follows from the boundedness of $F'_{\rm GUE}(x)$ on $[-1, \infty)$, which for example can be deduced by the log-concavity of $F'_{\rm GUE}(x)$ on $[0, \infty)$, see \cite[Theorem 5.1]{BLS17}.
\end{proof}

%
\subsection{Preliminary results} \label{Section6.2} In this section we summarize some results from \cite{BCDA} and prove Lemma \ref{TrKer}, which will be used in the proof of Proposition \ref{S6Main} in the next section. We begin with some useful notation.

\begin{definition}\label{DefDM}
Define
\begin{equation}\label{GfunS3}
G_{M,N}(z) := N \log \Gamma (z) - M \log \Gamma (\theta - z) - W_{M,N}   z - C_{M,N},
\end{equation}
where $C_{M,N} = N \log \Gamma (z_c) - M \log \Gamma (\theta - z_c) - W_{M,N}  z_c $, and $z_c, W_{M,N}$ are as in Definition \ref{DefScaleS6}.
We also define
$$G_\alpha (z) = M^{-1} G_{M,N}(z) =  \alpha \log \Gamma (z) -  \log \Gamma (\theta - z) - M^{-1} W_{M,N}   z - M^{-1}C_{M,N},$$
where we recall that $\alpha = N/M \in [\delta, 1]$ as in Definition \ref{DefScaleS6}. We define the contour $\mathcal{D}$ to be the contour that consists of the vertical segment connecting $1- \i $ and $1 + \i $ and for $|\Im(z)| \geq 1$ it agrees with $C_{0, \pi/4}$ from Definition \ref{ContV}, see Figure \ref{S6_2}. We also define the contours $D_M$ through $D_M = z_c + \tilde\sigma_\alpha^{-1} M^{-1/3} \cdot \mathcal{D}$. The contours $\mathcal{D}$ and $D_M$ are oriented to have increasing imaginary part.
\end{definition}
\begin{figure}[h]
\scalebox{0.5}{\includegraphics{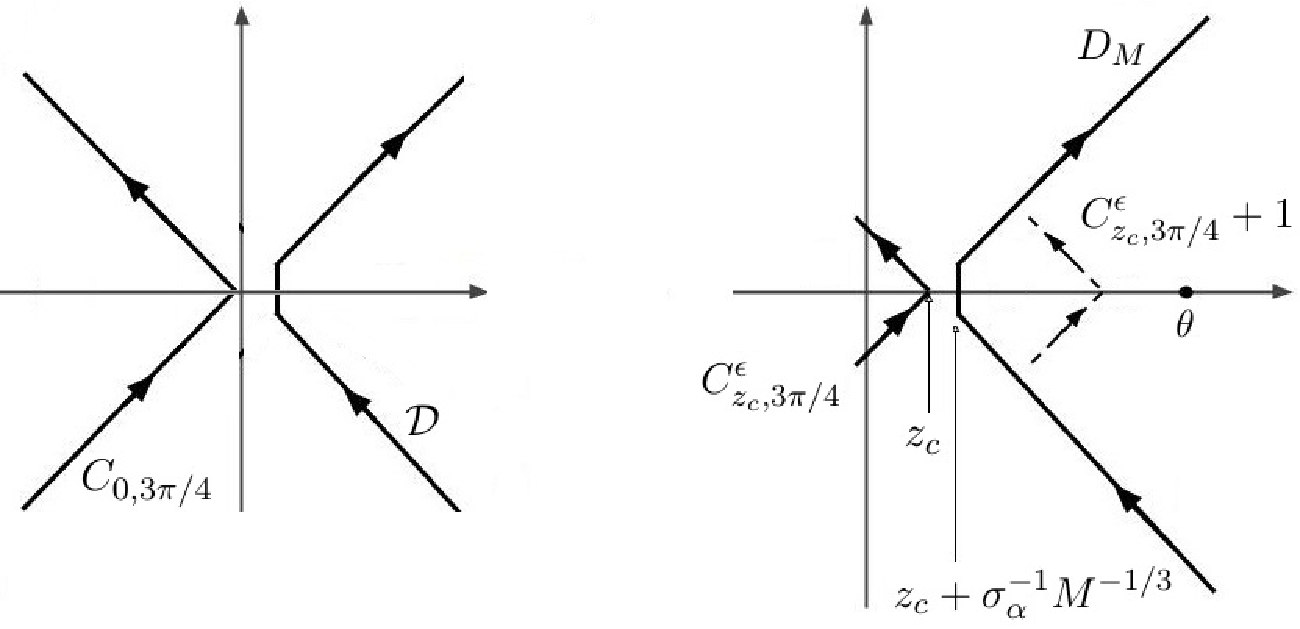}}
\captionsetup{width=\linewidth}
 \caption{The contours $\mathcal{D}$ and $D_M$. }
\label{S6_2}
\end{figure}

We next state several results from \cite{BCDA}.
\begin{lemma}\label{analGS3} \cite[Lemma 3.5]{BCDA} Fix $\theta > 0$, $\delta \in (0,1)$ and assume that $M \geq N \geq 1$, $N/M \in [\delta, 1]$.
There exist constants $C > 0$ and $r > 0$ depending on $\theta$ and $\delta$ such that $G_{\alpha}$ is analytic in the disc $|z - z_c| < r$ and the following hold for $|z - z_c| \leq r$:
\begin{equation}\label{powerGS3}
\begin{split}
&\left|G_\alpha(z)+ (z-z_c)^3\tilde\sigma^3_\alpha/3 \right| \leq C |z-z_c|^4;\\
&\Re[G_\alpha(z) ] \geq (\sqrt{2}/2)^3 |z-z_c|^3\tilde\sigma^3_\alpha/6  \mbox{ when $z \in C_{z_c, \phi}$ with $\phi =\pi/4$; }\\
&\Re[G_\alpha(z) ] \leq -  (\sqrt{2}/2)^3 |z-z_c|^3 \tilde\sigma^3_\alpha/6  \mbox{ when $z \in C_{z_c, \phi}$ with $\phi = 3\pi/4 $ }.
\end{split}
\end{equation}
In the above equations $z_c, \tilde\sigma_\alpha$ are as in Definition \ref{DefScaleS6} and $C_{z_c, \phi}$ is as in Definition \ref{ContV}.
\end{lemma}

\begin{lemma}\cite[Lemma 3.10]{BCDA} \label{SumBound} For any $\delta,\Delta > 0$ there exists $M_0 > 0$ such that if $M \geq M_0$
\begin{equation}\label{SumBoundE}
\sum_{n = 1}^\infty   \frac{n^{n/2} \Delta^nM^{n/3}}{n!} \leq e^{\delta M}.
\end{equation}
\end{lemma}

\begin{lemma} \label{TrKer} Suppose that $g(v,v',w)$ is a bounded complex-valued measurable function on $C_{0, \phi} \times C_{0, \phi} \times \mathcal{D}$, where $\phi \in [\pi/2, \pi)$, $C_{0, \phi}$ is as in Definition \ref{ContV} and $\mathcal{D}$ is as in Definition \ref{DefDM}. Assume that there are constants $A,B> 0$ such that for some $a, b \geq A$ we have
\begin{equation}\label{S6GBound}
|g(v,v',w)| \leq B \cdot \exp \left( - a |v|^3 - b |w|^3 \right).
\end{equation}
Then the kernel
$K(v,v') = \frac{1}{2\pi \i } \int_{\mathcal{D}} g(v,v',w) dw$
is well-defined on $C_{0, \phi} \times C_{0, \phi}$ and also the Fredholm determinant $\det(I + K)_{L^2(C_{0, \phi})}$ is well-defined. Furthermore, there is a constant $N_0 \geq 1$ (depending on $A,B$) such that for all $N \geq N_0$ we have
\begin{equation}\label{S6TruncBound}
\left| \det(I + K)_{L^2(C_{0, \phi})} - 1- \sum_{n = 1}^N \frac{1}{n!} \int_{C_{0, \phi}} \hspace{-3mm} \cdots \int_{C_{0, \phi}} \hspace{-3mm}  \det \left[ K^N(v_i, v_j) \right]_{i,j = 1}^n \prod_{i = 1}^n\frac{dv_i}{2\pi \i}  \right| \leq e^{-(N/3) \log N},
\end{equation}
where $ K^N(v, v') =  \frac{1}{2\pi \i } \int_{\mathcal{D}} g^N(v,v',w) dw$ with $g^N(v,v',w)  = {\bf 1}\{ \max( |{v}|, |{w}|,|{v}'|)  \leq N \} \cdot g(v,v',w)$.
\end{lemma}
\begin{proof} The fact that all the Fredholm determinants in the statement of the lemma are well-defined in terms of absolutely convergent series as in (\ref{fredholmDefS6}) follows from Lemmas 2.3 and 2.4 in \cite{BCDA}, whose proofs can be found in \cite[Section 2.5]{ED}. In the remainder we prove (\ref{S6TruncBound}).

In view of (\ref{S6GBound}) we can find a constant $D_1 > 0$ (depending on $A,B$) such that for all $v,v' \in C_{0, \phi}$
$$ \int_{\mathcal{D}} |g(v,v',w)| |dw| \leq D_1 \cdot e^{ - A |v|^3} \mbox{ and so } |K(v,v')| \leq D_1 e^{ - A |v|^3},$$
where $|dw|$ denotes integration with respect to arc length. This and Hadamard's inequality imply
\begin{equation}\label{TRE1}
\begin{split}
&\left| \sum_{n = N+1}^\infty \frac{1}{n!} \int_{C_{0, \phi}} \cdots \int_{C_{0, \phi}} \det \left[ K(v_i, v_j) \right]_{i,j = 1}^n \prod_{i = 1}^n \frac{dv_i}{2\pi \i} \right| \leq  \\
&\sum_{n = N+1}^\infty \frac{n^{n/2}}{n!} \int_{C_{0, \phi}} \cdots \int_{C_{0, \phi}} D_1^n \cdot \prod_{i = 1}^n e^{ - A |v_i|^3}|dv_1| \cdots |dv_n| \leq \sum_{n = N+1}^\infty \frac{(D_1 D_2)^n n^{n/2}}{n!} ,
\end{split}
\end{equation}
where $D_2 = \int_{C_{0, \phi}} e^{ - A |v|^3}|dv|.$ From \cite[(1) and (2)]{Rob} we have for each $n \geq 1$ that
\begin{equation}\label{Stirling}
 n! = \sqrt{2\pi} n^{n+1/2} e^{-n} e^{r_n} \mbox{, where } (12n+1)^{-1} < r_n < (12n)^{-1}.
\end{equation}
Equation (\ref{Stirling}) implies that if we set $a_n = \frac{(D_1 D_2)^n n^{n/2}}{n!}$ we have that $\frac{a_{n+1}}{a_n} \sim \frac{D_1D_2 \sqrt{e}}{\sqrt{n}}$ as $n \rightarrow \infty$ and so for all large $N$ we have
\begin{equation}\label{TRE2}
 \sum_{n = N+1}^\infty \frac{(D_1 D_2)^n n^{n/2}}{n!} \leq  \frac{(D_1 D_2)^N N^{N/2}}{N!} \leq \frac{e^{-(1/3)N \log N}}{2},
\end{equation}
where in the last inequality we used (\ref{Stirling}) again.

For $\sigma \in S_n$ we set $\chi^{\sigma}_n =  \prod_{i =1}^n {\bf 1}\{ \max( |{v}_i|, |{w}_i|,|{v}_{\sigma(i)}|)  \leq N \} $ , $C_1 =  \int_{C_{0, \phi}} e^{ -(A/2)|v|^3} |dv|$ and $C_2 =  \int_{\mathcal{D}} e^{ -(A/2)|w|^3} |dw|.$ By expanding determinants we get

\begin{equation*}
\begin{split}
&\left| \sum_{n = 1}^N \frac{1}{n!} \int_{C_{0, \phi}} \hspace{-3mm}\cdots \int_{C_{0, \phi}} \hspace{-3mm} \det \left[ K(v_i, v_j) \right]_{i,j = 1}^n \prod_{i = 1}^n \frac{dv_i}{2\pi \i} -  \sum_{n = 1}^N \frac{1}{n!} \int_{C_{0, \phi}} \hspace{-3mm} \cdots \int_{C_{0, \phi}} \hspace{-3mm}  \det \left[ K^N(v_i, v_j) \right]_{i,j = 1}^n \prod_{i = 1}^n \frac{dv_i}{2\pi \i}  \right| \leq  \\
& \sum_{n = 1}^N  \sum_{\sigma \in S_n} \frac{1}{n!} \int_{C_{0, \phi}} \cdots \int_{C_{0, \phi}}  \int_{\mathcal{D}} \cdots \int_{\mathcal{D}}(1 -\chi^{\sigma}_n) \cdot  \prod_{i = 1}^n |g(v_i,v_{\sigma(i)},w_i)| \prod_{i = 1}^n |dw_i|  \prod_{i = 1}^n |dv_i| \leq \\
&e^{-(A/2)N^3} \cdot \sum_{n = 1}^N  \int_{C_{0, \phi}} \cdots \int_{C_{0, \phi}}  \int_{\mathcal{D}} \cdots \int_{\mathcal{D}} B^{n} \cdot \prod_{i = 1}^n \exp \left( -(A/2) |v_i|^3 - (A/2)|w_i|^3 \right)  \prod_{i = 1}^n |dw_i|  \prod_{i = 1}^n |dv_i| \leq \\
& e^{-(A/2)N^3} \sum_{n = 1}^N   B^{n} C_1^n C_2^n \leq (1/2)e^{-(1/3)N \log N}.
\end{split}
\end{equation*}
The second inequality above requires some explanation that we give momentarily and the final inequality holds for all large enough $N$. Combining the last inequality with (\ref{TRE1}) and (\ref{TRE2}) we obtain (\ref{S6TruncBound}).

For the second inequality above we note that in order for $(1 -\chi^{\sigma}_n)$ to be non-zero (and hence equal to 1) we must have some $i$ for which either $|v_i|$ or $|w_i|$ exceeds $N$. For that $i$, we can use (\ref{S6GBound}) to show that $|g(v_i,v_{\sigma(i)},w_i)| \leq B e^{-a|v_i|^3-b|w_i|^3}\leq B e^{-(A/2)N^3} e^{-(A/2) |v_i|^3 - (A/2)|w_i|^3}$. This last inequality follows since
$-a|v_i|^3-b|w_i|^3 \leq -A|v_i|^3-A|w_i|^3\leq -(A/2)|v_i|^3-(A/2)|w_i|^3 - (A/2)N^3$, as follows from the fact that $a,b\geq A$ and at least one of $|v_i|$ or $|w_i|$ exceeds $N$.
\end{proof}

%
\subsection{Proof of Proposition \ref{S6Main}} \label{Section6.3} In this section we present the proof of Proposition \ref{S6Main}. For clarity, we split the proof into three steps. In the first step we show that we can truncate and deform the contours in the definition of $\det \left( I + K_u \right)_{L^2(C_{z_c,3\pi/4})}$ to more suitable ones without affecting the value of this Fredholm determinant significantly. In the second step we claim a certain finite sum estimate, see (\ref{BB1}), and deduce the statement of the proposition by combining this statement and Lemma \ref{TrKer}. Equation  (\ref{BB1}) is proved in the third and final step.

\smallskip
{\bf \raggedleft Step 1.} Let $r$ (depending on $\theta$ and $\delta$) be as in Lemma \ref{analGS3} and fix $\epsilon > 0$ such that $\epsilon < \min (r/2, 1/4)$. Let $C^{\epsilon}_{z_c, 3\pi/4}$ denote the portion of $C_{z_c, 3\pi/4}$ inside $B_{\epsilon}(z_c)$ -- the disc of radius $\epsilon$, centered at $z_c$. Repeating verbatim the argument in Step 2 of the proof of \cite[Theorem 1.7]{BCDA} (see (4.4) in that paper) we get that there is a constant $c_1 > 0$ (depending on $\theta, \delta, \epsilon$) such that for all large $M$
\begin{equation}\label{ST1}
 \left| \det (I + K_u)_{L^2(C_{z_c, 3\pi/4})} - \det (I + K_{u})_{L^2(C^{\epsilon}_{z_c, 3\pi/4})} \right|\leq e^{-c_1 M}.
\end{equation}

By Cauchy's theorem we may deform the $D_v(b,d, \pi/4)$ contour in the definition of $K_u(v,v')$ to $D_M$ from Definition \ref{DefDM} without affecting the value of the kernel as long as $M$ is sufficiently large. Indeed, notice that by our choice of $\epsilon \leq 1/4$ we have that $C^{\epsilon}_{z_c, 3\pi/4} + 1$ lies to the right of $D_M$ and so we do not cross any poles while deforming $D_v(b,d, \pi/4)$ to $D_M$, see Figure \ref{S6_2}. The decay estimates necessary to deform the contour near infinity come from \cite[Proposition 2.15]{BCDA} applied to $K = [0,\theta]$, $T =0$ and $\alpha_i = \theta$, $a_j = 0$ and $u$ as in Definition \ref{DefScaleS6}.

We write $D_M = D_M^{\epsilon, 0} \cup D_M^{\epsilon,1}$, where $D_M^{\epsilon, 0}$ is the portion of $D_M$ inside $B_\epsilon(z_c)$  and $D_M^{\epsilon, 1}$ is the portion outside. Define for $\beta \in \{0,1\}$ and $v, v' \in C^{\epsilon}_{z_c, 3\pi/4} $ the kernel
$$K^{\beta}_{u}(v,v') = \frac{1}{2\pi \i}\int_{D_M^{\epsilon, \beta}}\frac{F(v, w)}{w - v' }dw  \mbox{, where } F(v,w) = \frac{\pi}{\sin(\pi (v-w))} \cdot e^{G_{M,N}(v) - G_{M,N}(w)} \cdot e^{M^{1/3} (w- v) \tilde\sigma_{\alpha} x}.$$

From \cite[equation (4.11)]{BCDA} we have that there is a constant $C_2, c_{2} > 0$ (depending on $\theta, \delta, \epsilon$) such that for all large enough $M$
\begin{equation*}
\begin{split}
\left| \int_{C^{\epsilon}_{z_c, 3\pi/4}}  \cdots  \int_{C^{\epsilon}_{z_c, 3\pi/4}}  \det \left[ K^{\beta_i}_u(v_i, v_j) \right]_{i,j = 1}^n  \prod_{i = 1}^n dv_i\right| \leq C_2^{n} n^{n/2} M^{n/3} e^{-2 c_2 M\sum_{i = 1}^n \beta_i}.
\end{split}
\end{equation*}
Using the last inequality, the fact that $K_{u}(v,v') = K^{0}_{u}(v,v') + K^{1}_{u}(v,v') $, the definition of a Fredholm determinant (\ref{fredholmDefS6}), and the linearity of the determinant function we obtain
\begin{equation}\label{ST2}
 \left| \det (I + K_u)_{L^2(C^{\epsilon}_{z_c, 3\pi/4})} -  \det (I + K^0_u)_{L^2(C^{\epsilon}_{z_c, 3\pi/4})} \right| \leq e^{- 2c_2 M} \sum_{n = 1}^\infty \frac{2^n C_2^{n} n^{n/2} M^{n/3}}{n!} \leq e^{-c_2 M},
\end{equation}
where the last inequality used Lemma \ref{SumBound} with $\delta = c_2$ and $\Delta = 2 C_2$ and holds for large $M$.

\smallskip
Equations (\ref{ST1}) and (\ref{ST2}) show that $ \det (I + K_u)_{L^2(C_{z_c, 3\pi/4})} $ is very close to $ \det (I + K^0_u)_{L^2(C^{\epsilon}_{z_c, 3\pi/4})}$ and in the remainder of this step we rewrite $\det (I + K^0_u)_{L^2(C^{\epsilon}_{z_c, 3\pi/4})}$ in a more suitable form for the application of Lemma \ref{TrKer} in the next step.

Applying the change of variables $v_i = \tilde\sigma_{\alpha}^{-1}M^{-1/3} \tilde{v}_i + z_c$ and $w_i = \tilde\sigma_{\alpha}^{-1}M^{-1/3} \tilde{w}_i + z_c$ we can rewrite
\begin{equation}\label{ST3}
\det (I + K^0_u)_{L^2(C^{\epsilon}_{z_c, 3\pi/4})} = \det (I + K_{x,M})_{L^2(C_{0, 3\pi/4})},\quad\textrm{where}\quad K_{x,M}(\tilde{v}, \tilde{v}') =  \frac{1}{2\pi \i } \int_{\mathcal{D}} g_{x,M}(\tilde{v},\tilde{v}',\tilde{w}) d\tilde{w},
\end{equation}
$\mathcal{D}$ is as in Definition \ref{DefDM} and
$$g_{x,M}(\tilde{v},\tilde{v}',\tilde{w})  =  \frac{ {\bf 1}\{ \max( |\tilde{v}|, |\tilde{w}|,|\tilde{v}'|)  \leq \epsilon \tilde\sigma_{\alpha} M^{1/3} \}  \cdot \pi M^{-1/3} \tilde\sigma_{\alpha}^{-1}}{ \sin( \pi M^{-1/3} \tilde\sigma_{\alpha}^{-1} (\tilde{v} - \tilde{w}) )(\tilde{w} - \tilde{v}')} \frac{e^{M G_{\alpha}(\tilde{v} \tilde\sigma_{\alpha}^{-1} M^{-1/3} +z_c) + x \tilde{v}}}{e^{M G_{\alpha}(\tilde{w} \tilde\sigma_{\alpha}^{-1} M^{-1/3} +z_c) + x \tilde{w}}}.$$
In view of Lemma \ref{analGS3} we have that there exists a constant $B > 0$ such that
\begin{equation}\label{ST4}
\left|g_{x,M}(\tilde{v},\tilde{v}',\tilde{w})  \right| \leq B \cdot \exp \left( - (\sqrt{2}/2)^3 |\tilde{v}|^3/6 - (\sqrt{2}/2)^3 |\tilde{w}|^3/6  \right),
\end{equation}
where we used the fact that $|e^{x\tilde{v} - x \tilde{w}}| \leq 1$ (as $x \geq 0$) and $\frac{\pi M^{-1/3} \tilde\sigma_{\alpha}^{-1}}{ \sin( \pi M^{-1/3} \tilde\sigma_{\alpha}^{-1} (\tilde{v} - \tilde{w}) )}$ is bounded, cf. \cite[Lemma 3.9]{BCDA}. Finally, we combine (\ref{ST1}), (\ref{ST2}) and (\ref{ST3}) to get
 \begin{equation}\label{ST5}
\left|\det (I + K_u)_{L^2(C_{z_c, 3\pi/4})}-   \det (I + K_{x,M})_{L^2(C_{0, 3\pi/4})} \right| \leq e^{-c_1 M} + e^{-c_2M}.
\end{equation}

{\bf \raggedleft Step 2.}  We first note that by repeating verbatim the proof of \cite[Lemma C.1]{BCF} we have that
\begin{equation}\label{ST6}
F_{\rm GUE}(x) = \det (I + \tilde{K}_x)_{L^2(C_{0, 3\pi/4})},\quad \textrm{where}\quad \tilde{K}_x(\tilde{v}, \tilde{v'}) = \frac{1}{2\pi \i} \int_{\mathcal{D}} \frac{e^{-\tilde{v}^3/3 + \tilde{w}^3/3 - x \tilde{w} + x \tilde{v}}}{(\tilde{v} - \tilde{w})(\tilde{w} - \tilde{v}')} d\tilde{w}.
\end{equation}

Set $M_1 = \lceil \frac{\log M}{\log \log M} \rceil$ and  $\chi^{\sigma}_n =  \prod_{i = 1}^n {\bf 1}\{ \max( |\tilde{v}_i|, |\tilde{w}_i|,|\tilde{v}_{\sigma(i)}|)  \leq M_1 \}$ for $\sigma \in S_n$. We claim that
\begin{equation}\label{BB1}
\begin{split}
& 0 = \limsup_{M \rightarrow \infty} \sup_{x \geq 0} M^{\epsilon_0} \cdot  \sum_{n = 1}^{M_1} \sum_{\sigma \in S_n} \frac{1}{n!}\int_{C_{0, 3\pi/4}} \cdots \int_{C_{0, 3\pi/4}} \int_{\mathcal{D}} \cdots \int_{\mathcal{D}} \\
& \chi^{\sigma}_n \cdot  \left|  \prod_{i =1}^n g_{x,M}(\tilde{v}_i,\tilde{v}_{\sigma(i)},\tilde{w}_i) -  \prod_{i =1}^n \frac{e^{ -\tilde{v}_i^3/3 + \tilde{w}_i^3/3 - x \tilde{w}_i + x \tilde{v}_i}}{(\tilde{v}_i - \tilde{w}_i)(\tilde{w}_i - \tilde{v}_{\sigma(i)})} \right| \prod_{i = 1}^n |d\tilde{v}_i|\prod_{i = 1}^n |d\tilde{w}_i|.
\end{split}
\end{equation}
We will prove (\ref{BB1}) in the next step. Here we assume its validity and conclude the proof of (\ref{S6MainEq}).

\smallskip
In view of (\ref{ST4}) and Lemma \ref{TrKer} we have for all large $M$ that
\begin{equation}\label{ST7}
\begin{split}
&\Bigg{|}  \det (I + K_{x,M})_{L^2(C_{0, 3\pi/4})} - 1 - \sum_{n = 1}^{M_1} \sum_{\sigma \in S_n} \frac{1}{n!} \int_{C_{0, 3\pi/4}} \hspace{-4mm}\cdots \int_{C_{0, 3\pi/4}} \int_{\mathcal{D}} \cdots \int_{\mathcal{D}}   \\
&  \chi^{\sigma}_n \cdot \prod_{i = 1}^n g_{x,M}(\tilde{v}_i,\tilde{v}_{\sigma(i)},\tilde{w}_i)  \prod_{i = 1}^n \frac{d\tilde{w}_i}{2\pi \i} \prod_{i = 1}^n \frac{d\tilde{v}_i}{2\pi \i} \Bigg{|} \leq e^{-(M_1/3) \log M_1},
\end{split}
\end{equation}
where we expanded the determinants. From (\ref{ST6}) and Lemma \ref{TrKer} we have for all large $M$ that
\begin{equation}\label{ST8}
\begin{split}
&\Bigg{|} F_{\rm GUE}(x) - 1 - \sum_{n = 1}^{M_1}\sum_{\sigma \in S_n} \frac{1}{n!}  \int_{C_{0, 3\pi/4}} \hspace{-4mm} \cdots \int_{C_{0, 3\pi/4}} \int_{\mathcal{D}} \cdots \int_{\mathcal{D}}  \\
&\chi^{\sigma}_n \cdot \prod_{i = 1}^n   \frac{e^{ -\tilde{v}_i^3/3 + \tilde{w}_i^3/3 - x \tilde{w}_i + x \tilde{v}_i}}{(\tilde{v}_i - \tilde{w}_i)(\tilde{w}_i - \tilde{v}_{\sigma(i)})} \prod_{i = 1}^n \frac{d\tilde{w}_i}{2\pi \i} \prod_{i = 1}^n \frac{d\tilde{v}_i}{2\pi \i}  \Bigg{|} \leq e^{-(M_1/3) \log M_1}.
\end{split}
\end{equation}
Combining (\ref{ST7}), (\ref{ST8}) with (\ref{BB1}) and (\ref{ST5}) we conclude (\ref{S6MainEq}).

\smallskip
{\bf \raggedleft Step 3.} In this step we prove (\ref{BB1}). We first note that for all large enough $M$ (so that $M_1 \leq \epsilon \tilde\sigma_{\alpha} M^{1/3}$), each $n \in \{1, \dots, M_1\}$, $\sigma \in S_n$, $\tilde{v_i} \in C_{0, 3\pi/4}$ and $\tilde{w}_i \in \mathcal{D}$ for $i = 1, \dots, n$ we have
\begin{equation*}
\begin{split}
& \chi^{\sigma}_n \cdot  \left|  \prod_{i =1}^n g_{x,M}(\tilde{v}_i,\tilde{v}_{\sigma(i)},\tilde{w}_i) -  \prod_{i =1}^n \frac{e^{ -\tilde{v}_i^3/3 + \tilde{w}_i^3/3 - x \tilde{w}_i + x \tilde{v}_i}}{(\tilde{v}_i - \tilde{w}_i)(\tilde{w}_i - \tilde{v}_{\sigma(i)})} \right| \leq\chi^{\sigma}_n \cdot \prod_{i =1}^n \left|  \frac{e^{ -\tilde{v}_i^3/3 + \tilde{w}_i^3/3}}{(\tilde{v}_i - \tilde{w}_i)(\tilde{w}_i - \tilde{v}_{\sigma(i)})}  \right| \cdot  \\
&  \left|1 -  \prod_{i = 1}^n\frac{ (\tilde{v}_i - \tilde{w}_i)  \pi M^{-1/3} \tilde\sigma_{\alpha}^{-1}}{ \sin( \pi M^{-1/3} \tilde\sigma_{\alpha}^{-1} (\tilde{v}_i - \tilde{w}_i) )} \frac{e^{M G_{\alpha}(\tilde{v} \tilde\sigma_{\alpha}^{-1} M^{-1/3} +z_c) + \tilde{v}_i^3/3 }}{e^{M G_{\alpha}(\tilde{w} \tilde\sigma_{\alpha}^{-1} M^{-1/3} +z_c) + \tilde{w}_i^3/3}} \right|,
\end{split}
\end{equation*}
where we also used that $x \geq 0$. In addition, from (\ref{powerGS3}) we know that for some $D_1 > 0$ we have
$$\left| M G_{\alpha}( z \tilde\sigma_{\alpha}^{-1} M^{-1/3} +z_c) + z^3/3 \right| \leq D_1 M_1^4 M^{-1/3},$$
for all $z \in \mathbb{C}$ such that $|z| \leq M_1$. In addition, for some $D_2 > 0$ and all $\tilde{v} \in C_{0, 3\pi/4}$ and $\tilde{w} \in \mathcal{D}$
$$\left|1 - \frac{ (\tilde{v}- \tilde{w})  \pi M^{-1/3} \tilde\sigma_{\alpha}^{-1}}{ \sin( \pi M^{-1/3} \tilde\sigma_{\alpha}^{-1} (\tilde{v} - \tilde{w}) )} \right| \leq D_2 M_1 M^{-1/3}, \mbox{ provided that $\max( |\tilde{v}|, |\tilde{w}|) \leq M_1$}.$$
Combining the latter inequalities, the fact that $n \leq M_1$ and the fact that $M_1 = \lceil \frac{\log M}{\log \log M} \rceil$ we conclude that there is a constant $D_3 > 0$ such that for all large $M$
\begin{equation*}
\begin{split}
& \chi^{\sigma}_n \cdot  \left|  \prod_{i =1}^n g_{x,M}(\tilde{v}_i,\tilde{v}_{\sigma(i)},\tilde{w}_i) -  \prod_{i =1}^n \frac{e^{ -\tilde{v}_i^3/3 + \tilde{w}_i^3/3 - x \tilde{w}_i + x \tilde{v}_i}}{(\tilde{v}_i - \tilde{w}_i)(\tilde{w}_i - \tilde{v}_{\sigma(i)})} \right| \leq D_3 \cdot \chi^{\sigma}_n \cdot M_1^5 M^{-1/3} \cdot  \\
& \prod_{i =1}^n \left|  \frac{e^{ -\tilde{v}_i^3/3 + \tilde{w}_i^3/3}}{(\tilde{v}_i - \tilde{w}_i)(\tilde{w}_i - \tilde{v}_{\sigma(i)})}  \right| \leq D_3 \cdot \chi^{\sigma}_n \cdot M_1^5 M^{-1/3} \cdot \prod_{i =1}^n \left|  e^{ -\tilde{v}_i^3/3 + \tilde{w}_i^3/3}  \right|  ,
\end{split}
\end{equation*}
where in the last inequality we used that $|\tilde{v} - \tilde{w}| \geq 1$ for $\tilde{v} \in C_{0, 3\pi/4}$ and $\tilde{w} \in \mathcal{D}$. Combining the above estimates we conclude that the right side of (\ref{BB1}) is upper bounded by
$$\limsup_{M \rightarrow \infty}  M^{\epsilon_0} \sum_{n = 1}^{M_1} D_3 M_1^5 M^{-1/3} D_4^n,  $$
where $D_4 = \int_{C_{0, 3\pi/4}} \int_{\mathcal{D}} \left|  e^{ -\tilde{v}^3/3 + \tilde{w}^3/3}  \right| |d\tilde{w}| |d\tilde{w}|$. Since the above sum is at most $D_3 (D_4 + 1)^{M_1} M_1^5 M^{\epsilon_0-1/3}$ and $M_1 = \lceil \frac{\log M}{\log \log M} \rceil$ we see that the limit is $0$, which proves (\ref{BB1}).

\end{appendix}

\bibliographystyle{amsalpha}
\bibliography{PD}

\end{document}